\documentclass[11pt,twoside]{article}

\usepackage{mathrsfs,amsfonts,amsmath,amssymb}
\usepackage{latexsym}
\newtheorem{satz}{Theorem}
\newtheorem{proposition}[satz]{Proposition}
\newtheorem{theorem}[satz]{Theorem}
\newtheorem{lemma}[satz]{Lemma}
\newtheorem{definition}[satz]{Definition}
\newtheorem{corollary}[satz]{Corollary}
\newtheorem{remark}[satz]{Remark}
\newtheorem{example}[satz]{Example}

\def\eps{\varepsilon}
\def\_phi{\varphi}

\def\a{\alpha}
\def\d{\delta}
\def\la{\lambda}

\def\F{{\mathbb F}}
\def\L{\Lambda}

\def\o{\omega}
\def\ov{\overline}

\def\C{{\mathbb C}}
\def\R{{\mathbb R}}
\def\E{\mathsf {E}}
\def\T{{\mathbb T}}

\def\Z_N{{\mathbb Z}_N}
\def\Z{{\mathbb Z}}
\def\N{{\mathbb N}}
\def\U{{\mathcal U}}

\def\f{{\mathbb F}}
\def\Gr{{\mathbf G}}

\def\det{{\rm det\,}}

\def\G{\Gamma}
\def\FF{\widehat}
\def\c{\circ}
\def\D{\Delta}

\def\T{\mathsf {T}}
\def\I{\mathcal {I}}
\def\P{\mathcal {P}}

\def\Do{{\mathsf D}}
\def\No{{\mathsf N}}
\def\Q{\mathsf{Q}}

\def\SL{{\rm SL\,}}
\def\GL{{\rm GL\,}}
\def\U{{\rm U\,}}
\def\B{{\rm B\,}}

\author{Shkredov I.D.}
\title{ On asymptotic formulae in some sum--product questions
	\footnote{
		This work was supported by grant
		Russian Scientific Foundation RSF 14--11--00433.}
}
\date{}
\begin{document}
	\maketitle

\begin{center}
	Annotation.
\end{center}

{\it \small
	In 
	this 
	paper we obtain a series of asymptotic formulae in the sum--product phenomena 
	over 
	the prime field $\F_p$.
	In the proofs we use usual incidence theorems in  $\F_p$, as well as the growth result in $\SL_2 (\F_p)$ due to  Helfgott. 
	Here some of our applications: \\
	$\bullet~$ a new bound for the number of the  solutions to the equation $(a_1-a_2) (a_3-a_4) = (a'_1-a'_2) (a'_3-a'_4)$, $\,a_i, a'_i\in A$, $A$ is an arbitrary subset of $\F_p$, \\
	$\bullet~$ a new effective bound for multilinear  exponential sums of Bourgain, \\ 
	$\bullet~$ an asymptotic analogue of the Balog--Wooley decomposition theorem, \\
	$\bullet~$ growth of $p_1(b) + 1/(a+p_2 (b))$, where $a,b$ runs over two subsets of $\F_p$, $p_1,p_2 \in \F_p [x]$ are two non--constant polynomials, \\ 
	$\bullet~$ new bounds for some exponential sums with multiplicative and additive characters.  
}
\\

	\section{Introduction}
	\label{sec:introduction}

	Let $p$ be an odd prime number, and $\F_p$ be the finite field.
	Having two sets $A,B\subset \F_p$, define the  \textit{sumset}, the \textit{product set} and the \textit{quotient set} of $A$ and $B$ as 
	$$A+B:=\{a+b ~:~ a\in{A},\,b\in{B}\}\,,$$
	$$AB:=\{ab ~:~ a\in{A},\,b\in{B}\}\,,$$
	and 
	$$A/B:=\{a/b ~:~ a\in{A},\,b\in{B},\,b\neq0\}\,,$$
	correspondingly.
	Our paper is devoted to so--called the {\it sum--product phenomenon} 
	in 
	$\F_p$ which was developed in papers \cite{AMRS}--\cite{GK}, \cite{K_mult}---\cite{RSS},
	\cite{s_E_k}---\cite{SZ_inc}, \cite{TTT}, \cite{Vinh} and in many others.
	This is extensively 
	growing 
	area of mathematics with plenty of applications to Number Theory, Additive Combinatorics, Computer Science and Dynamical Systems. 
	It seems like at the moment there is "the second wave"\, of results and applications in this field see, e.g.,  \cite{BW}, \cite{MP}---\cite{PS}, \cite{RSS} 
	and this wave 
	is connected with a fundamental incidence result of Rudnev \cite{Rudnev_pp} (see a simple proof of his 
	theorem 
	in \cite{Zeeuw}
	and also the famous Guth--Katz  \cite{Guth_Katz} solution of the Erd\H{o}s distinct distance problem which contains the required technique for such incidence results), as well as with more applicable 
	{\it energy} versions of the sum--product phenomenon \cite{BW}, \cite{BG}, \cite{KS2}, \cite{MP}, \cite{MPR-NRS} and \cite{RSS}. 
	The sum--product phenomenon asserts that either the sumset or the product set of a set must be large up to some natural algebraic constrains.  
	One of 
	the strongest form of this principle is the Erd\H{o}s--Szemer\'{e}di  conjecture \cite{ES} 
	which says that for any sufficiently large set $A$ and an arbitrary $\epsilon>0$ one has
	$$\max{\{|A+A|,|AA|\}}\gg{|A|^{2-\epsilon}} \,.$$
	At the moment the best results in the direction  can be found in 
	\cite{soly}, 
	\cite{KS1}, \cite{KS2}, \cite{RSS} 
	and in \cite{AMRS}, \cite{RRS}  for  $\R$ and $\F_p$, respectively. 
	For example, let us recall the main results from \cite{AMRS}, \cite{RRS}.

	\begin{theorem}
		Let $A\subseteq \F_p$ be an arbitrary set and $|A|<p^{5/8}$. 
		Then 
		\begin{equation}\label{f:1/5_F_p}
		\max\{ |A+A|, |AA| \} \gg |A|^{1+1/5} \,.
		\end{equation}
		\label{t:1/5_F_p}
	\end{theorem}

	As one can  see the bound above works for small sets only and this is an usual 
	thing for 
	the results in this area. 
	On the other hand, the exact  behaviour of the maximum in (\ref{f:1/5_F_p}) and other sum--product quantities are known just for very large sets 
	having its sizes 
	comparable  to the characteristic $p$,
	see, e.g., \cite{G}, \cite{Garaev_survey}, \cite{TTT}, \cite{Vinh}. 
	Even in strong recent paper \cite{SZ_inc}  containing an optimal estimate for the number of point/lines incidences in the case of Cartesian products  we have just an upper bound for such incidences  but not an asymptotic. 
	The first 
	result in the sum--product theory which gives us an asymptotic formula for a sum--product quantity and which works  for sets of any sizes  
	was proved in \cite[Theorem 10]{MPR-NRS}.

	\begin{theorem}
		Let $A\subseteq \F_p$ be a set and $\Q(A)$ be the number of collinear quadruples in $A\times A$.
		Then
		\begin{equation}\label{f:Q_1}
		\Q(A) = \frac{|A|^8}{p^2} + O(|A|^5 \log |A|) \,.
		\end{equation}	
		Further for the number $\T(A)$ of  collinear triples in $A \times A$ one has
		\begin{equation}\label{f:Q_2}
		\T(A) = \frac{|A|^6}{p} + O(p^{1/2} |A|^{7/2}) \,.
		\end{equation}		
		\label{t:Q}
	\end{theorem}

	It is known that formula (\ref{f:Q_1}) is sharp up to logarithms but (\ref{f:Q_2}) is probably not, 
	see \cite{MPR-NRS}.

	One of the aims of our paper is to prove a series of new asymptotic formulae in the considered 
	area. 
	In the proofs we use usual incidence theorems in  $\F_p$, see \cite{Rudnev_pp}, \cite{SZ_inc} and other papers,  
	as well as the  growth result  in $\SL_2 (\F_p)$ due to Helfgott \cite{Harald}.
	So, our another aim  is to obtain some new applications (also, see recent papers \cite{Brendan_rich},  \cite{NG_S} where other applications were found) of classical graph (group) expansion phenomena,  see \cite{BG_SL}, \cite{Gill}, \cite{Gow_random}, \cite{Harald}, \cite{SX} and others.

	Our first asymptotic formula concerns to the quantity
	$$
	\T^{+}_k (A) := |\{ (a_1,\dots,a_k,a'_1,\dots,a'_k) \in A^{2k} ~:~ a_1 + \dots + a_k = a'_1 + \dots + a'_k \}|
	$$
	(similarly one can define its  multiplicative analogue $\T^{\times}_k (A)$) 
	in the case when $A$ is a multiplicative subgroup of $\F_p^*$ 
	(see Theorem \ref{t:T_k_G} below). 
	In papers \cite{Bourgain_more}, \cite{Bourgain_DH}, \cite{BGK}, \cite{s_Bourgain} 
	just upper bounds for $\T^{+}_k (A)$ can be found but not an asymptotic formula.

	\begin{theorem}
		Let 
		$\G \subseteq \F_p^*$  be a multiplicative subgroup. 
		Then for any $k\ge 1$ 
		one has 
		\begin{equation}\label{f:T_k_G_intr}
		0\le \T^{+}_{2^k} (\G) - \frac{|\G|^{2^{k+1}}}{p}
		\le 
		2^{3k^2} (C_* \log^4 p)^{k-1} \cdot |\G|^{2^{k+1} - \frac{(k+7)}{2}} \T^{+}_2 (\G) \,,
		\end{equation}
		where $C_* >0$ is an absolute constant.
		\label{t:T_k_G_intr}
	\end{theorem}

	In Section \ref{sec:proof} we obtain new asymptotic formulae and bounds for the quantities 
	$$
	|\{ (a_1-a_2) \dots (a_{2k-1}-a_{2k}) = (a'_1-a'_2) \dots (a'_{2k-1}-a'_{2k}) ~:~ a_i,\, a'_i\in A \}| 
	$$
	as well as for 
	$$
	|\{ a_1 a_2 + \dots + a_{2k-1} a_{2k} = a'_1 a'_2 + \dots + a'_{2k-1} a'_{2k} ~:~ a_i,\, a'_i\in A\}|
	$$
	It allows us to improve estimates for exponential sums of Petridis and Shparlinski \cite{PS}, see Corollary \ref{c:PS_new} below.

	\begin{corollary}
		Given three sets $X,Y,Z\subseteq \F_p$, $|X| \ge |Y| \ge |Z|$ and three complex 
		weights $\rho = (\rho_{x,y})$, $\sigma = (\sigma_{x,z})$, $\tau = (\tau_{y,z})$ all  bounded by one, 
		we have 	
		$$
		\sum_{x \in X,\, y\in Y,\, z\in Z} \rho_{x,y} \sigma_{x,z} \tau_{y,z} e(xyz)  \ll p^{1/8} |X|^{7/8} |Y|^{29/32} |Z|^{29/32} 
		(|Y||Z|)^{-1/3072} \,,
		$$
		provided $|Y|, |Z| < p^{48/97}$.  
		\label{c:PS_new_intr}
	\end{corollary}

	Moreover we obtain a new effective bound for such sums in an optimal range $|X| |Y| |Z| \ge p^{1+\d}$ (and for higher sums). 
	Previously, Bourgain \cite[Theorem 1]{B_multilinear} obtained $(\d/10)^{10^4}$ instead  of $\frac{\d }{8 \log (8/\delta)+4}$, 
	see formula \eqref{f:etropy_exp_intr} below. 
	In general, our saving has the form $p^{-\d/(C_1 \log (C_2 r/\d))^r}$ for $r$ sets instead of $p^{-(\d/r)^{Cr}}$ from \cite[Theorem A]{B_multilinear}.
	Here $C,C_1,C_2>0$ are some absolute constants.

	\begin{theorem}
		Let $X,Y,Z \subseteq \F_p$ be arbitrary sets such that for some $\d >0$ the following holds
		\begin{equation}\label{f:t_mult_3}
		|X| |Y| |Z| \ge p^{1+\d} \,.
		\end{equation}
		Then 
		\begin{equation}\label{f:etropy_exp_intr}
		\sum_{x \in X,\, y\in Y,\, z\in Z} e(xyz) \ll |X||Y||Z| \cdot p^{-\frac{\d }{8 \log (8/\delta)+4}} \,.
		\end{equation}
		\label{t:etropy_exp_intr}
	\end{theorem}

	Our next result is an asymptotic version of the Balog--Wooley \cite{BW} decomposition Theorem, also, see \cite{KS2}, \cite{RRS}, \cite{s_E_k}
	(consult  Theorem \ref{t:BW_as} and Corollary \ref{c:p_sum-prod} below). 
	In particular, it gives us an asymptotic variant of Theorem \ref{t:1/5_F_p} (signs $\lesssim, \gtrsim$  mean some powers of logarithm of $|A|$).

	\begin{theorem}
		Let $A\subseteq \F_p$ be a set and let $1\le M \le p/(2|A|)$ be a parameter.
		There exist two disjoint subsets $B$ and $C$ of $A$ such that $A = B\sqcup C$ 
		and 
		\begin{equation}\label{f:BW_as_1_intr}
		0\le \T_2^{+} (B) - \frac{|B|^4}{p} \le \frac{|A|^{2/3} |B|^{7/3}}{M} \,,
		\end{equation}
		and for  any set $X \subseteq \F_p$ one has
		\begin{equation}\label{f:BW_as_2_intr}
		\T_2^{\times} (C,X)  \lesssim  \frac{M^2 |X|^2 |A|^2 }{p} + M^{3/2} |A|^{} |X|^{3/2} \,.
		\end{equation}
		In particular, for any set $A\subseteq \F_p$ 
		either 
		$$ 
		|A+A| \ge 5^{-1} \min\{|A|^{6/5}, p/2\} 
		$$ 
		or 
		$$|AA| \gtrsim \min\{ p |A|^{-2/5}, |A|^{6/5} \} \,.$$ 
		
		\label{t:BW_as_intr}
	\end{theorem}

	In the last Section \ref{sec:SL} we consider the expansion in $\SL_2 (\F_p)$ and obtain some combinatorial applications of the celebrated Helfgott's growth result.  
	%
	Our first 
	theorem concerns to the intersection of the inverses of additively rich sets $A$,  
	see Corollary \ref{c:1/A}. 
	
	\begin{theorem}
		Let $A, B\subseteq \F_p$, $|B| \ge p^\eps$, $\eps > 0$ and $|A+B| \le K |A|$.
		Then for any $\la \neq 0$ one has 
		$$
		\left| \left\{ \frac{1}{a_1} - \frac{1}{a_2} = \la ~:~ a_1, a_2 \in A \right\} \right| \le 
		\frac{K^2 |A|^2}{p} + 2 K |A| p^{-1/2^{k+2}} \,, 
		$$
		where $k=k(\eps)$. 
		Also,
		$$
		\T_2^{+} (1/A, 1/B) - \frac{K^2 |A|^2 |B|^2}{p}  \ll K^{5/4} |A|^{5/4} |B|^{3/2}  + K^2 |A|^2 \,. 
		$$
		\label{t:1/A_intr}
	\end{theorem}

	Theorem above can be extended to general polynomial  maps (and even to rational functions), see Corollary \ref{c:pol}.

	\begin{theorem}
		Let $p_1,p_2 \in \F_p [x]$ be any non--constant polynomials.
		Then for any $A, B \subseteq \F_p$, $|B| \ge p^\eps$, $\eps>0$ one has 
		$$
		0\le \left| \left\{ p_1(b) + \frac{1}{a+p_2 (b)} = p_1(b') + \frac{1}{a'+p_2 (b')} ~:~ a,a'\in A,\, b,b' \in B \right\} \right|
		- \frac{|A|^2 |B|^2}{p}
		\le
		$$
		\begin{equation}\label{f:pol_intr} 
		\le
		2 |A| |B|^2 p^{-1/2^{k+2}} \,,
		\end{equation}
		where $k=k(\eps, \deg p_1, \deg p_2)$. 
		In particular, 
		$$
		\left| \left\{ p_1(b) + \frac{1}{a+p_2 (b)} ~:~ a \in A,\, b \in B \right\} \right| \gg \min \{ p, |A| p^{1/2^{k+2}} \} \,.
		$$
		\label{t:pol_intr}
	\end{theorem}

	Also, we break the square--root barrier for exponential sums of the form (and many other exponential  sums)
	$$
	e\left( y\left( \frac{1}{x+b_1} + b_2 \right) \right) \,,  \quad \quad 
	e\left(  y\left( p_1(b) + \frac{1}{x+p_2 (b)}  \right) \right) \,, 
	$$
	$$
	\chi \left( y+ b_2 + \frac{1}{x+b_1}  \right) \,, \quad \quad 
	\chi \left( y+ p_1(b) + \frac{1}{x+p_2 (b)}  \right) \,,
	$$
	see Corollaries \ref{c:new_exp_sums}, \ref{c:R[A,B,B]} below.   
	Here the variables $x,y$ belong to some sets $X,Y$, further $b_1,b_2\in B$, $|B| > p^{\eps}$, 
	$e$ and $\chi$ are any non--principal additive/multiplicative characters   and $p_1,p_2 \in \F_p[x]$ are non--constant polynomials.

	\bigskip

	Finally, we 
	obtain 
	an expansion result of another sort (see Corollary \ref{c:GL2_rational} from Section \ref{sec:GL}). 

	\begin{corollary}
		Let $A\subseteq \F_p$, $B_1, B_2, B_3 \subseteq \F_p$, $B:=|B_1|=|B_2|=|B_3| >p^\eps$. 
		Suppose that $|B_3-B_1 B_2| \le B^2 p^{-\eps}$.
		Then there is $\d = \d (\eps) >0$ such that 
		\begin{equation}\label{f:GL2_rational_intr}
		\left| \left\{ \frac{a+b_1}{a b_2  + b_3 } ~:~ a\in A\,, b_j \in B_j \right\} \right| \gg \min \{ p, |A| p^\d \} \,.
		\end{equation}
		\label{c:GL2_rational_intr}
	\end{corollary}



	The author is grateful to 
	Igor Shparlinski, Nikolay Moshchevitin, Brendan Murphy and Maxim Korolev   
	for useful discussions.


\section{Notation}
\label{sec:definitions}

In this paper $p$ is an odd prime number, 
$\F_p = \Z/p\Z$ and $\F_p^* = \F_p \setminus \{0\}$. 
%
%
We denote the Fourier transform of a function  $f : \F_p \to \C$ by~$\FF{f},$
\begin{equation}\label{F:Fourier}
\FF{f}(\xi) =  \sum_{x \in \F_p} f(x) e( -\xi \cdot x) \,,
\end{equation}
where $e(x) = e^{2\pi i x/p}$. 
We rely on the following basic identities. 
The first one is called the Plancherel formula and its particular case $f=g$ is called the Parseval identity 
\begin{equation}\label{F_Par}
\sum_{x\in \F_p} f(x) \ov{g (x)}
=
\frac{1}{p} \sum_{\xi \in \F_p} \widehat{f} (\xi) \ov{\widehat{g} (\xi)} \,.
\end{equation}
Another  particular case of (\ref{F_Par}) is 
\begin{equation}\label{svertka}
\sum_{y\in \F_p} \Big|\sum_{x\in \F_p} f(x) g(y-x) \Big|^2
= \frac{1}{p} \sum_{\xi \in \F_p} \big|\widehat{f} (\xi)\big|^2 \big|\widehat{g} (\xi)\big|^2 \,,
\end{equation}
and the identity 
\begin{equation}\label{f:inverse}
f(x) = \frac{1}{p} \sum_{\xi \in \F_p} \FF{f}(\xi) e(\xi \cdot x) 
\end{equation}
is called the inversion formula. 
Further let $f,g : \f_p \to \C$ be two functions.
Put
\begin{equation}\label{f:convolutions}
(f*g) (x) := \sum_{y\in \f_p} f(y) g(x-y) \quad \mbox{ and } \quad
(f\circ g) (x) := \sum_{y\in \f_p} \ov{f(y)} g(y+x)  \,.
\end{equation}
Then
\begin{equation}\label{f:F_svertka}
\FF{f*g} = \FF{f} \FF{g} \quad \mbox{ and } \quad 
\FF{f \circ g} 
= \ov{\FF{f}} \FF{g} \,.
\end{equation}
Put
$\E^{+}(A,B)$ for the {\it common additive energy} of two sets $A,B \subseteq \f_p$
(see, e.g., \cite{TV}), that is, 
$$
\E^{+} (A,B) = |\{ (a_1,a_2,b_1,b_2) \in A\times A \times B \times B ~:~ a_1+b_1 = a_2+b_2 \}| \,.
$$
If $A=B$, then  we simply write $\E^{+} (A)$ instead of $\E^{+} (A,A)$
and the quantity $\E^{+} (A)$ is called the {\it additive energy} in this case. 
Clearly,
\begin{equation*}\label{f:energy_convolution}
\E^{+} (A,B) = \sum_x (A*B) (x)^2 = \sum_x (A \circ B) (x)^2 = \sum_x (A \circ A) (x) (B \circ B) (x)
\end{equation*}
and by (\ref{svertka}),
\begin{equation}\label{f:energy_Fourier}
\E(A,B) = \frac{1}{p} \sum_{\xi} |\FF{A} (\xi)|^2 |\FF{B} (\xi)|^2 \,.
\end{equation}
Also, notice that
\begin{equation}\label{f:E_CS}
\E^{+} (A,B) \le \min \{ |A|^2 |B|, |B|^2 |A|, |A|^{3/2} |B|^{3/2} \} \,.
\end{equation}
Sometimes we write $\E^{+}(f_1,f_2,f_3,f_4)$ for the additive energy of four real functions, namely,
$$
\E^{+}(f_1,f_2,f_3,f_4) = \sum_{x,y,z} f_1 (x) f_2 (y) f_3 (x+z)  f_4 (y+z) \,.
$$
It can be shown using the H\"older inequality (see, e.g., \cite{TV}) that 
\begin{equation}\label{f:E_Ho}
\E^{+}(f_1,f_2,f_3,f_4) \le (\E^{+} (f_1) \E^{+} (f_1) \E^{+} (f_1) \E^{+} (f_1))^{1/4} \,.
\end{equation}
In the same way define the {\it common multiplicative energy} of two sets $A,B \subseteq \f_p$
$$
\E^{\times} (A,B) =  |\{ (a_1,a_2,b_1,b_2) \in A\times A \times B \times B ~:~ a_1 b_1 = a_2 b_2 \}| \,. 
$$
Certainly, the multiplicative energy $\E^{\times} (A,B)$ can be expressed in terms of multiplicative convolutions 
similar to (\ref{f:convolutions}).
Further the definitions and the formulae above take place in an arbitrary abelian group $\Gr$.  
If there is no difference between $\E^{+}$ and $\E^\times$ or there is the only operation on the considered group $\Gr$, then we write just $\E$.

Sometimes we  use representation function notations like $r_{AB} (x)$ or $r_{A+B} (x)$, which counts the number of ways $x \in \F_p$ can be expressed as a product $ab$ or a sum $a+b$ with $a\in A$, $b\in B$, respectively. 
For example, $|A| = r_{A-A}(0)$ and  $\E^{+} (A) = r_{A+A-A-A}(0)=\sum_x r^2_{A+A} (x) = \sum_x r^2_{A-A} (x)$.  
In this paper we use the same letter to denote a set $A\subseteq \F_p$ and  its characteristic function $A: \F_p \to \{0,1 \}$. 
Thus $r_{A+B} (x) = (A*B) (x)$, say.
Having $P\subseteq A-A$ we write $\sigma_P (A) := \sum_{x \in P} r_{A-A} (x)$. 
Also, we write $f_A (x)$ for the balanced function of a set $A\subseteq \F_p$, namely, $f_A (x) = A(x) - |A|/p$.

Now consider two families of higher energies. Firstly, let
\begin{equation}\label{def:T_k}
\T^{+}_k (A) := |\{ (a_1,\dots,a_k,a'_1,\dots,a'_k) \in A^{2k} ~:~ a_1 + \dots + a_k = a'_1 + \dots + a'_k \}|
=
\frac{1}{p} \sum_{\xi} |\FF{A} (\xi)|^{2k}
\,.
\end{equation}
It is useful to note that 
$$
\T^{+}_{2k} (A) = |\{ (a_1,\dots,a_{2k},a'_1,\dots,a'_{2k}) \in A^{4k} ~:~ (a_1 + \dots + a_k) + (a_{k+1} + \dots + a_{2k})=
$$
$$
= (a'_1 + \dots + a'_k) + (a'_{k+1} + \dots + a'_{2k}) \}| 
=
$$
\begin{equation}\label{f:T_2k_r}
= \sum_{x,y,z} r_{kA} (x) r_{kA} (y) r_{kA} (x+z) r_{kA} (y+z) \,,
\end{equation}
so one can rewrite $\T^{+}_{2k} (A)$ via the additive energy of the function $r_{kA} (x)$. 
Sometimes we use $\T^{+}_k (f)$ for an arbitrary function $f$. 
It is easy to check that
\begin{equation}\label{f:T_f_12}
\T^{+}_k (f) \le \| f\|^{2}_1 \T^{+}_{k-1} (f) \,,
\end{equation}
and hence by the Parseval identity 
\begin{equation}\label{f:T_f_12'}
\T^{+}_k (f) \le  \| f\|^{2k-2}_1  \| f\|^2_2 \,.
\end{equation}
Secondly, for $k\ge 2$, we put 
\begin{equation}\label{def:E_k}
\E^{+}_k (A) = \sum_{x\in \F_p} (A\c A)(x)^k = \sum_{x\in \F_p} r_{A-A}^k (x) = 
\E^{+} (\Delta_k (A),A^{k}) \,,
\end{equation}
where 
$$
\Delta_k (A) := \{ (a,a, \dots, a)\in A^k \}\,.
$$ 
Thus $\E^{+}_2 (A) = \T^{+}_2 (A)= \E^{+} (A)$. 
Also, notice that we always have 
$|A|^k \le \E^{+}_k (A) \le |A|^{k+1}$ and moreover 
\begin{equation}\label{f:E_k_crude}
\E^{+}_k (A) \le |A|^{k-l} \E^{+}_l (A) \,, \quad \quad \forall l\le k \,.
\end{equation} 
Finally, let us remark that by definition (\ref{def:E_k}) one has $\E^{+}_1 (A) = |A|^2$.
Similarly, one can define $\E^{+} (f)$ for an arbitrary function $f$. 
From the inversion formula and the Parseval identity, it follows that
\begin{equation}\label{f:E_k_Fourier}
\E^{+}_k (f) = p^{1-k} \sum_{x_1+\dots+x_k = 0} |\FF{f} (x_1)|^2 \dots |\FF{f} (x_k)|^2 \ge 0 \,.
\end{equation}  
Some
further 
results about the  properties of the energies $\E^{+}_k$ can be found in \cite{SS1}.
Again, sometimes we use $\E^{+}_k (f)$ for an arbitrary function $f$
and the first  formula from   (\ref{def:E_k}) allows to define $\E^{+}_k (A)$ for any positive $k$.  
It was proved in \cite[Proposition 16]{s_E_k} that  $(\E^{+}_k (f))^{1/2k}$ is a norm for even $k$ and a real function $f$. 
The fact that $(\T^{+}_k (f))^{1/2k}$ is a norm is contained in \cite{TV} and follows from a generalization of inequality (\ref{f:E_Ho}).

We write  $\d\{ x\} =1$ if $x=0$ and $\d\{ x\} =0$
otherwise. 


All logarithms are to base $2.$ The signs $\ll$ and $\gg$ are the usual Vinogradov symbols.
If we have a set $A$, then we will write $a \lesssim b$ or $b \gtrsim a$ if $a = O(b \cdot \log^c |A|)$, $c>0$.
When the constants in the signs  depend on a parameter $M$, we write $\ll_M$ and $\gg_M$. 
For a positive integer $n,$ we set $[n]=\{1,\ldots,n\}.$ 
We do not normalize $L_p$--norms of functions. So, $\|f\|_p = \left( \sum_{x} |f(x)|^p \right)^{1/p}$ for any complex function $f$.



\section{Preliminaries}
\label{sec:preliminaries}

First of all, we need a general design bound for the number of incidences.  
Let $\mathcal{P} \subseteq \F_q^3$ be a set of points  and $\Pi$ be a collection of planes in $\F_q^3$. 
Having $p\in \mathcal{P}$ and $\pi \in \Pi$ we write 
\begin{displaymath}
\I (p,\pi) = \left\{ \begin{array}{ll}
1 & \textrm{if } q\in \pi\\
0 & \textrm{otherwise}
\end{array} \right.
\end{displaymath}
So, $\I$ is $|\P| \times |\Pi|$ matrix. If $\P = \F_q^3$ and $\Pi$ is the family of all planes in $\F_q^3$, then we obtain the matrix $\ov{\I}$ and $\I$ is a submatrix of $\ov{\I}$. 
One can easily calculate $\ov{\I}^* \ov{\I}$ and $\ov{\I}\, \ov{\I}^*$ in the projective plane 
$\mathbb{P}\F_q^3$ 
and check that both of this matrices have form 
$aI + bJ$, where $a,b$ some numbers, $I$ is the identity matrix and $J$ is the matrix having all entries equal 1's see, e.g., \cite{TTT,Vinh}. 
In particular, vectors  $(1,\dots,1)$ of the correspondent lengths are the first eigenvectors of $\ov{\I}^* \ov{\I}$ and $\ov{\I}\, \ov{\I}^*$. 
Moreover, one can check that in our case of points and planes the following holds  $a = q^2$ and $b= q+1$ (see \cite{TTT,Vinh}). 
Having 
these facts in mind and using the singular decomposition (see, e.g., \cite{Horn-Johnson}), we 
derive 
that 
for any functions $\a : \P \to \C$, $\beta : \Pi \to \C$ one has 
\begin{equation}\label{f:Vinh}
|\sum_{p,\pi} \I (p,\pi) \a(p) \beta(\pi) | = |\sum_{p,\pi} \ov{\I} (p,\pi) \a(p) \beta(\pi)| \le q \| \a \|_2 \| \beta \|_2 \,,
\end{equation}
provided either $\sum_{p\in \P} \a(p) = 0$ or $\sum_{\pi\in \Pi} \beta (\pi) = 0$.
Of course, similar arguments work not just for points/planes incidences but, e.g., points/lines incidences and so on.

\bigskip

A much more deep 
result on incidences is contained in \cite{Rudnev_pp} (or see \cite[Theorem 8]{MPR-NRS} and the proof of Corollary 2 from \cite{MP}).
In the proof of formula (\ref{f:Misha+_a}) one should use an incidence bound from \cite[Section 3]{MP_MJCNT}.

\begin{theorem}
	Let $p$ be an odd prime, $\mathcal{P} \subseteq \F_p^3$ be a set of points and $\Pi$ be a collection of planes in $\F_p^3$. 
	Suppose that $|\mathcal{P}| \le |\Pi|$ and that $k$ is the maximum number of collinear points in $\mathcal{P}$. 
	Then the number of point--planes incidences satisfies 
	\begin{equation}\label{f:Misha+}
	\mathcal{I} (\mathcal{P}, \Pi)  \ll \frac{|\mathcal{P}| |\Pi|}{p} + |\mathcal{P}|^{1/2} |\Pi| + k |\mathcal{P}| \,.	
	\end{equation}
	More precisely,
	\begin{equation}\label{f:Misha+_a}
	\mathcal{I} (\mathcal{P}, \Pi)  - \frac{|\mathcal{P}| |\Pi|}{p} \ll |\mathcal{P}|^{1/2} |\Pi| + k |\mathcal{P}| \,.	
	\end{equation}
	\label{t:Misha+}	
\end{theorem}


\begin{corollary}
	Let $\a,\beta$ be non--negative functions, $C\subseteq \F_p$ be a set. 
	Suppose that 
	\begin{equation}\label{f:weight_inc_cond}
	\max\{ \|\a\|_1 \|\beta\|^{-1}_1 \|\a\|^{-1}_2 \|\beta\|_2, \|\a\|^{-1}_1 \|\beta\|_1 \|\a\|_2 \|\beta\|^{-1}_2 \}\le |C|^{1/2} \le \frac{\|\a\|_1 \|\beta\|_1}{\|\a\|_2 \|\beta\|_2}
	\end{equation}
	and put $L = \log (\|\a\|_1 \|\beta\|_1 |C| / (\|\a\|_2 \|\beta\|_2 ))$.
	Then
	\begin{equation}\label{f:weight_inc1} 
	\sum_x r^{2}_{\a\beta+C} (x) - \frac{(\|\a\|_1 \|\beta\|_1 |C|)^2}{p} 
	\ll L^4 \|\a\|_1 \|\beta\|_1 \|\a\|_2 \|\beta\|_2 |C|^{3/2} \,,
	\end{equation}
	and
	\begin{equation}\label{f:weight_inc2}
	\sum_x r^{2}_{\a(\beta+C)} (x) - \frac{(\|\a\|_1 \|\beta\|_1 |C|)^2}{p}  
	\ll L^4 \|\a\|_1 \|\beta\|_1 \|\a\|_2 \|\beta\|_2 |C|^{3/2} \,.
	\end{equation}
	\label{cor:weight_inc}
\end{corollary}
\begin{proof}
	We obtain (\ref{f:weight_inc1}) because the proof of (\ref{f:weight_inc2}) is similar.
	Let $f (x) = C(x) - |C| / p$.
	Then 
	$$
	\sum_x r^{2}_{\a\beta+\gamma} (x) = \frac{(\|\a\|_1 \|\beta\|_1 |C|)^2}{p} + \sum_x r^{2}_{\a\beta+f} (x) \,.
	$$
	Split the level set of $\a,\beta$ into level  sets $P_j (\a)$, $P_j (\beta)$  where the functions $\a$, $\beta$ differ at most twice, correspondingly. 
	Clearly, there are at most $L$ such sets because if, say, $\a (x) \le \eps := 2^{-2}|C|^{-1/2} \|\beta\|^{-1}_1 \|\a\|_2 \|\beta\|_2$, then
	$$
	\eps \|\a\|_1 |C|^2 \|\beta\|^2_1 \le 2^{-2} |C|^{3/2} \|\a\|_1 \|\beta\|_1 \|\a\|_2 \|\beta\|_2 \,,
	$$
	so it is negligible and hence the inequality  $2^j \eps \le \|\a\|_1$ gives the required bound. 
	Using the pigeonhole principle and positivity of the operator $r_{f-f} (x-y)$, we find some $\D_A,\D_B$ and $A \subseteq P_j (\a)$, 
	$B\subseteq P_j (\beta)$ such that
	$$
	\sigma:= \sum_x r^{2}_{\a\beta+f} (x) \ll L^4 \D^2_A \D^2_B \sum_x r^{2}_{AB+f} (x) \,.
	$$
	On the one hand, in view of (\ref{f:Vinh}) the last sum is bounded by 
	\begin{equation}\label{tmp:23.02.2018_1}
	\sigma \le  L^4 \D^2_A \D^2_B p \|f\|^2_2 |A| |B| \le L^4 \D^2_A \D^2_B p |A| |B| |C| \,.
	\end{equation}
	On the other hand, using Theorem \ref{t:Misha+} (one can consult paper \cite{AMRS})
	$$
	\sum_x r^{2}_{AB+f} (x) = \frac{|A|^2 |B|^2 |C|^2}{p} + \sum_x r^{2}_{AB+C} (x) 
	\ll 
	$$
	$$
	\ll 
	\frac{|A|^2 |B|^2 |C|^2}{p} + (|A||B||C|)^{3/2} + |A||B||C| \max\{ |A|, |B|, |C|\} \,.
	$$
	If the second term in the last formula dominates, then we are done.
	If the first term is the largest one,  then $p\le (|A||B||C|)^{1/2}$ and \eqref{tmp:23.02.2018_1} gives us \
	$$
	\sigma \le L^4 \D^2_A \D^2_B p |A| |B| |C| \le L^4 \D^2_A \D^2_B  (|A||B||C|)^{3/2} \le L^4 \|\a\|_1 \|\beta\|_1 \|\a\|_2 \|\beta\|_2 |C|^{3/2} 
	$$
	as required. 
	Finally, condition (\ref{f:weight_inc_cond}) implies that the third term is negligible. 
	This completes the proof. 
	$\hfill\Box$
\end{proof}

\bigskip

Now we obtain a simple  asymptotic formula for the number of points/lines incidences in the case when  the set of points form a Cartesian product.

\begin{lemma}
	Let $A,B \subseteq \F_p$ be sets, $|A| \le |B|$, $\mathcal{P} = A\times B$, and $ \mathcal{L}$ be a collection of lines in $\F^2_p$.
	Then 
	\begin{equation}\label{f:line/point_as}
	\mathcal{I}(\mathcal{P}, \mathcal{L}) - \frac{|A| |B| |\mathcal{L}|}{p} \ll |A|^{3/4} |B|^{1/2} |\mathcal{L}|^{3/4} + |\mathcal{L}| + |A| |B| \,.
	\end{equation}
	\label{l:line/point_as}	
\end{lemma}
\begin{proof}
	Let $f(x) = B (x) - |B|/p$ be the balanced function of the set $B$. 
	Then, using the natural notation, we get 
	$$
	\mathcal{I}(\mathcal{P}, \mathcal{L}) = \frac{|A||B| |\mathcal{L}|}{p} + \mathcal{I}(f \otimes A, \mathcal{L}) \,,
	$$
	where we count the number of incidences with the weight $f(x) A(a)$.
	Using the design bound for points/lines incidences, we obtain
	\begin{equation}\label{tmp:08.01.2018_1}
	\mathcal{I}(f \otimes A, \mathcal{L}) \le \| f \|_2 (p|A| |\mathcal{L}|)^{1/2} \le  (p|A| |B| |\mathcal{L}|)^{1/2}  \,.
	\end{equation}
	By an analogue of the Szemer\'edi--Trotter theorem in $\F_p$, see \cite{SZ_inc} (or \cite[Theorem 7]{MP}, \cite[Theorem 9]{MPR-NRS}), 
	we have 
	\begin{equation}\label{tmp:08.01.2018_2}
	\mathcal{I}(f \otimes A, \mathcal{L}) \ll |A|^{3/4} |B|^{1/2} |\mathcal{L}|^{3/4} + |\mathcal{L}| + |A| |B| \,,
	\end{equation} 
	provided $|A| |\mathcal{L}| \le p^2$. 
	But if $|A| |\mathcal{L}| > p^2$, then by (\ref{tmp:08.01.2018_1}), we see that 
	$$
	\mathcal{I}(f \otimes A, \mathcal{L}) \le  (p|A| |B| |\mathcal{L}|)^{1/2}   \le |A|^{3/4} |B|^{1/2} |\mathcal{L}|^{3/4} 
	$$ 		
	and the last bound is even better than (\ref{tmp:08.01.2018_2}).	
	This completes the proof. 
	$\hfill\Box$
\end{proof}

\bigskip

We need a lemma from \cite{s_Bourgain} which is a consequences of the main result from \cite{Rudnev_pp} or Theorem \ref{t:Misha+}.


\begin{lemma}
	Let $A,Q\subseteq \F_p$ be two sets, $A,Q\neq \{0\}$, $M\ge 1$ be a real number, and
	$|QA|\le M|Q|$. 
	Then
	\begin{equation}\label{f:AA_small_energy}
	\E^{+} (Q) \le C_* \left( \frac{M^2 |Q|^4}{p} + \frac{M^{3/2} |Q|^3}{|A|^{1/2}} \right)  \,,
	\end{equation} 
	where $C_* \ge 1$ is an absolute constant. 
	\label{l:AA_small_energy}
\end{lemma}


The second lemma can be obtained in the same vein.

\begin{lemma}
	Let $A, B\subseteq \F_p$, 
	and $|A+B| \le K |A|$.
	Then 
	\begin{equation}\label{f:1/A_energy}
	\E^{+} (1/A,1/B) - \frac{K^2 |A|^2 |B|^2}{p} \ll K^{5/4} |A|^{5/4} |B|^{3/2}  + K^2 |A|^2 \,. 
	\end{equation}
	\label{l:1/A_energy}
\end{lemma}
\begin{proof}
	Indeed, for any $\a,\beta$ the following holds 
	$$
	\left( \frac{1}{\a} + \frac{1}{\beta} \right)^{-1} = \frac{\a \beta}{\a + \beta} = \beta - \beta^2 \cdot \frac{1}{\a + \beta} \,.
	$$
	Hence
	$$
	\E^{+} (1/A,1/B) \le \left| \left\{ b_1 - b^2_1 x = b_2 - b^2_2 y ~:~ b_1,b_2 \in B,\, x,y \in (A+B)^{-1} \right\}\right| 
	=
	\mathcal{I}(\mathcal{P}, \mathcal{L}) \,,
	$$
	where $\mathcal{P} = (A+B) \times (A+B)$, $\mathcal{L} = \{ l_{b_1,b_2} \}$ and line $l_{b_1,b_2}$ is defined by the equation
	$b_1 - b^2_1 x = b_2 - b^2_2 y$. 
	Applying Lemma \ref{l:line/point_as}, we get 	
	$$
	\mathcal{I}(\mathcal{P}, \mathcal{L}) - \frac{|A+B|^2 |B|^2}{p} 
	\ll
	|A+B|^{5/4} |B|^{3/2} + |B|^2 + |A+B|^2 \,. 
	$$
	Clearly, $|B| \le |A+B| \le K|A|$ and hence 
	$$
	\E^{+} (1/A,1/B)  - \frac{K^2 |A|^2 |B|^2}{p}  \ll K^{5/4} |A|^{5/4} |B|^{3/2}  + K^2 |A|^2
	$$
	as required. 
	$\hfill\Box$
\end{proof}

\bigskip

The next result is a slight generalization of \cite[Lemma 10]{s_Bourgain}.

\begin{lemma}
	Let 
	$f$ be a real 
	function and $P\subseteq \F^*_p$ be a set.
	Then for any $k\ge 1$ one has 
	\begin{equation}\label{f:change_QG}	
	\left( \sum_{x\in P} r^k_{f-f} (x)  \right)^4 \le 
	\| f\|^{4k}_2 \E^{+}_{2k} (f) \E^{+} (P)  \,.
	\end{equation}
	\label{l:change_QG}	
\end{lemma}
\begin{proof}
	We have 
	$$
	\left( \sum_{x\in P} r^k_{f-f} (x)  \right)^2
	=  \left( \sum_{x_1,\dots,x_k} \prod_{j=1}^k f(x_j) \sum_{y} P(y) f (y+x_1) \dots f(y+x_k)  \right)^2
	\le
	$$
	$$
	\le
	\| f\|^{2k}_2 \sum_{x_1,\dots,x_k} |\sum_{y} P(y) f (y+x_1) \dots f(y+x_k)|^2
	=
	\| f\|^{2k}_2 \sum_x r_{P-P} (x) r^k_{f-f} (x) \,.
	$$
	Hence by the Cauchy--Schwarz inequality, we obtain
	$$
	\left( \sum_{x\in X} r^k_{f-f} (x)  \right)^4
	\le
	\| f\|^{4k}_2 \E^{+}_{2k} (f) \E^{+} (P) 
	$$
	as required. 
	$\hfill\Box$
\end{proof}



\bigskip

Let $A,B,C,D \subseteq \F_p$ be four sets.
By $\Q(A,B,C,D)$ we denote the number of {\it collinear quadruples} in $A\times A$, $B\times B$, $C\times C$, $D\times D$.
If $A=B=C=D$, then we write $\Q(A)$ for $\Q(A,A,A,A)$.
Recent results on the quantity $\Q(A)$ can be found in \cite{Petridis_quadruples} and \cite{MPR-NRS}.
It is easy to see (or consult \cite{MPR-NRS}) that 
\begin{equation}\label{def:Q}
\Q(A,B,C,D)  = \left| \left\{ \frac{b'-a'}{b-a} =  \frac{c'-a'}{c-a} = \frac{d'-a'}{d-a} ~:~ a,a'\in A,\, b,b'\in B,\, c,c'\in C,\, d,d'\in D \right\} \right| 
\end{equation}
\begin{equation}\label{f:Q_E_3}
=
\sum_{a,a'\in A} \sum_x r_{(B-a)/(B-a')} (x)  r_{(C-a)/(C-a')} (x)   r_{(D-a)/(D-a')} (x)  \,.
\end{equation}
Notice that in (\ref{def:Q}), we mean that the condition, say, $b=a$ implies $c=d=b=a$ or, in other words, that all four points $(a,a'), (b,b'), (c,c'), (d,d')$ have the same abscissa.
More rigorously, the summation in (\ref{f:Q_E_3})  should be taken over 
$\F_p \cup \{+\infty\}$,
where $x=+\infty$ means that the denominator in any fraction $x=\frac{b'-a'}{b-a}$ from, say, $r_{(B-a)/(B-a')} (x)$  equals zero. 
Anyway, it is easy to see that the contribution of the point $+\infty$ is at most $O(M^5)$, where $M=\max\{ |A|, |B|, |C|, |D| \}$, and hence it is negligible (see, say, Theorem \ref{t:Q} above). 
Further defining a function $q_{A,B,C,D} (x,y)$ (see \cite{MPR-NRS}) as 
\begin{equation}\label{f:def_t,q}
q_{A,B,C,D} (x,y) :=  \left| \left\{ \frac{b-a}{c-a} = x,\, \frac{d-a}{c-a} = y ~:~  a\in A,\, b\in B,\, c\in C,\, d\in D \right\} \right| \,,
\end{equation}
we obtain another formula for the quantity $\Q(A,B,C,D)$, namely, 
$$
\Q(A,B,C,D) 	= \sum_{x,y} q^2_{A,B,C,D} (x,y) \,.
$$


An optimal  (up to logarithms factors) upper bound for $\Q(A)$ was obtained in \cite{MPR-NRS}, \cite{Petridis_quadruples}, see Theorem \ref{t:Q} from the Introduction. 
We need  a simple lemma about the same bound for a generalization of the quantity $\Q(A)$.
The proof is analogous to the proof \cite[Lemma 6]{ShZ} and \cite[Lemma 5]{s_sumset}.

\begin{lemma}
	Let $A,B \subseteq \F_p$ be two sets, 
	$|B| \le |A| \le \sqrt{p}$.  
	Then 
	\begin{equation}\label{f:Q(A,B,C,D)}
	\Q(B,A,A,A) \ll |A|^{15/4}  |B|^{5/4} \log^2 |A| + \T(A) 
	\,.
	\end{equation}
	\label{l:Q(A,B,C,D)}
\end{lemma}

It is known \cite[Proposition 2.5]{AMRS} that $\T(A) \ll |A|^{9/2}$, provided $|A| \le p^{2/3}$ (also, see  Theorem \ref{t:Q} from the Introduction). 
So, the term $\T(A)$ in (\ref{f:Q(A,B,C,D)}) is negligible if $A$ and $B$ have comparable sizes, say.  

\bigskip


Proposition 16 from \cite{RSS} contains a combinatorial lemma, see Lemma \ref{l:Misha_c} below. 

\begin{lemma}
	Let $(\Gr,+)$ be an abelian group.
	Also, let $A\subseteq \Gr$ be a set, $P\subseteq A-A$, $P=-P$.
	Then there is $A_* \subseteq A$ and a number $q$, $q \lesssim |A_*|$
	such that for any $x\in A_*$ one has $r_{A+P} (x) \ge q$,
	and $\sum_{x\in P} r_{A-A} (x) \sim |A_*| q$.
	\label{l:Misha_c}
\end{lemma}

Another combinatorial result is \cite[Theorem 13]{RSS}.

\begin{theorem}
	Let $(\Gr,+)$ be an abelian group.
	Also, let $A\subseteq \Gr$ be a set, $K\ge 1$ be a real number, and $k\ge 2$ be an integer.
	Suppose that $\E^+ (A) \ge |A|^3 / K$.
	Then there are sets $A_* \subseteq A$, $P \subseteq A-A$ such that $|A_*| \ge |A|/(8kK)$,
	$|P| \le 8kK|A|$
	and for any $a_1,\dots, a_k \in A_*$ one has
	\begin{equation}\label{f:BSzG_Schoen}
	|A \cap (P+a_1) \cap \dots \cap (P+a_k)| \ge \frac{|A|}{4K} \,.
	\end{equation}
	\label{t:BSzG_Schoen}
\end{theorem}

We need a result on the energy of a set which is obtained using the eigenvalues method, see \cite{s_ineq}, \cite{s_mixed}, \cite{MPR-NRS}, \cite{MRSS}. 
In this form an analogue of the result above was appeared first time in \cite[Theorem 28]{MPR-NRS}. 
One can decrease number of logarithms slightly but it is not our aim.

\begin{theorem}
	Let $A$ be a finite subset of an abelian group $(\Gr,+)$.
	Suppose there are parameters $D_1$ and $D_2$ such that
	\[
	\E^+_3(A) \leq D_1|A|^3
	\]
	and for any finite set $B\subset \Gr$
	\[
	\E^+ (A,B) \leq D_2|A||B|^{3/2}.
	\]
	Then
	\[
	\E^+ (A) \ll D_1^{6/13}D_2^{2/13}|A|^{32/13} \log^{12/13}|A| \,.
	\]
	\label{thm:eigenval}
\end{theorem}


It is implicit in the proof of Theorem~\ref{thm:eigenval} that the bound for $\E(A,B)$ only needs to hold for $|B|\leq 4|A|^4/\E^{+}(A)$.

\bigskip 

Theorem~\ref{thm:eigenval}
implies the following bound for the multiplicative energy of a subset of $\F_p$ with
large additive energy.

\begin{corollary}
	Let $A\subseteq \F_p$ and $\E^{+} (A) \ge |A|^3/K$. 
	Then there is $A_*\subseteq A$, $|A_*| \ge |A|/(16K)$ such that for any 
	$B\subseteq \F_p$ 
	the following holds 
	\begin{equation}\label{f:pre1}
	\E^\times (A_*,B) \ll \frac{K^4 |A|^2 |B|^2}{p} +  K^{7/2} |B|^{3/2} |A| \,,
	\end{equation}
	and if $|A|K \le \sqrt{p}$, then 
	\begin{equation}\label{f:pre2}
	\E^\times_3 (A_*) \ll K^{23/4} |A|^3 \log^2 |A| \,.
	\end{equation}
	In particular, if $|A|K \le \sqrt{p}$, then 
	\begin{equation}\label{f:pre3}
	\E^\times (A_*) \lesssim K^{83/26} |A|^{32/13} \,.
	\end{equation}
	\label{c:pre_eigen}
\end{corollary}
\begin{proof}
	Applying Theorem \ref{t:BSzG_Schoen} with $k=3$, we find two sets $A_* \subseteq A$, $P \subseteq A-A$, $|A_*| \ge |A|/(24K)$, $|P| \le 24K|A|$
	such that for any $a_1,a_2, a_3 \in A_*$ one has
	\begin{equation}\label{tmp:26.10_0}
	|A \cap (P+a_1) \cap (P+a_2) \cap (P+a_3)| \ge \frac{|A|}{4K} \,.
	\end{equation}
	Then
	\begin{equation}\label{tmp:26.10_1}
	\E^\times (A_*,B) \le (|A|/ 4K)^{-2} \left| \left\{ (a-p)b = (a'-p')b' ~:~ a,a'\in A,\, b,b'\in B,\, p,p'\in P \right\}\right| \,. 
	\end{equation}
	Clearly, the number of the  solutions to equation (\ref{tmp:26.10_1}) can be interpreted as points/planes incidences. 
	Hence 
	applying 
	Theorem \ref{t:Misha+}, we obtain 
	\begin{equation}\label{tmp:26.10_2}
	\E^\times (A_*,B) \ll (|A|/ K)^{-2} \left( \frac{|A|^2 |B|^2 |P|^2}{p} + (|A||B||P|)^{3/2} + |A||B||P| \max\{|A|,|B|,|P|\}  \right) \,.
	\end{equation}
	In view of the desired bound  (\ref{f:pre1}) one can assume that $|B|\ge K^7$, $|A| \ge |B|^{1/2} K^{7/2}$ (otherwise trivial bounds (\ref{f:E_CS}), namely,  
	$\E^\times (A_*,B)\le \min \{ |A| |B|^2, |A|^2 |B| \}$ work better). 
	Also, 
	(\ref{tmp:26.10_0}) implies, trivially, $|P| \ge |A|/(4K)$ and we can assume that $|B|\leq 4|A|^4/\E^{+}(A) \ll K|A|$.
	Thus it is easy to check that the third term in  (\ref{tmp:26.10_2}) is negligible and using $|P| \ll K|A|$,  we obtain (\ref{f:pre1}).

	To prove  (\ref{f:pre2}) we notice that in view of (\ref{tmp:26.10_0}) and (\ref{def:Q}) one has 
	$$
	\E^\times_3 (A_*) = \left| \left\{ \a/\a' = \beta/\beta' = \gamma/\gamma' ~:~ \a,\a', \beta,\beta', \gamma, \gamma' \in A_* \right\}\right| 
	\le
	$$
	$$
	\le 
	(|A|/ 4K)^{-2} \left| \left\{ \frac{b-a}{b'-a'} = \frac{c-a}{c'-a'} = \frac{d-a}{d'-a'} ~:~ a,a'\in A,\, b,b',c,c',d,d' \in P \right\}\right| 
	\ll
	$$
	$$
	\ll
	(|A|/  K)^{-2}  \Q(A,P,P,P)	\,.
	$$
	Suppose that $|A|\le |P| \le \sqrt{p}$. 
	One can assume that $K\le |A|^{4/23}$ because otherwise there is nothing to prove. 
	It remains to estimate $\Q(A,P,P,P)$ and we have by Lemma \ref{l:Q(A,B,C,D)} that 
	\begin{equation}\label{tmp:08.01_3}
	\Q(A,P,P,P) \ll |P|^{15/4} |A|^{5/4} \log^2 |A| + \T(P) \,.
	\end{equation}
	Thus in view of $\T(P) \ll |P|^{9/2}$, see \cite[Proposition 2.5]{AMRS} the second term in 
	(\ref{tmp:08.01_3}) 
	is negligible.  
	Then applying Lemma \ref{l:Q(A,B,C,D)} and the bound $|P| \le 24 K|A|$, we obtain  (\ref{f:pre2}).
	If $|A| > |P|$, then we get even better estimate for $\E^\times_3 (A_*)$. 
	Finally, using Theorem \ref{thm:eigenval}, we derive from (\ref{f:pre1}), (\ref{f:pre2}) the desired bound (\ref{f:pre3})
	(because $|B|\leq 4|A|^4/\E^{+}(A)$ and $|A| K \le \sqrt{p}$ we see that the second term in (\ref{f:pre1}) dominates). 
	This completes the proof. 
	$\hfill\Box$ 
\end{proof}

\bigskip

In \cite{MRSS} some better bounds for the energy were obtained but they work in a  situation which is opposite to Corollary \ref{c:pre_eigen}, namely, when the product set (not the sumset) is small.

\bigskip

Now consider the group $\SL_2 (\F_p)$  of matrices 
\[
g=
\left( {\begin{array}{cc}
	a & b \\
	c & d \\
	\end{array} } \right) \,, \quad \quad a,b,c,d\in \F_p \,, \quad \quad ad-bc=1 \,,
\]
which acts on $\F_p$ by 
$$
g z := \frac{az+b}{cz+d}  \,,\quad \quad z \in \F_p \,.
$$
There are two important subgroups in $\SL_2 (\F_p)$. 
Let $\B$ be the standard  Borel subgroup of upper--triangular matrices, namely, elements of $\B$ are 
\[
b = b_{r,q} = 
\left( {\begin{array}{cc}
	r & q \\
	0 & r^{-1} \\
	\end{array} } \right) \,, \quad \quad q\in \F_p \,,  \quad r\in \F_p \setminus \{0\}  \,.
\]
Also, let $\U \subseteq \B$ be the standard unipotent subgroup. 
In other words, elements of $\U$ are 
\[
u = u_q = 
\left( {\begin{array}{cc}
	1 & q \\
	0 & 1 \\
	\end{array} } \right) \,, \quad \quad q\in \F_p  \,.
\]

Having a group which is acting of a set,  one can define a convolution which is slightly generalizes  the ordinary convolution. 

\begin{definition}
	Let $F: \SL_2 (\F_p) \to \mathbb{C}$ and $f : \F_p \to \mathbb{C}$ be two functions.
	Define the 
	convolution  of $F * f : \F_p \to \mathbb{C}$ as 
	$$
	(F*f) (x) := \sum_{g\in \SL_2 (\F_p)} F(g) f(g^{-1} x) \,.
	$$
\end{definition}

Let us 
mention 
a well--known lemma (see \cite{BG_SL},  \cite{Gill}, \cite{SX} and other papers) on convolutions in $\SL_2 (\F_p)$ which follows from 
the 
well--known 
Frobenius Theorem \cite{Frobenius} on representations of $\SL_2 (\F_p)$. 
For the sake of completeness we add the proof of this lemma in the Appendix. 

\begin{lemma}
	Let $f : \F_p \to \mathbb{C}$ be a function such that $\sum_x f(x) = 0$. 
	Then for any functions $F: \SL_2 (\F_p) \to \mathbb{C}$ and $\_phi : \F_p \to \mathbb{C}$ one has 
	\begin{equation}\label{f:Frobenious}
	\sum_{x\in \F_p} (F*f) (x) \_phi(x) \le 2p \| F\|_2 \|\_phi\|_2 \| f\|_2 \,.
	\end{equation}
	\label{l:Frobenious}
\end{lemma}


Finally, we need the classification of subgroups of $\SL_2 (\F_p)$, see \cite{Suzuki}.

\begin{theorem}
	Let $p$ be a prime and $p\ge 5$. Then any subgroup of $\SL_2 (\F_p)$ is isomorphic to one of the following subgroups:\\
	$(1)~$ Finite groups $A_4$, $S_4$, $A_5$. \\
	$(2)~$ The dihedral groups of order 
	$4\left(\frac{p\pm 1}{2}\right)$
	and their subgroups.\\	
	$(3)~$ A Borel subgroup 
	of order 
	$p(p-1)$ 
	and its   subgroups.
	\label{t:classification}
\end{theorem}

We finish this section recalling the celebrated  result of Helfgott \cite{Harald} on the growth in $\SL_2 (\F_p)$.

\begin{theorem}
	Let $A\subseteq \SL_2 (\F_p)$. Assume that $|A|< p^{3-\d}$ for $\d>0$ and $A$ is not contained in any proper subgroup of $\SL_2 (\F_p)$. 
	Then there is a positive function $\kappa (\d) > 0$ such that 
	$$
	|AAA|\gg_\d |A|^{1+\kappa (\d)} \,.
	$$ 
	\label{t:Harald_SL2}
\end{theorem}

\section{First results}
\label{sec:first_results}

Throughout this section $\G$ is a multiplicative subgroup of $\F^*_p$. 
Such subgroups were studied by various authors and 
many deep results about subgroups were obtained, e.g., \cite{Bourgain_DH}, \cite{BC}, \cite{BGK}, \cite{KShp}, \cite{s_Bourgain} and others.
In this section 
we
find  
upper bounds for $\T^{+}_k (f)$, $\E^{+}_k (f)$ 
and for the exponential sums over $f$, where $f$ is an arbitrary $\G$--invariant function, that is, $f(x\gamma) = f(x)$ for all $\gamma \in \G$.
The main difference between our new theorems and results from \cite{s_Bourgain} is, firstly, that we consider general functions $f$ and, secondly, 
the absence of any restrictions on size of support of $f$ (but not on size of $\G$, of course) 
similar to $\R$ where we have no such restrictions, see our previous paper \cite{s_Bourgain}.

We begin with the quantity $\T^{+}_k (f)$ and we use $\T^{+}_2 (f)$ in  bounds below   
to make our results sharper. 
Of course, one can replace this quantity to 
$\|f\|_1^2 \|f\|_2^2$ (see formula (\ref{f:T_f_12'})) of by something even smaller using Lemma \ref{l:AA_small_energy}.

\begin{theorem}
	Let 
	$f$ be a $\G$--invariant 
	complex 
	function with $\sum_{x} f(x) = 0$.
	Then for any $k\ge 1$ 
	one has 
	\begin{equation}\label{f:T_k_G}
	\T^{+}_{2^k} (f) \le
	2^{3k^2} (C_* \log^4 p)^{k-1} \cdot \| f\|_1^{2^{k+1} -4} |\G|^{\frac{(1-k)}{2}} \T^{+}_2 (f) \,,
	\end{equation}
	where $C_*$ is the absolute constant from Lemma \ref{l:AA_small_energy}. 
	\label{t:T_k_G}
\end{theorem}
\begin{proof}
	For $k=1$ bound (\ref{f:T_k_G}) is trivial, so below we will assume that $k\ge 2$. 
	Fix any $s \ge 2$ and put $L=L_s := s\log p$.
	Our aim 
	is to prove 
	\begin{equation}\label{f:T_2s,T_s}
	\T^{+}_{2s} (f) \le 128 C_* L^4_s \| f\|_1^{2s}  \T^{+}_{s} (f) |\G|^{-1/2} \,.
	\end{equation}
	After that we use induction and obtain  
	$$
	\T^{+}_{2^k} (f) \le (128C_*)^{k-1} \log^{4(k-1)} p \cdot 2^{4((k-1)+ (k-2) + \dots+2)}\| f\|_1^{2^{k}+\dots + 4} |\G|^{\frac{-(k-1)}{2}} \T^{+}_2 (f) 
	=
	$$
	$$
	2^{2k^2+5k-11} (C_* \log^4 p)^{k-1} \cdot \| f\|_1^{2^{k+1} -4} |\G|^{\frac{(1-k)}{2}} \T^{+}_2 (f) 
	\le
	2^{3k^2} (C_* \log^4 p)^{k-1} \cdot \| f\|_1^{2^{k+1} -4} |\G|^{\frac{(1-k)}{2}} \T^{+}_2 (f) 
	$$
	and this coincides with (\ref{f:T_k_G}).

	To prove (\ref{f:T_2s,T_s}) we notice that 
	by formula (\ref{f:T_2k_r})
	one has 
	$$
	\T^{+}_{2s} (f) = \sum_{x,y,z} r_{sf} (x) r_{sf} (y) r_{sf} (x+z)  r_{sf} (y+z) \,.
	$$
	Here as usual we have denoted by $r_{sf} (x)$ the function $r_{f+\dots+f} (x)$, where the number of $f$'s in the sum 
	is $s$. 
	We give two upper bounds for $\T^{+}_{2s} (f)$
	and first of all, notice that from the last formula, it follows that $\T^{+}_{2s} (f)$ equals
	$$
	\sigma := |\G|^{-2} \sum_{\gamma_1,\gamma_2 \in \G}\, \sum_{a,b,c,d} r_{sf} (a) r_{sf} (b) r_{sf} (c) r_{sf} (d) \cdot \d\{ a+\gamma_1 b = c+ \gamma_2 d \} 
	$$
	plus the term $\mathcal{E}$ which corresponds to $a,b,c,d$
	equals zero (see below).
	Consider the set of points $\P \subseteq \F^3_p$, each point $p$ indexed by $(\gamma_1,c,d)$ and the set of planes $\Pi \subseteq \F^3_p$ indexed by $(a,b,\gamma_2)$ and each 
	$\pi = \pi_{a,b,\gamma_2} \in \Pi$ has the form 
	$\pi : a+xb=y+\gamma_2 z$. 
	Then in terms of formula (\ref{f:Vinh}) one has $\d \{ a+\gamma_1 b = c+ \gamma_2 d \} = \I (p,\pi)$
	for $p=(\gamma_1,c,d)$ and $\pi = \pi_{a,b,\gamma_2}$. 
	By the assumption $\sum_x f(x)=0$. It follows that $\sum_x r_{sf} (x) = 0$ and hence 
	\begin{equation}\label{tmp:24.10_1}
	\sigma = |\G|^{-2} \sum_{\gamma_1,\gamma_2} f_\G (\gamma_1) f_\G (\gamma_2) \sum_{a,b,c,d} r_{sf} (a) r_{sf} (b) r_{sf} (c) r_{sf} (d) \cdot \d\{ a+\gamma_1 b = c+ \gamma_2 d \} \,,
	\end{equation}
	where $f_{\G} (x) = \G(x) - |\G|/p$ is the balanced function of $\G$. 
	In a similar way, considering for all nonzero $x$ the function $R(x) = |\G|^{-1} \sum_{y} f_\G (y) r_{sf} (x y^{-1})$, we obtain
	$$
	\sigma = |\G|^{-4} \sum_{\gamma_1,\gamma_2, \gamma_3, \gamma_4 } f_\G (\gamma_1) f_\G (\gamma_2) f_\G (\gamma_3) f_\G (\gamma_4) \sum_{a,b,c,d} r_{sf} (a) r_{sf} (b) r_{sf} (c) r_{sf} (d) \cdot \d\{ \gamma_1 a +\gamma_2 b = \gamma_3 c+ \gamma_4 d \} 
	$$	
	\begin{equation}\label{tmp:24.10_1_4}	
	= \sum_{x,y,z} R (x) R (y) R (x+z)  R (y+z) \,. 
	\end{equation}
	Clearly, $R(0) = 0$, further $R(x) = r_{sf} (x)$, $x\neq 0$ and  $\| R\|_\infty = \| r_{sf} \|_\infty$ if one considers the function  $r_{sf}$ as a function on $\F^*_p$ only. 
	Also, notice that $\|f_{\G}\|_1 < 2|\G|$  and hence 
	$$
	\| R\|_1 \le  |\G|^{-1} \| r_{sf} \|_1 \|f_{\G}\|_1 < 2  \| r_{sf} \|_1 \le 2 \| f \|^s_1 \,.
	$$
	
	Now put $\rho = \T^{+}_{2s} (f) /(64 \| f\|_1^{3s})$.
	Since
	$$
	\left| \sum_{x,y,z ~:~ |R (x)| \le \rho} R (x) R (y) R (x+z)  R (y+z) \right| 
	\le 
	8\rho \| f\|_1^{3s}  = \T^{+}_{2s} (f) /8 \,,
	$$
	it follows that 
	$$
	\T^{+}_{2s} (f)
	\le 
	\frac{3}{2} 
	\sum'_{x,y,z} R (x) R (y) R (x+z)  R (y+z) + 
	\frac{3}{2} 
	\mathcal{E} \,,
	$$
	where the sum $\sum'$ above (we denote it as $\T^{'}_{2s} (f)$) is taken over nonzero variables  $x,y,x+z,y+z$ with  $|R (x)|, |R (y)|, |R (x+z)|, |R (y+z)| > \rho$ 
	and by (\ref{f:T_f_12}) 
	$$
	\mathcal{E} \le 4 |r_{sf} (0)| \left| \sum_{y,z} r_{sf} (x) r_{sf} (y) r_{sf} (x+z) r_{sf} (y+z)  \right|
	\le
	4 r_{s|f|} (0) \| f \|^s_1 \T^{+}_s (f) 
	\le
	4 \|f\|^2_2  \| f \|^{2s-2}_1 \T^{+}_s (f)  \,.
	$$
	Let us compare the obtained estimate  for $\mathcal{E}$ with the upper bound in (\ref{f:T_2s,T_s}). 
	By the assumption $f$ is $\G$--invariant function and hence $\| f\|_1 = |\G| \sum_{\xi \in \F_p^*/\G} |f(\xi)|$, 
	as well as 
	\begin{equation}\label{f:L_2-L^2_1}
	\| f\|^2_2 = |\G| \sum_{\xi \in \F_p^*/\G} |f(\xi)|^2 \le |\G|^{-1} \|f \|^2_1 \,.
	\end{equation}
	In particular, 
	\begin{equation}\label{f:E_bound}
	\mathcal{E} \le  4 \| f \|^{2s}_1 \T^{+}_s (f) |\G|^{-1}  
	\le 4 \| f \|^{2s}_1 \T^{+}_s (f) |\G|^{-1/2} \,.
	\end{equation}
	Thus the obtained estimate  for $\mathcal{E}$ is much smaller than the upper bound for $\T^{+}_{2s} (f)$ in (\ref{f:T_2s,T_s}).
	Hence  if (\ref{f:T_2s,T_s}) holds, then there is nothing to prove and in the opposite case, we get
	\begin{equation}\label{f:E_bound'}
	\mathcal{E} \le \T^{+}_{2s} (f) / (32C_* L^4) \,,
	\end{equation}
	so it is negligible.  Also, we can assume that $\T'_{2s} (f) >0$ because otherwise there is nothing to prove.

	Put $P_j = \{ x ~:~ \rho 2^{j-1} < |r_{sf} (x)| \le \rho 2^{j}\} \subseteq \F^*_p$, $j\in \N$.  
	By the H\"older inequality or, alternatively, 
	counting trivial solutions $a_j= a'_j$ to equation (\ref{def:T_k}), 
	we have  $\T^{+}_{s} (f) \ge \|f\|^{2s}_2$. 
	If (\ref{f:T_2s,T_s}) does not hold, then, in particular, 
	\begin{equation}\label{tmp:01.10_2}
	\T^{+}_{2s} (f)\ge 2^7 \| f\|^{2s}_1 \T^{+}_s (f) |\G|^{-1/2} \ge 2^7 \| f\|^{2s}_1 \|f\|^{2s}_2 |\G|^{-1/2}
	\end{equation} 
	and hence the possible number of the sets $P_j$ does not exceed $L$.  
	Indeed, for any $x$ one has $|r_{s\G} (x)| \le \| f\|_1^{s-2} \|f\|^2_2$ and hence 
	$\rho 2^{j-1} = 2^{j-7} \T^{+}_{2s} (f) \| f\|_1^{-3s}$ must be less than $\| f\|_1^{s-2} \|f\|^2_2$ otherwise the correspondent set $P_j$ is empty.
	In other words,
	using the H\"older inequality one more time,  
	as well as 
	bound (\ref{tmp:01.10_2}), we obtain
	$$
	2^{j-7} \le \| f\|_1^{4s-2} \|f\|^2_2 / \T^{+}_{2s} (f) \le \| f\|_1^{2s-2} \|f\|^{-(2s-2)}_2 |\G|^{1/2} /2^7
	\le
	$$
	$$
	\le p^{s-1} |\G|^{1/2} /2^7 < p^{s} /2^7
	$$
	as required. 
	By the Dirichlet principle there is $\D = \rho 2^{j_0-1}$, 
	and a set $P=P_{j_0}$ such that 
	\begin{equation}\label{tmp:15.01_1}
	\T^{+}_{2s} (f) 
	\le 
	\frac{3}{2} 
	L^4 (2\D)^4 \E^{+} (P) + \frac{3}{2} 
	\mathcal{E} 
	= \T'_{2s} (f) + 
	\frac{3}{2} 
	\mathcal{E} \,.
	\end{equation}
	Indeed, putting $g_i (x) = P_i (x) r_{sf} (x)$, and using (\ref{f:E_Ho}), we get 
	$$
	\sum'_{x,y,z} r_{sf} (x) r_{sf} (y) r_{sf} (x+z) r_{sf} (y+z) 
	\le
	\sum_{i,j,k,l=1}^L\, \sum_{x,y,z} g_i (x) g_j (y) g_k (x+z)  g_l (y+z)
	\le 
	$$
	$$
	\le
	\sum_{i,j,k,l=1}^L (\E^{+} (g_i) \E^{+} (g_j) \E^{+} (g_k) \E^{+} (g_l) )^{1/4} 
	=
	\left( \sum_{i=1}^L (\E^{+} (g_i))^{1/4} \right)^4
	\le
	$$
	$$
	\le
	L^3  \sum_{i=1}^L \E^{+} (g_i) \le L^4 \max_i \E^{+} (g_i) \,. 
	$$ 
	Certainly, the sum  $\sum'_{x,y,z} R (x) R (y) R (x+z) R (y+z)$  can be estimated in a similar way and one can check that 
	all functions  $R_i (x) = |\G|^{-1} \sum_{y} f_\G (y) r_i (x y^{-1})$ have zero mean and $\|R_i\|_\infty \le \| r_i \|_\infty$. 
	Moreover we always have $|P| \D^2 \le \T^{+}_{s} (f)$ and 
	\begin{equation}\label{f:P_D_T_k}
	|P| \D \le \sum_{x\in P} |r_{sf} (x)| \le \sum_x |r_{sf} (x)| \le \sum_x r_{s|f|} (x) = \| f\|_1^s \,.
	\end{equation}	
	Using Lemma \ref{l:AA_small_energy}, we obtain
	$$
	\E^{+} (P) 
	\le
	C_* \left( \frac{|P|^4}{p} + \frac{|P|^3}{|\G|^{1/2}} \right) \,.
	$$
	Hence
	\begin{equation}\label{tmp:30.03_1}
	\T'_{2s} (f)
	\le 
	3 (16 C_*) L^4 \left( \frac{\D^4 |P|^4}{p} + \frac{\D^4 |P|^3}{|\G|^{1/2}} \right) \,.
	\end{equation}
	Suppose that 
	the second term in (\ref{tmp:30.03_1}) dominates. 
	Then in view of $|P| \D^2 \le \T^{+}_{s} (\G)$ and $|P| \D \le \| f\|_1^s$, we have 
	$$
	|P|^3 \D^4 = (P \D)^2 P \D^2 \le  \| f\|_1^{2s}  \T^{+}_{s} (f) \,.
	$$
	In other words, the second term in (\ref{tmp:30.03_1}) does not exceed 
	\begin{equation}\label{tmp:01.10_1}
	3(16 C_*) L^4 
	\| f\|_1^{2s}  \T^{+}_{s} (f) |\G|^{-1/2} \,.
	\end{equation}
	and inequality  (\ref{f:T_2s,T_s}) (also, recall bound (\ref{f:E_bound})) is proved.

	If the first term in (\ref{tmp:30.03_1}) dominates, then we notice that $|P| |\G|^{1/2} \ge p$ and use another bound.
	By $\sum_x f_\G (x) = 0$, formulae (\ref{tmp:24.10_1}),  (\ref{tmp:15.01_1}) and (\ref{f:Vinh}),  
	as well as 
	the last estimate, we have
	\begin{equation}\label{f:first_T}
	\T^{'}_{2s} (f) \le 3\cdot 8 L^4 p (4 |P| \D^2)^2 |\G|^{-1} = 3\cdot 2^7 |P|^3 \D^4 |\G|^{-1}  \le 3\cdot  16 |P|^3 \D^4 |\G|^{-1/2} 
	\end{equation}
	for $|\G| \ge 2^6$. 
	Of course 
	quantity (\ref{f:first_T}) 
	is less than the second term in  (\ref{tmp:30.03_1}). 
	If $|\G| < 2^6$, then it is easy to check that (\ref{f:T_k_G}) takes place. 
	Combining 
	the obtained bound (\ref{f:first_T}) 
	with (\ref{tmp:01.10_1}), we see that in any case 
	$$
	\T'_{2s} (f) \le 
	6(16 C_*) L^4 \| f\|_1^{2s}  \T^{+}_{s} (f) |\G|^{-1/2} \,.
	$$
	Finally, using (\ref{f:E_bound}), we have
	$$
	\T^{+}_{2s} (f) \le 128 C_* L^4 \| f\|_1^{2s}  \T^{+}_{s} (f) |\G|^{-1/2}
	$$
	This completes the proof. 
	$\hfill\Box$
\end{proof}

\bigskip 

Now we are ready to obtain an upper bound for the exponential sums over any $\G$--invariant function $f$.

\begin{corollary}
	Let 
	$\G \subseteq \F_p^*$  be a multiplicative subgroup, $|\G| \ge p^\d$, $\d>0$, and let  
	$f$ be a $\G$--invariant 
	complex 
	function with $\sum_{x} f(x) = 0$.
	Then for all sufficiently large $p$ one has 
	\begin{equation}\label{f:exp_sums}
	\max_{\xi} |\FF{f} (\xi)| \ll \| f\|_1 \cdot p^{-\frac{5\d}{2^{7+ 2\d^{-1}}}} \,.
	\end{equation}
	Further we have a nontrivial upper bound $o(\| f\|_1)$ for the maximum in (\ref{f:exp_sums})  if 
	\begin{equation}\label{f:exp_sums'}	
	\log |\G| \ge \frac{C\log p}{\log \log p} \,,
	\end{equation}
	where $C>2$ is any constant.
	\label{c:exp_sums}
\end{corollary}
\begin{proof}
	By $\rho$ denote the maximum in (\ref{f:exp_sums}).
	It is attained at some nonzero $\xi$ because $\sum_{x} f(x) = 0$.
	Then by Theorem \ref{t:T_k_G}, a trivial bound which follows from (\ref{f:T_f_12'}), namely, $\T^{+}_2 (f) \le \| f\|^2_1 \|f\|^2_2$ 
	and formula (\ref{def:T_k}), we obtain 
	\begin{equation}\label{tmp:31.03_1}
	|\G| \rho^{2^{k+1}} \le p \T^{+}_{2^k} (f) \le  
	p 2^{3k^2} (C_* \log^4 p)^{k-1} \cdot \| f\|_1^{2^{k+1} -2} \| f\|^2_2 |\G|^{\frac{(1-k)}{2}} \,.
	\end{equation}
	Using formula (\ref{f:L_2-L^2_1}), we get
	$$
	|\G| \rho^{2^{k+1}} \le p 2^{3k^2} (C_* \log^4 p)^{k-1} \cdot \| f\|_1^{2^{k+1}} |\G|^{-\frac{(k+3)}{2}} \,.
	$$
	Put $k = \lceil 2\log p /\log |\G| + 4 \rceil \le 2/\d +5$, say.
	Also, notice that 
	\begin{equation}\label{tmp:k_choice}
	\frac{p \log^{4(k-1)} p }{|\G|^{k/2}} 
	\le 
	1
	\end{equation}
	because $k\ge 2\log p /\log |\G| + 4$ and $p$ is a sufficiently large number depending on $\delta$ (the choice of $k$ is slightly larger than $2\log p /\log |\G|$ to "kill"\, $p$ by division by $|\G|^{k/2}$, as well as  logarithms $\log^{4(k-1)} p$). 	
	Taking  a power $1/2^{k+1}$ from both parts of (\ref{tmp:31.03_1}), we see 
	in view of (\ref{tmp:k_choice}) 
	that 
	$$
	\rho \ll \| f\|_1 \cdot |\G|^{-\frac{5}{2^{k+2}}} 
	\ll
	\| f\|_1 \cdot p^{-\frac{5\d}{2^{7+ 2\d^{-1}}}} \,.
	$$
	To prove the second part of our corollary just notice that the same choice of $k$ gives something nontrivial if
	$2^k \ll  \eps \log |\G|$ 
	for any $\eps>0$. 
	In other words, it is enough to have 
	$$
	k \le \frac{2\log p}{\log |\G|} + 5  \le \log \log |\G| - \log (1/\eps) \,.
	$$
	It means that the inequality  $\log |\G| \ge C \log p /(\log \log p)$ for any $C>2$ is enough. 
	This completes the proof. 
	$\hfill\Box$
\end{proof}

\bigskip

Let us obtain a new general bound for $\E^{+}_k (f)$.

\begin{theorem}
	Let 
	$f$ be a $\G$--invariant function 
	real 
	function with $\sum_{x} f(x) = 0$.
	Then for positive integer $k\ge 2$  
	either 
	$$
	\E^{+}_{2^{k+1}} (f) 
	\le
	32 C^{1/4}_* (1+\log (\| f\|_1 \|f\|^{-1}_2))  \| f\|^{2^{k+1}}_2 \E^{+}_{2^k} (f) |\G|^{-1/8} 
	$$
	or
	$$
	\E^{+}_{2^{k+1}} (f) 
	\le
	2 \| f\|^{2^{k+2}}_2 \,.
	$$
	Here $C_*$ is the absolute constant from Lemma \ref{l:AA_small_energy}. 
	In particular, if $k$ is chosen as 
	\begin{equation}\label{cond:k_Q_shift}
	|\G|^{\frac{k-1}{8}} \ge  (32 C^{1/4}_* (1+\log (\| f\|_1 \|f\|^{-1}_2)))^{k-1} \|f\|_1^2 \|f\|^{-2}_2 \,,
	\end{equation}
	then $\E^{+}_{2^{k+1}} (f) \le 2 \| f\|^{2^{k+2}}_2$. 
	\label{t:Q_shift}
\end{theorem}
\begin{proof}
	Fix an even integer $l\ge 1$ and prove that either 
	\begin{equation}\label{f:iteration_5.2}
	\E^{+}_{4l} (f) 
	\le 
	32 C^{1/4}_* (1+\log  (\| f\|_1 \|f\|^{-1}_2))  \| f\|^{4l}_2 \E^{+}_{2l} (f) |\G|^{-1/8} 
	\end{equation}
	or
	\begin{equation}\label{f:iteration_5.2'}
	\E^{+}_{4l} (f) \le 2\|f\|_2^{8l} \,.
	\end{equation}
	After that it requires just to use induction to see
	$$
	\E^{+}_{2^{k+1}} (f) 
	\le
	(32 C^{1/4}_* (1+\log (\| f\|_1 \|f\|^{-1}_2))^{k-1}  \| f\|^{2^{k+1} + \dots + 2^{3}}_2 \E^{+}_{4} (f) |\G|^{-(k-1)/8} 
	\le
	$$
	$$
	\le
	(32 C^{1/4}_* (1+\log (\| f\|_1 \|f\|^{-1}_2))^{k-1}  \| f\|^{2^{k+2} - 8}_2 \E^{+}_{4} (f) |\G|^{-(k-1)/8} \,.
	$$
	Trivially, $\E^{+}_{4} (f) \le \| f\|^{6}_2 \|f\|^2_1$ and hence 
	$$
	\| f\|^{2^{k+2}}_{2}
	\le
	\E^{+}_{2^{k+1}} (f) 
	\le
	(32 C^{1/4}_* (1+\log (\| f\|_1 \|f\|^{-1}_2))^{k-1}  \| f\|^{2^{k+2} - 2}_2 \|f\|^2_1 |\G|^{-(k-1)/8} \,.
	$$	
	Thus if 
	$$
	|\G|^{\frac{k-1}{8}} \ge  (32 C^{1/4}_* (1+ \log (\| f\|_1 \|f\|^{-1}_2))^{k-1} \|f\|_1^2 \|f\|^{-2}_2 \,,
	$$
	then, clearly, $\E^{+}_{2^{k+1}} (f) \le 2 \| f\|^{2^{k+2}}_2$.

	We give two upper bounds for $\E^{+}_{4l} (f)$. 
	Firstly, 
	let us remark that for any positive integer $n$
	there exists a function $F$ such that 
	\begin{equation}\label{f:F_def}
	r^{n}_{f-f} (x) = r_{F-F} (x) \,.
	\end{equation}	
	Indeed, 
	from the definition of the 
	required 
	function $F$ and formula (\ref{f:E_k_Fourier}) one has 
	$$
	|\FF{F} (x)|^{2} = p^{1-n} \sum_{y_1 +\dots + y_{n} = x} 
	|\FF{f} (y_1)|^2 \dots |\FF{f} (y_{n})|^2 \ge 0 
	$$
	and we can choose the Fourier transform of $F$ taking, say,  a positive square root of the left--hand side of the previous formula. 
	It defines our function $F$ (but not uniquely, even in the case $n=1$ one can take $F(x) = f(x)$ or $F(x) = f(-x)$, say). 
	In particular, by the Parseval identity, we get 
	\begin{equation}\label{tmp:28.04.2017_1}
	\| F \|_2^2 = p^{-1} \| \FF{F} \|_2^2 
	= p^{-n} \sum_x \sum_{y_1 +\dots + y_n = x} 
	|\FF{f} (y_1)|^2 \dots |\FF{f} (y_n)|^2
	= \|f\|_2^{2n} \,.
	\end{equation}
	To obtain another proof of the last 
	equality 
	just substitute $x=0$ into (\ref{f:F_def}) and notice that $(F\circ F) (0) = \| F\|^2_2  = r^n_{f-f} (0) = (f \circ f)^n (0) = \|f\|_2^{2n}$. 
	Applying these arguments for $n=4l-1$, we obtain
	$$
	\E^{+}_{4l} (f) = \sum_x r_{f-f} (x) r^{4l-1}_{f-f} (x) = \sum_x r_{f-f} (x) r_{F-F} (x) \,.
	$$
	By the assumption the function $f$ is $\G$--invariant.
	Thus 
	$$
	\E^{+}_{4l} (f) 
	= |\G|^{-2} \sum_{\gamma_1,\gamma_2 \in \G}\, \sum_{a,b,c,d} F(a) f(b) F (c) f (d) \cdot \d\{ a+\gamma_1 b = c+ \gamma_2 d \} 
	\,.
	$$ 
	Consider the set of points $\P \subseteq \F^3_p$, each point $p$ indexed by $(\gamma_1,c,d)$ and the set of planes $\Pi \subseteq \F^3_p$ indexed by $(a,b,\gamma_2)$ and each $\pi \in \Pi$ has the form 
	$\pi : a+xb=y+\gamma_2 z$. 
	Then we have as in Theorem  \ref{t:T_k_G} that $\d \{ a+\gamma_1 b = c+ \gamma_2 d \} = \I (p,\pi)$. 
	By the assumption $\sum_x f(x)=0$. 
	Besides $\| F \|^2_2 = \|f\|_2^{8l-2}$.  
	Hence by (\ref{f:Vinh}), we have
	\begin{equation}\label{f:first_E}
	\E^{+}_{4l} (f) 
	\le p \|f\|_2^{8l} |\G|^{-1} \,.
	\end{equation}

	Now let us give another bound for $\E^{+}_{4l} (f)$. 
	Put $g(x)  = r^l_{f-f} (x)$, $L = 2 + 2\log (\| f\|_1 \|f\|^{-1}_2)$
	and $\E'_{4l} (f) = \E^{+}_{4l} (f) - \|f\|_2^{8l}$. 
	We will assume below that $\E'_{4l} (f) \ge 2^{-1} \E^{+}_{4l} (f) > 0$ because otherwise 
	the required inequality 
	(\ref{f:iteration_5.2'}) 
	follows 
	immediately. 
	Similarly, we can assume that $\E'_{4l} (f) \ge \| f \|^{8l}_2$  because otherwise $\E^{+}_{4l} (f) \le 2\| f \|^{8l}_2$ and we are done. 
	Further put $\rho^{4-1/l} = 2^{-1} \| f\|^{8l}_2 \| f\|^{-2}_1$ and $P_j = \{ x ~:~ \rho 2^{j-1} < g(x) \le \rho 2^j \}$. 
	Clearly, 
	$$\sum_{x ~:~ g(x) < \rho} g^4 (x) < \rho^{4-1/l} \| f\|^2_1 = 2^{-1} \| f\|^{8l}_2 \,.$$
	Thus the number of the sets $P_j$ does not exceed $L$. 
	Indeed, for any $x$ one has $g(x) \le \|f\|^{2l}_2$ and hence $2^{j-1} \rho$ must be less than  $\|f\|^{2l}_2$ because  otherwise $P_j$ is empty. 
	Whence for $j \ge 3$ 
	$$
	2^{j-2} \| f\|^{8l}_2 \| f\|^{-2}_1 \le 2^{(j-1)(4-1/l)-1}  \| f\|^{8l}_2 \| f\|^{-2}_1 = 2^{(j-1)(4-1/l)} \rho^{4-1/l} \le \|f\|^{8l-2}_2 
	$$
	and $j\le 2+2\log (\| f\|_1 \|f\|^{-1}_2) := L$ (if $j< 3$, then the last bound holds trivially). 
	Notice that $\log (\| f\|_1 \|f\|^{-1}_2) \ge 0$. 
	Using the Dirichlet principle, we find a set $P=P_{j_0}$ and a positive number $\D=\rho 2^{j_0-1}$ such that $P = \{ x ~:~ \D < g(x) \le 2\D \} \subseteq \F_p^*$ and 
	\begin{equation}\label{tmp:02.10_1}
	\E'_{4l} (f) \le 2L \sum_{x \in P} r^{4l}_{f-f} (x)  \le 2L  \|f \|^{3l}_2 \sum_{x \in P} r^{5l/2}_{f-f} (x) := \|f \|^{3l}_2 \sigma \,. 
	\end{equation}
	Applying Lemma \ref{l:change_QG}, combining with Lemma \ref{l:AA_small_energy}, we obtain	
	\begin{equation*}\label{tmp:30.04_1}
	\sigma \le 2L (2\D)^{3/2} \sum_{x \in P} r^{l}_{f-f} (x) 
	\le
	2^{5/2} C^{1/4}_* L \D^{3/2} \| f\|^{l}_2 (\E^{+}_{2l} (f))^{1/4} \left( \frac{|P|^4}{p} + \frac{|P|^3}{|\G|^{1/2}}\right)^{1/4} 
	\le 
	\end{equation*}
	\begin{equation}\label{tmp:01.10_3} 
	\le
	2^{5/2} C^{1/4}_* L  \| f\|^{l}_2 (\E^{+}_{2l} (f))^{1/4} \left( \frac{\D^6 |P|^4}{p} + \frac{\D^6 |P|^3}{|\G|^{1/2}}\right)^{1/4} \,.
	\end{equation}
	Suppose that the second term in (\ref{tmp:01.10_3}) dominates. 
	Since $l$ is an even number, we have $\D |P| \le \E^{+}_{l} (f)$, $\D^2 |P| \le \E^{+}_{2l} (f)$ and hence 
	$\D^6 |P|^3 \le (\E^{+}_{2l} (f))^3$. 
	It follows that
	$$
	2^{-1} \E^{+}_{4l} (f)  \le 
	\E'_{4l} (f) 
	\le
	8 C^{1/4}_* L  \| f\|^{4l}_2 \E^{+}_{2l} (f) |\G|^{-1/8} 
	$$
	and we obtain (\ref{f:iteration_5.2}). 
	Now if the first term in (\ref{tmp:01.10_3}) dominates, then $|P| |\G|^{1/2} \ge p$ and returning to (\ref{tmp:02.10_1}), we have
	$$
	\E'_{4l} (f) \le 32L \D^4 |P| \,.
	$$
	Multiplying this inequality by $|P|$ and using $\D^2 |P| \le \E^{+}_{2l} (f)$, we get
	$$
	|P| \E'_{4l} (f) \le 32L (\E^{+}_{2l} (f))^2 \,.
	$$ 
	Recalling (\ref{f:first_E}), 
	applying the inequality 
	$|P| |\G|^{1/2} \ge p$ and the last bound, we obtain 
	$$
	\E^{+}_{4l} (f) 
	\le
	p  \|f\|_2^{8l} |\G|^{-1}
	\le
	|P| \|f\|_2^{8l} |\G|^{-1/2} 
	\le
	32L \|f\|_2^{8l} (\E^{+}_{2l} (f))^2 |\G|^{-1/2} (\E'_{4l} (f))^{-1} \,.
	$$
	Whence 
	in view of the inequality $2^{-1} \E^{+}_{4l} (f)  \le \E'_{4l} (f)$, we have 
	$$
	\E^{+}_{4l} (f) \le 8L^{1/2} \|f\|_2^{4l} \E^{+}_{2l} (f) |\G|^{-1/4} \,.
	$$
	Thus we see that the required inequality (\ref{f:iteration_5.2}) takes place in any case. 
	This completes the proof. 
	$\hfill\Box$
\end{proof}


\begin{remark}
	The upper bound in Theorem \ref{t:Q_shift} is optimal. 
	Indeed, let $\chi(x)$ be the Legendre symbol.
	In other words, if $R$ is the set of quadratic residues and $\chi_0 (x)$ is the trivial character, then 
	$\chi(x) = 2R(x)-\chi_0 (x)$. Let $\G\subseteq R$ be a multiplicative subgroup. 
	Then $\chi (x)$ is a real $\G$--invariant function and
	$\sum_x \chi(x) = 0$. 
	By standard formulas for characters see, e.g., \cite{BEW_book} one has for any $k\ge 2$ that 
	$\E^{+}_k (\chi) = \sum_x (\chi \circ \chi)^k (x) = (p-1)^k + (p-1) (-1)^k \sim p^k \sim \| \chi\|^{2k}_2$. 
\end{remark}


\begin{remark}
	\label{c:max_f_E_k}
	Let $f$ be a real $\G$--invariant function  with zero mean and let $T=\|f \|_1 / \| f\|_2$. By (\ref{f:L_2-L^2_1}) we have $T\ge |\G|^{1/2}$. 
	Choosing an integer  $k = C\log T/\log |\G|$ with sufficiently large constant $C>0$,  we satisfy condition (\ref{cond:k_Q_shift}), provided $\log T \ll |\G|$. 
	Under this condition, we get 
	$$
	|\G| \left( \max_{x \neq 0} |(f\circ f)(x)| \right)^{2^{k+1}} \le 2 \| f\|^{2^{k+2}} \,.
	$$
	It follows that
	\begin{equation}\label{f:max_fcf}
	\max_{x \neq 0} |(f\circ f)(x)| \le \| f\|_2^2 \cdot (2 |\G|^{-1})^{1/2^{k+1}} \,.
	\end{equation}
	Thus we have obtained a non--trivial upper bound for the quantity $\max_{x \neq 0} |(f\circ f)(x)|$ if the condition
	$$
	\log |\G| \gg \frac{\log T}{\log \log T}
	$$
	holds. Of course, we last bound implies $\log T \ll |\G|$. 
	Some applications of such sort of bounds can be found in papers \cite{Bourgain_more}, \cite{s_Bourgain}.
\end{remark}

\begin{corollary}
	Let 
	$\G \subseteq \F_p^*$  be a multiplicative subgroup, $|\G| \ge p^\d$ and 
	$f$ be a $\G$--invariant  
	real 
	function with $\sum_{x} f(x) = 0$.
	Also, let $k$ is chosen as 
	\begin{equation}\label{f:G+1}
	|\G|^{\frac{k-1}{8}} \ge  (32 C^{1/4}_* (1+\log (\| f\|_1 \|f\|^{-1}_2)))^{k-1} \|f\|_1^2 \|f\|^{-2}_2 \,,
	\end{equation}
	and $s=\lceil 2\log \| f\|_1 / \log (|\G|/2) \rceil$.
	Then 
	$$
	\E^{\times}_{2^{k+s+1}+1} (f+1) \le 3 \| f\|^{2^{k+s+2}+2}_2\,.
	$$
	\label{c:G+1}
\end{corollary}
\begin{proof}
	For any integer $l$, we have
	$$
	\E^{\times}_{l} (f+1) = \| f\|^{2l}_2 + |f^{2l}(-1)| + \sum_{x \neq 0,1} r^{l}_{(f+1)/(f+1)} (x) 
	\le
	$$
	$$
	\le 
	2 \| f\|^{2l}_2 + \sum_{x \neq 0,1} r^{l}_{(f+1)/(f+1)} (x)  \,.
	$$
	Here we have used that $\sum_{x} f(x) = 0$.
	Further for any $\a \neq 0,1$ put $f^\a (x) = f(\a^{-1} x)$. 
	Then 
	$$
	r_{(f+1)/(f+1)} (x) = r_{f-f^x} (x-1) \,.
	$$
	Take $k$ such that 
	$$
	|\G|^{\frac{k-1}{8}} \ge  (32 C^{1/4}_* (1+\log (\| f\|_1 \|f\|^{-1}_2)))^{k-1} \|f\|_1^2 \|f\|^{-2}_2 \,.
	$$
	As in Corollary \ref{c:max_f_E_k}, we have for any $x\neq 0,1$ 
	$$
	|\G| \left(\max_{y\neq 0} |(f^x \circ f)(y)| \right)^{2^{k+1}} \le \sum_y r^{2^{k+1}}_{f-f^x} (y)
	\le
	(\E^{+}_{2^{k+1}} (f^x) \E^{+}_{2^{k+1}} (f))^{1/2} = \E^{+}_{2^{k+1}} (f) \,.
	$$
	Here we have used the Cauchy--Schwarz inequality. 
	Applying Theorem \ref{t:Q_shift}, we obtain
	$$
	\max_{y\neq 0} |(f^x \circ f)(y)| \le \| f\|_2^2 \cdot (2 |\G|^{-1})^{1/2^{k+1}} \,.
	$$
	Thus for any $s$, we have 
	$$
	\E^{\times}_{2^{k+1+s}+1} (f+1) \le 2 \| f\|^{2^{k+2+s}+2}_2 +  \| f\|^{2^{k+2+s}}_2 \| f\|_1^2 (2/|\G|)^{s}
	=
	\| f\|^{2^{k+2+s}+2}_2 (2  + \| f\|^2_1 (2/|\G|)^{s} ) \,.
	$$
	Taking $s$ such that 
	$$
	(|\G|/2)^s \ge \| f \|^2_1 \,,
	$$
	or, in other words, $s\ge 2\log \| f\|_1 / \log (|\G|/2)$, 
	we obtain the required result.
	This completes the proof. 
	$\hfill\Box$
\end{proof}

\bigskip

For example, let $f(x) = Q(x) - |Q|/p$, where $Q$ is any $\G$--invariant set, $\log |Q| \ll |\G|$. 
Then $k \sim \log |Q|/\log |\G|$ and $s \sim \log |Q|/\log |\G|$, so we have the same bound as in Theorem \ref{t:Q_shift} for more or less the same order 
$l \sim k,s$ 
of the 
energy $\E_l (f)$.

\bigskip 

The next result shows that smallness of the energy $\E^{+}_k$ allows to obtain upper bounds for sums of types (\ref{f:E_k_sigma}) and (\ref{f:E_k_sigma+}).
The arguments of the proof are rather general. 
Estimate (\ref{f:E_k_sigma+}) allows to give an alternative proof of formula   
(\ref{f:max_fcf}). 
Also, putting $s=2^k$ and $B=P$ in  formula (\ref{f:E_k_sigma+}) one can derive 
Lemma \ref{l:change_QG}. 

\begin{corollary}
	Let 
	$\G \subseteq \F_p^*$  be a multiplicative subgroup and 
	$f$ be a $\G$--invariant function 
	real 
	function with $\sum_{x} f(x) = 0$.
	If $k$ is chosen as 
	\begin{equation}\label{cond:k_Q_shift_E_k}
	|\G|^{\frac{k-1}{8}} \ge  (32 C^{1/4}_* (1+\log (\| f\|_1 \|f\|^{-1}_2)))^{k-1} \|f\|_1^2 \|f\|^{-2}_2 \,,
	\end{equation}
	then for any set $B\subseteq \F_p$ one has
	\begin{equation}\label{f:E_k_sigma}
	p^{-1} \sum_x |\FF{f} (x)|^2 r_{B-B} (x) \le \| f\|_2^{2} |B| \left( \frac{2\E^{+} (B)}{p|B|^2} \right)^{1/2^{k+1}} 
	\end{equation}
	and for any function $g : \F_p \to \C$ 
	and a positive integer  $s \le 2^k$  
	the following holds 
	\begin{equation}\label{f:E_k_sigma+}
	\left| \sum_{x\in B} (g  \circ f)^s  (x) \right| \le |B| \|g\|^s_2 \|f\|^s_2  \left( \frac{2\E^{+} (B)}{|B|^4}\right)^{s/2^{k+2}} \,.
	\end{equation}
	\label{c:E_k_sigma}
\end{corollary} 
\begin{proof}
	Denote by $\sigma$ the sum from (\ref{f:E_k_sigma}) and put $\mu(x) = p^{-1} |\FF{f} (x)|^2$.
	Clearly,
	$$
	\sigma = \sum_{x\in B} (\mu * B)(x) \,.
	$$
	Then using the H\"older inequality, we obtain
	$$
	\sigma^{2^k} \le |B|^{2^k-1} \sum_{x\in B} (\mu *_{2^k} \mu * B)(x) = |B|^{2^k-1} \sum_{x} (\mu *_{2^k} \mu )(x) (B \circ B) (x) \,.
	$$
	Applying the H\"older inequality one more time,  as well as  formula (\ref{f:E_k_Fourier}), we get
	$$
	\sigma^{2^{k+1}} \le |B|^{2^{k+1}-2} \E^{+} (B) \T^{+}_{2^k} (\mu) = p^{-1} |B|^{2^{k+1}-2} \E^{+} (B) \E^{+}_{2^{k+1}} (f) \,.
	$$ 
	By our choice of the parameter $k$ and Theorem \ref{t:Q_shift}, we have $\E^{+}_{2^{k+1}} (f) \le 2 \| f\|_2^{2^{k+2}}$ and hence 
	$$
	\sigma \le |B| \| f\|_2^{2} \left( \frac{2\E^{+} (B)}{p|B|^2} \right)^{1/2^{k+1}} \,.
	$$ 
	
	It remains to prove (\ref{f:E_k_sigma+}). 
	Using the H\"older inequality twice,  one has 
	$$
	\left( \sum_{x\in B} (g  \circ f)^s (x) \right)^{2^{k+2} /s}
	\le
	|B|^{2^{k+2}/s-4} \left( \sum_{x\in B} (g  \circ f )^{2^{k}} (x) \right)^{4}
	=
	$$
	$$	
	=
	|B|^{2^{k+2}/s-4} \left( \sum_{y_1,\dots, y_{2^{k}}} g(y_1) \dots g(y_{2^{k}}) \sum_{x \in B} f(y_1+x) \dots f(y_{2^{k}}+x) \right)^{4}
	\le
	$$
	$$
	\le
	|B|^{2^{k+2}/s-4} \| g\|^{2^{k+2}}_2  \left( \sum_x (f\c f)^{2^{k}} (x) (B\c B) (x)  \right)^{2}
	\le
	|B|^{2^{k+2}/s-4} \| g\|^{2^{k+2}}_2 \E_{2^{k+1}} (f) \E^{+} (B) \,.
	$$
	By our choice of the parameter $k$ and Theorem \ref{t:Q_shift}, we have $\E^{+}_{2^{k+1}} (f) \le 2 \| f\|_2^{2^{k+2}}$.
	Hence
	$$
	\left| \sum_{x\in B} (g  \circ f)^s (x) \right| \le  \|g\|^s_2 \|f\|^s_2 |B| \left( \frac{2\E^{+} (B)}{|B|^4}\right)^{s/2^{k+2}} \,.
	$$
	as required. 
	$\hfill\Box$
\end{proof}

\bigskip 

Estimate (\ref{f:E_k_sigma+}) shows that the smallness of $\E_k (A)$ energy implies that the sums from this inequality are small. 
It is easy to see that the reverse direction takes place as well. 
Indeed, suppose  in contrary that $\E_{n} (A) \ge  M |A|^n$ for a parameter $M\ge 1$ and  
for all positive integers $n$. 
Also, let $l$ be a positive integer and 
let $P = P_l$ be a set as in the proof of Theorem \ref{t:Q_shift} such that $\E_{l+1} (A) \sim |P| \D^{l+1}$ and $\Delta < r_{A-A} (x) \le 2\D$ on $P$. 
Then $|P| \gtrsim M$, and  using our assumption (let $s=1$ for simplicity)
$$
\E_{l+1} (A) \lesssim \D^{l} \sigma_P (A) \ll \D^{l} |A| |P| |P|^{-\eps} \lesssim |A| \E_l (A) M^{-\eps} \,,
$$ 
where $\eps>0$ is a constant from our assumption. 
So, after $t$ applications of this argument, we get 
$$
M |A|^{l+t} \le \E_{l+t} (A) \lesssim |A|^t \E_l (A) M^{-\eps t} \le |A|^{t+l+1} M^{-\eps t} 
$$
and hence after $t$ steps such that $M^{\eps t +1} \gtrsim |A|$ we obtain a contradiction and it means that, in particular, 
$\E_{t+1} (A) \le M|A|^t$.
For example, if $M=|A|^\delta$, then  $\E_{t+1} (A) \le |A|^{t+\delta}$ for
$t\gg \frac{1}{\eps \d}$, say.

\section{On some sum--product quantities with six and eight variables}
\label{sec:proof}

For any set $A\subseteq \F_p$  let 
\begin{equation}
\Do^\times (A) = \Do^\times_2 (A) 
:= |\{ (a_1-a_2) (a_3-a_4) = (a'_1-a'_2) (a'_3-a'_4) ~:~ a_i,\, a'_i\in A \}| \,,
\label{def:D_times}
\end{equation}
and more generally for $k\ge 1$ 
$$
\Do^\times_k (A) := |\{ (a_1-a_2) \dots (a_{2k-1}-a_{2k}) = (a'_1-a'_2) \dots (a'_{2k-1}-a'_{2k}) ~:~ a_i,\, a'_i\in A \}| \,.
$$
Clearly, $\Do^\times_1 (A) = \E^{+} (A)$. 
Sometimes, we need $\Do^\times_k (A,B)$ for two sets $A,B$ and even more generally $\Do^\times_k (\a,\beta)$ for two functions  $\a,\beta$.

Our task is to 
estimate the quantities $\Do^\times (A)$, $\Do^\times_k (A)$.  
The quantity  $\Do^\times (A)$ (and similar $\Do^\times_k (A)$) can be interpreted as the number of incidences between points and planes 
(see details in \cite{AMRS})
\begin{equation}\label{f:D_pp}
(a_1-a_2) \la = (a'_1-a'_2) \mu \,,
\end{equation}
counting with the weights $|\{ a_3-a_4=\la ~:~ a_3,a_4\in A\}|$ and $|\{ a'_3-a'_4=\mu ~:~ a'_3,a'_4\in A\}|$.

\begin{theorem}
	Let $A \subseteq \F_p$ be a set. Then
	\begin{equation}\label{f:D_times_2-}
	\Do^\times (A) -  \frac{|A|^8}{p} \ll (\log |A|)^2 |A|^5 (\E^{+} (A))^{1/2} \,.
	\end{equation}
	Moreover, for all $k\ge 2$ one has 
	\begin{equation}\label{f:D_times_2}
	\Do^\times_k (A) -  \frac{|A|^{4k}}{p} \ll (\log |A|)^4 |A|^{4k-2-2^{-k+2}} \E^+ (A)^{1/2^{k-1}} \,.
	\end{equation}		
	Generally, for any non--negative function $\a$ and $\beta (x) = A(x)$, 
	the following holds 
	\begin{equation}\label{f:D_times_3}
	\Do^\times_k (\a, \beta) -  \frac{\|\a\|^{2k}_1 \| \beta \|^{2k}_1}{p} 
	\ll 
	L^8 
	(\| \a\|_1 \| \beta \|_1)^{2k-2}
	(\| \a\|_2 \| \beta \|_2)^{2-2^{-k+2}}
	\E^+ (\a,\beta)^{1/2^{k-1}} \,,
	\end{equation}	
	where 	$L := \log (\|\a\|_1 \|\beta\|_1 |A| / (\|\a\|_2 \|\beta\|_2 ))$. 
	\label{t:D_times}
\end{theorem}
\begin{proof}
	We have
	$$
	\Do^\times (A) = \sum_{\la,\mu} r_{A-A} (\la) r_{A-A} (\mu) n(\la,\mu) \,,
	$$
	where 
	$
	n_{A,A} (\la,\mu) = \sum_x r_{(A-A)\la} (x) r_{(A-A)\mu} (x) \,.
	$ 
	Consider the balanced function $f (x) = f_A(x) = A(x) - |A|/p$. 
	Then we have
	$$
	\Do^\times (A) = \frac{|A|^8}{p} + \sum_{\la,\mu} r_{A-A} (\la) r_{A-A} (\mu) n_{f,f} (\la,\mu)
	=
	\frac{|A|^8}{p} + \sigma \,.
	$$
	Our task is to estimate the error term $\sigma$. 
	Put $L=\log |A|$.
	Splitting the sum, we get 
	$$
	\sigma \ll \sum_{i,j=1}^L \sum_{\la,\mu} n_{f,f} (\la,\mu) r^{(i)}_{A-A} (\la) r^{(j)}_{A-A} (\mu) \,,
	$$
	where by $r^{(j)}_{A-A} (\mu)$ 
	we have denoted  
	the restriction of the function $r_{A-A}$ on some set $P_j$ with 
	$\D_j < r^{(j)}_{A-A} (\mu) \le 2\D_j$, $\mu \in P_j$ and $\D_j >0$ is some number. 
	Clearly, the operator $n_{f,f} (\la,\mu)$ is non--negatively defined and hence
	$$
	\sigma  \ll L \sum_{j=1}^L \sum_{\la,\mu} n_{f,f} (\la,\mu) r^{(j)}_{A-A} (\la) r^{(j)}_{A-A} (\mu) \,.
	$$  
	By the pigeonhole principle there is some $\D =\D_j$ and $P=P_j$ such that
	\begin{equation}\label{f:sigma_n_f}
	\sigma \ll L^2 \D^2 \sum_{\la,\mu} n_{f,f} (\la,\mu) P(\la) P(\mu) \,.
	\end{equation}
	Since $\sum_x f(x) = 0$, we 
	derive from (\ref{f:D_pp}) and (\ref{f:Vinh}) that
	\begin{equation}\label{tmp:23.10_1}
	\sigma \ll L^2 p |A|^2 |P| \D^2 \,.
	\end{equation}
	Now we obtain another bound for $\sigma$.	
	Using Theorem \ref{t:Misha+}, we 
	get 
	\begin{equation}\label{tmp:23.10_2}
	\sigma \ll L^2 \D^2 \left( \frac{|A|^4 |P|^2}{p} + |A|^3 |P|^{3/2}  + |A|^2 |P| \max\{ |A|, |P|\} \right) \,.
	\end{equation}
	Since $P\subseteq A-A$ it is easy to see that the term $|A|^2 |P| \max\{ |A|, |P|\}$ is negligible comparable to $|A|^3 |P|^{3/2}$. 
	Suppose  that the second term in the last formula dominates. 
	Then
	$$
	\sigma \ll L^2 |A|^3 (\D |P|) \cdot (\D^2 |P|)^{1/2} \,.
	$$ 
	Clearly, 
	\begin{equation}\label{tmp:23.10_0}
	\D |P| \le \sum_{x} r^{(j)}_{A-A} (x) \le \sum_{x} r_{A-A} (x) \le |A|^2 \,, 
	\end{equation}
	and
	$$
	\D^2 |P| 
	\le
	\sum_{x} (r^{(j)}_{A-A} (x))^2 
	\le \E^{+} (A) \,. 
	$$
	Hence
	$$
	\sigma \ll L^2 |A|^5 (\E^{+} (A))^{1/2}
	$$
	and we are done.
	If the first term in formula (\ref{tmp:23.10_2}) is the largest one, then $p\le |A| |P|^{1/2}$ and inequality (\ref{tmp:23.10_1}) gives us
	$$
	\sigma \le L^2 p |A|^2 	\D^2 |P| \ll  L^2 |A|^3 \D^2 |P|^{3/2} \,.
	$$
	We see that it is smaller than the second term in (\ref{tmp:23.10_2}) 
	and hence we have proved (\ref{f:D_times_2-}). 
	Another way to bound \eqref{f:sigma_n_f} is just use estimate \eqref{f:Misha+_a} of  Theorem \ref{t:Misha+}.

	To obtain (\ref{f:D_times_2}) we firstly, notice that 
	$$
	\Do^\times (A) = \frac{|A|^8}{p} + \sum_{\la,\mu} r_{A-A} (\la) r_{A-A} (\mu) n_{f,f} (\la,\mu)
	=
	\frac{|A|^8}{p} + \sum_{\la,\mu} r_{f-f} (\la) r_{f-f} (\mu) n_{f,f} (\la,\mu) 
	=
	\frac{|A|^8}{p} + \sigma 
	$$
	and using the Dirichlet principle  as above, we can find a set $P$ and a number $\D$ such that $\D < |r_{f-f} (\mu)| \le 2\D$  on $P$ and 
	$$
	\sigma \ll L^2 \D^2 \sum_{\la,\mu} n_{f,f} (\la,\mu) P(\la) P(\mu) 
	$$
	(from the Fourier transform, say, it is easy to see that $n_{f,f} (\la,\mu) \ge 0$ for any function $f$ but actually one can avoid this step of the proof). 
	Secondly, we have
	$$
	\D |P| \le \sum_{x\in P} |r_{f-f} (x)| \le   
	\sum_x (r_{A-A} (x) + |A|^2/p) \le 2|A|^2 \,,
	$$
	and 
	$$
	\D^2 |P|  
	\le
	\sum_{x\in P} (r_{f-f} (x))^2 
	\le \E^{+} (f) \,. 
	$$
	Thus one can refine the upper bound for $\sigma$, namely,
	$$
	\sigma  = \Do^{\times} (f) \ll L^2 |A|^5 (\E^{+} (f))^{1/2} \,.
	$$
	Similarly as above, we get for any $k\ge 2$
	$$
	\Do^{\times}_k (f) \ll L^2 |A|^{2k+1} (\Do^{\times}_{k-1} (f))^{1/2} \,.
	$$
	Hence by induction 
	\begin{equation}\label{tmp:23.10_3}
	\Do^{\times}_k (f) \ll L^4 |A|^{4k-2-2^{-k+2}} \E^+ (f)^{1/2^{k-1}} \le L^4 |A|^{4k-2-2^{-k+2}} \E^+ (A)^{1/2^{k-1}} \,.
	\end{equation}
	If the third term in (\ref{tmp:23.10_2}) dominates (we are considering the quantity $\Do^{\times}_k (f)$ now), then 
	the term $|A|^{4k-2}$ appears 
	but it is easy to check that it does not exceed the last estimate in (\ref{tmp:23.10_3}) because $\E^+ (A) \ge |A|^2$. 
	Finally, to obtain \eqref{f:D_times_3} use Corollary \ref{cor:weight_inc} to estimate the required number of incidences with weights $\a, \beta$.
	Using 
	\begin{equation}\label{f:E_ab_norms}
	\E^+ (\a,\beta) 
	\le 
	\min \{ \| \a \|_2 \|\beta\|_2 \| \a \|_1 \|\beta\|_1, \| \a \|^2_2 \|\beta\|^2_1, \| \a \|^2_1 \|\beta\|^2_2\}
	\end{equation} 
	as well as $\E^+ (\a,\beta) \ge \| \a \|^2_2 \| \beta \|_2^2$ (at other steps of our iterative procedure similar bounds work)
	and  $\|\a \|_1 \ge \|\a\|_2$ for $\a(x) \ge 0$, one can  check that all conditions (\ref{f:weight_inc_cond}) are satisfied.
	This completes the proof. 
	$\hfill\Box$
\end{proof}


\begin{remark}
	\label{r:D'(A)}
	Similarly, one can obtain an upper bound for the quantity  
	$$ 
	\Do' (A) = 
	|\{ a_1 a_2 +a_3 a_4 = a'_1 a'_2 +a'_3 a'_4 ~:~ a_i,\, a'_i\in A \}| \,,
	$$
	as well as 
	for higher 
	energies 
	$\Do'_k (A) =\T^{+}_k (r_{AA})$, $\Do'_k (A,B)$ and even $\Do'_k (\a,\beta)$ for an arbitrary non--negative function $\a$ and  $\beta=A(x)$. 
	In this case  $\E^{+}$ in \eqref{f:D_times_2-}---\eqref{f:D_times_3} should be changed to $\E^{\times}$.
	Notice that our bound for $\Do'_k (A)$ is better than the correspondent bound in \cite[Theorem 2]{B_multilinear}.
	Besides some additional conditions on $A$ and $k$ are required in \cite{B_multilinear}.

	Is is easy to see that our error term for $\Do'_k (A)$ cannot be significantly improved for large $k$. 
	Indeed, considering $A$ to be a small interval $[n]$, we get $|AA| \ge |A|^{2-\eps}$ and $|kAA| \ll_k |A|^2$.
	Hence  
	one cannot obtain something better than a quadratic saving
	for the error term. 
\end{remark}	

\begin{remark}
	The same method works for convex sets \cite{ik} in the real setting where one obtains 
	$$ 
	\T^{+}_k (A) \ll (\log |A|)^4 |A|^{2k-2-2^{-k+3}} \E^+ (A)^{1/2^{k-2}} \ll (\log |A|)^4 |A|^{2k-2+2^{-k+1}} 
	$$ 
	for any convex set $A$, and we have used that  $\E^{+} (A) \ll |A|^{5/2}$. 
	This coincides with the main result of \cite{ik} up to logarithms. 
	Applying the best known bound for the additive energy of a convex set (see \cite{s_mixed}), namely, $\E^{+} (A) \lesssim |A|^{32/13}$, 
	we obtain an improvement. 
\end{remark}


Theorem above immediately implies a consequence on the growth of the products of the differences 
(here we use a trivial upper bound for the energy $\E^{+} (A) \le |A|^3$).

\begin{corollary}
	Let $A \subseteq \F_p$ be a set. Then for any $\eps> 0$ and an arbitrary integer $k\ge 1$ one has 
	$$
	|(A-A)^k| \gg \min \{ p, |A|^{2-2^{1-k}-\eps} \} \,.
	$$
\end{corollary}

\bigskip

Another quick consequence of Theorem \ref{t:D_times} is (also, see Section 4 from \cite{MPR-NRS})

\begin{corollary}
	Let $A \subseteq \F_p$ be a set, $|A| \le p^{9/16}$. Then
	$$
	|(A-A)(A-A)| \gg \min\{ p, |A|^{3/2+c}\} \,,
	$$
	where $c>0$ is an absolute constant. 
\end{corollary}
\begin{proof}
	By Theorem \ref{t:D_times} and the Cauchy--Schwarz inequality, we obtain
	\begin{equation}\label{tmp:12.02.2017_1}
	|A|^8 \ll Q \left( \frac{|A|^8}{p} + |A|^5 (\E^{+} (A))^{1/2} \log^2 |A| \right) \,,
	\end{equation}
	where 
	$Q= |(A-A)(A-A)|$. 
	Thus if the first term in (\ref{tmp:12.02.2017_1}) dominates, then we are done. 
	Further if $\E^{+} (A) \le |A|^{3-\eps}$, where $\eps>0$ is some small constant, then we are done.
	If not, then put $M = |A|^\eps$ and apply the Balog--Szemer\'edi--Gowers Theorem \cite{TV}, finding $A'\subseteq A$, 
	$|A'| \gg_M |A|$ and $|A'+A'| \ll_M |A'|$. 
	Using Theorem 4 from \cite{MPR-NRS}, we have for any $a\in A$ that 
	$$
	|(A-A)(A-A)| \ge  |(A-a)(A-a)| \gtrsim_M |A|^{14/9} \,,
	$$	
	provided $|A| \le p^{9/16}$.  
	This completes the proof. 
	$\hfill\Box$
\end{proof}

\bigskip

The same argument works for the set $\frac{A-A}{A-A}$ but in this situation much better bounds are known, see \cite{Redei}, \cite{Szonyi}.
Lower bound for the sets of the form $(A-A)(A-A)$, $(A-A)/(A-A)$ in general fields $\F_q$ can be found in paper \cite{MP_F_q}. 

\bigskip

Similar arguments allow us to formulate the second part of Theorem \ref{t:Q} as 

\begin{theorem}
	Let $A\subseteq \F_p$ be a set and $\T(A)$ of  collinear triples in $A\times A$.
	Then
	$$
	0 \le \T(A) - \frac{|A|^6}{p}  \ll \min\left\{ p^{1/2} |A|^{7/2}, |A|^{9/2} \right\} \,.
	$$	
	\label{t:Q_new}
\end{theorem}
\begin{proof}
	We use the arguments from \cite{SZ_inc}. 
	Put $f(x) = A(x) - |A|/p$. 
	It is easy to see that the quantity $\T(A)$  equals the number of incidences  between the planes
	\begin{equation}\label{tmp:20.12_1}
	\frac{x}{a''-a'} -\frac{a'}{a''-a'} - \a y +z =0
	\end{equation}
	and the points 
	$$
	\left( a, \frac{1}{\a''-\a'}, \frac{\a'}{\a''-\a'} \right) \,.
	$$
	Hence as in the proof of Theorem \ref{t:D_times}, we have 
	$$
	\T(A) = \frac{|A|^6}{p}  + \sigma \,,
	$$
	where the sum $\sigma$ 
	counts the number of incidences (\ref{tmp:20.12_1}) 
	with the weight $f(a) f(\a)$. 
	Hence by (\ref{f:Vinh}), we get $\sigma \le p |A|^3$ and by Theorem \ref{t:Misha+}, we have 
	$\sigma \ll \frac{|A|^6}{p} + |A|^{9/2}$. 
	If $|A|^{3/2} \ge p$, then   $\sigma \le p |A|^3 \le |A|^{9/2}$. 
	If $|A|^{3/2} < p$, then $\sigma \ll \frac{|A|^6}{p} + |A|^{9/2} \ll |A|^{9/2}$. 
	In any case $\sigma \ll |A|^{9/2}$. 
	Combining this with the bound from Theorem \ref{t:Q}, we obtain the required result. 
	This completes the proof. 
	$\hfill\Box$
\end{proof}

\bigskip 

For any sets $A, B, C \subseteq \F_p$  put
$$
\No (A) = |\{ a (b - c) = a' (b' - c') ~:~ a,a' \in A,\, b,b'\in B,\, c,c'\in C \}| \,.  
$$
We write $\No (A)$ if $A=B=C$.  
Now we prove an upper bound for the quantity $\No(A)$ which is better than $O(|A|^{9/2})$ for sets $A$  with small energies  $\E^{\times} (A)$ and $\E^{+} (A)$,
namely, when $(\E^{\times} (A))^2 \E^{+} (A) \ll |A|^{8-\eps}$, $\eps>0$.

\begin{corollary}
	Let $A,B \subseteq \F_p$ be sets. Then
	$$
	\No(B,A,A) - \frac{|A|^4|B|^2}{p} \ll 
	(\E^{\times} (B))^{1/2} (\E^{+} (A))^{1/4} |A|^{5/2} \log |A|  \,.
	$$
	\label{c:N(A)}
\end{corollary}
\begin{proof}
	Put $f(x) = A(x) - |A|/p$. 
	We have
	$$
	\No(B,B,A) = \sum_{\la} r_{B/B} (\la) r_{(A-A)/(A-A)} (\la) = \frac{|A|^4|B|^2}{p} + \sum_{\la} r_{B/B} (\la) r_{(f-f)/(A-A)} (\la) 
	=
	\frac{|A|^4|B|^2}{p} + \sigma	\,.
	$$
	By the Cauchy--Schwarz inequality, we 
	get 
	$$
	\sigma^2 = \left( \sum_{\la} r_{B/B} (\la) r_{(f-f)/(A-A)} (\la) \right)^2	
	\le
	\sum_{\la} r^2_{B/B} (\la) \cdot \sum_{\la} r^2_{(f-f)/(A-A)} (\la) 
	\le
	$$
	$$
	\le
	\E^{\times} (B) \left( \Do^\times (A) - \frac{|A|^8}{p} \right) \,.
	$$
	Using Theorem \ref{t:D_times}, we obtain the required result. 
	$\hfill\Box$
\end{proof}

\bigskip

Similarly to $\No(A)$, put 

\begin{equation}\label{def:N'(A)}
\No' (A) = |\{ a_1 a_2 + a_3 = a'_1 a'_2 + a'_3 ~:~ a_i,\, a'_i \in A\}| \,.  
\end{equation}

\begin{corollary}
	Let $A \subseteq \F_p$ be a set. Then
	$$	
	\No'(A) - \frac{|A|^6}{p} \ll  |A|^{5/2} \E^{+} (A)^{1/2} \E^{\times} (A)^{1/4} \log |A| \,.
	$$
	\label{c:N'(A)}
\end{corollary}
\begin{proof}
	Put $f(x) = A(x) - |A|/p$.  
	As in the proof of Theorem \ref{t:D_times}, we have  
	$$
	\No'(A) = \frac{|A|^6}{p} + \sigma \,,
	$$
	where the sum $\sigma$ counts the number of the solutions to equation (\ref{def:N'(A)}) with the weights $f(a_j)$, $f(a'_j)$.
	Thus by Theorem \ref{t:Misha+}, the Dirichlet principle  and the H\"{o}lder inequality, we get 
	$$
	\sigma = \E^{+} (f,r_{ff}) \le (\E^{+} (f))^2 (\E^{+} (r_{ff}))^{1/2} \ll 
	\log |A| \cdot \D \E^{+} (A)^{1/2} \E^{+} (r_{ff}, P)^{1/2} 
	\ll 
	$$
	\begin{equation}\label{tmp:20.12_2}
	\ll \log |A| \cdot \D \E^{+} (A)^{1/2}  \left( \frac{|A|^4 |P|^2}{p} + |A|^3 |P|^{3/2} \right)^{1/2} \,,
	\end{equation}
	where $\D < |r_{ff} (x)| \le 2\D$ on the set  $P$.
	Here we have used 
	the definition of $\E^{+} (r_{ff}, P)$, namely, 
	$$
	\E^{+} (r_{ff}, P) 
	= |\{ p + a_1 a_2 = p' + a'_1 a'_2 ~:~ p, p' \in P \}| \,, 
	$$
	counting with the weights $f(a_1)f(a_2)$, $f(a'_1)f(a'_2)$. 
	Again one can assume that the second term in 
	(\ref{tmp:20.12_2}) 
	dominates.
	Hence using $\D |P| \ll |A|^2$, $\D^2 |P| \le \E^\times (f) \le \E^\times (A)$, we obtain 
	$$
	\sigma 
	\ll
	\log |A| \cdot  \E^{+} (A)^{1/2} |A|^{5/2} \E^{\times} (A)^{1/4} 
	$$
	as required. 	
	$\hfill\Box$
\end{proof}

\bigskip 

The results above imply
an 
estimate for some average sums of the energies (other results on such quantities can be found in \cite{RSS}).

\begin{corollary}
	Let $A \subseteq \F_p$ be a set. Then
	$$
	\sum_{x\in X} \E^{+} (A,xA) - \frac{|X| |A|^4}{p} \ll  |A|^{5/4} \E^{+} (A)^{3/4} \E^{\times} (X)^{1/4} \log^{1/2} |A|  \,.
	$$
\end{corollary}
\begin{proof}
	Indeed, putting $f(x) = A(x) - |A|/p$, we have 
	$$
	\sigma:= \sum_{x\in X} \E^{+} (A,xA) = \sum_{\la} r_{A-A} (\la) r_{(A-A)X}(\la) 
	=
	\frac{|X| |A|^4}{p} + \sum_{\la} r_{A-A} (\la) r_{(f-f)X}(\la) \,.
	$$
	Hence by the H\"older inequality, we get
	$$
	\sigma - \frac{|X| |A|^4}{p} \ll
	\E^{+} (A)^{1/2} \left( \sum_{\la} r^2_{(f-f)X}(\la) \right)^{1/2}\,. 
	$$
	Using  the H\"older inequality one more time, we obtain  
	$$
	\sum_{\la} r^2_{(f-f)X}(\la) \le \E^{\times} (X)^{1/2} \left( \Do^{\times} (A)- \frac{|A|^8}{p} \right)^{1/2} \,.
	$$
	Applying Theorem \ref{t:D_times}, we  have finally
	$$
	\sigma \ll |A|^{5/4} \E^{+} (A)^{3/4} \E^{\times} (X)^{1/4} \log^{1/2} |A|  
	$$
	as required. 
	$\hfill\Box$
\end{proof}

\bigskip

Notice that it was proved in \cite[Corollary 8]{RSS} that for any $A\subseteq \F_p$, $|A| \le p^{3/5}$ there exist two disjoint sets $B$ and $C$ of $A$, each of cardinality $\ge |A|/3$, such that 
$\E^{\times} (B)^3 \E^{+} (C)^2 \lesssim |A|^{14}$. 
%
%

\bigskip

{\Large \bf Unconditional upper bounds for $\Do^\times (A)$, $\Do' (A)$
	and  multilinear exponential sums}
\label{sec:unconditional}

\bigskip

The aim of this section is to prove 

\begin{theorem}
	Let $A\subseteq \F_p$ be a set, $|A| \le p^{2846/4991}$. 
	Then for any $c<\frac{1}{434}$ one has
	$$
	\Do^\times (A) \ll |A|^{13/2 -c} \,.
	$$
	Further if 
	$|A| \le p^{48/97}$,
	then for any 
	$c_1< \frac{1}{192}$ 
	the following holds  
	$$
	\Do^\times (A) \ll |A|^{13/2 -c_1} \,.
	$$
	\label{t:D_uncond}
\end{theorem}

We need two lemmas from \cite[Section 4.5]{MPR-NRS}.

\begin{lemma}
	Let $A\subseteq \F_p$ be a set and $|A+A| =M|A|$, $|A|\le p^{13/23} M^{25/92}$.  
	Then 
	for any $\a\in \F_p$ one has 
	$$
	\E^\times (A+\alpha) \lesssim M^{51/26} |A|^{32/13} \,.
	$$
	\label{l:E(A+a)_1}
\end{lemma}

\begin{lemma}
	Let $A\subseteq \F_p$ be a set and $|AA| =M|A|$, $|A|\le p^{13/23} M^{10/23}$.  
	Then 
	for any $\a\in \F^*_p$ one has 
	$$
	\E^\times (A+\alpha) \lesssim M^{33/13} |A|^{32/13} \,.
	$$
	\label{l:E(A+a)_2}
\end{lemma}


We have a connection between  the quantities $\E^{\times} (A-\a)$  and $\Do^\times (A)$, namely 
\begin{equation}\label{f:D_E_a}
\Do^\times (A) \le |A|^4 \max_{\a\in A} \E^{\times} (A-\a) \,.
\end{equation}
Indeed, just fix four variables $a_2,a'_2, a_4,a'_4$ in (\ref{def:D_times}).

\bigskip

Now we are ready to prove Theorem \ref{t:D_uncond}.

\bigskip

\begin{proof}
	Let $K$ be a parameter and $D=\Do^\times (A) - \frac{|A|^8}{p}$. 
	Our proof is a sort of an algorithm. 
	If $D \lesssim  |A|^{13/2} /K^{1/2}$, then we are done.
	If not, then $\E^{+} (A) \gtrsim |A|^3 /K$ because otherwise  
	by Theorem \ref{t:D_times} we have $D \lesssim |A|^{13/2}/ K^{1/2}$.
	So, we suppose that $\E^{+} (A) \gtrsim |A|^3 /K$. 
	Applying the Balog--Szemer\'{e}di--Gowers Theorem (see the required form  of this result in \cite{BG}), we find $A'\subseteq A$,  $|A'| \gtrsim |A|/K$ such that $|A'+A'| \lesssim K^4 |A'|^3 |A|^{-2}$. 
	By Lemma \ref{l:E(A+a)_1} and estimate (\ref{f:D_E_a}), we have
	$$
	\Do^\times (A') \lesssim |A|^4 |A'|^{5/2-1/26} K^{102/13} M^{51/13} \,,
	$$
	where $M=|A'|/|A|$. 
	The condition of the lemma takes place if 
	\begin{equation}\label{tmp:27.10_1}
	|A'| \le p^{13/23} (K^4 (|A'|/|A|)^2 )^{25/92} 
	\end{equation}
	and we will check 
	(\ref{tmp:27.10_1}) 
	later. 
	After that consider $A\setminus A'$ and continue our algorithm with this set. 
	We obtain disjoint sets $A_1=A'$, $A_2, \dots$ and, clearly, $\sum_j |A_j| \le |A|$. 
	Finally, in view of (\ref{f:D_E_a}) and the norm property of $\E^{\times} (\cdot)$,  we 
	get 
	an upper bound for $D$, namely,  
	$$
	D \lesssim |A|^{13/2} K^{-1/2} 
	+
	K^{102/13} |A|^4 \left(\sum_j (|A_j|^{5/2-1/26}  (|A_j|/|A|)^{51/13})^{1/4} \right)^{4}
	\lesssim
	$$
	$$
	\lesssim
	|A|^{13/2} K^{-1/2} + K^{102/13} |A|^4 |A|^{5/2-1/26} \,.
	$$
	Optimizing over $K$, that is, taking $K = |A|^{1/217}$ we obtain the required bound
	because condition (\ref{tmp:27.10_1}) follows from  
	$$
	|A'|^{42} |A|^{50} \le |A|^{92} \le p^{52} K^{100}
	$$
	or, in other words, from $|A| \le p^{(52 + 100/217)/92} = p^{2846/4991}$.
	Also, in view of the condition  $|A| \le p^{2846/4991}$ the term $|A|^8/p$ is negligible.

	Similarly, using bound (\ref{f:pre3}) of Corollary \ref{c:pre_eigen} and the same calculations, we see that  
	$$
	D \lesssim |A|^{13/2} K^{-1/2} 
	+
	K^{83/26} |A|^4 |A|^{5/2-1/26}  + 
	K^{83/26} \left(\sum_j |A_j| (|A|^{5/2-1/26} )^{1/4} \right)^{4}
	\lesssim
	$$
	$$
	\lesssim
	|A|^{13/2} K^{-1/2} + K^{83/26} |A|^4 |A|^{5/2-1/26} \,.
	$$
	Optimizing over $K$, that is, taking $K = |A|^{1/96}$ we obtain the required bound
	because the condition $|A|K \le \sqrt{p}$ follows from  $|A| \le p^{48/97}$.
	This completes the proof. 
	$\hfill\Box$
\end{proof}

\begin{remark}
	The same arguments, combining with Lemma \ref{l:E(A+a)_2} (or it refinement from \cite{MRSS}) 
	allow us to prove that either
	$\Do'(A) \ll |A|^{13/2-c}$ or $\Do'(A+\a) \ll |A|^{13/2-c}$ for any $\a\neq 0$ and all sufficiently small sets $A$
	(also, see 
	Remark \ref{r:D'(A)}).
	Here $c>0$ is an absolute constant.
\end{remark}


Given three sets $X,Y,Z\subseteq \F_p$ and three complex 
weights $\a = (\a_x)_{x\in X}$, $\beta = (\beta_y)_{y\in Y}$, $\gamma = (\gamma_x)_{x\in Z}$ all bounded by one, put
$$
S(X,Y,Z;\a,\beta,\gamma) = \sum_{x \in X,\, y\in Y,\, z\in Z} \a_x \beta_y \gamma_z e(xyz) \,.
$$
Similarly, for some complex weights $\rho = (\rho_{x,y})$, $\sigma = (\sigma_{x,z})$, $\tau = (\tau_{y,z})$ all  bounded by one, we define
$$
T(X,Y,Z;\rho,\sigma,\tau) = \sum_{x \in X,\, y\in Y,\, z\in Z} \rho_{x,y} \sigma_{x,z} \tau_{y,z} e(xyz) \,.
$$
Such sums were studied in \cite{PS}. 
Using Corollary \ref{c:N(A)} and Theorem \ref{t:D_uncond}, we improve \cite[Theorems 1.3]{PS} and refine  
\cite[Theorems 1.1]{PS} for sets with small energies 
(similar bound can be obtained for a 
correspondent 
sum with four variables, see \cite[Theorem 1.2, 1.4]{PS} 
for large range). 

\begin{corollary}
	Let $|X| \ge |Y| \ge |Z|$.
	Then 
	$$
	S(X,Y,Z;\a,\beta,\gamma) \ll \log^{1/4} |Y| \cdot p^{1/4} |X|^{3/4} |Y|^{5/8} |Z|^{1/2} (\E^{\times} (Z))^{1/8} (\E^{+} (Y))^{1/16} 
	+ |X|^{3/4} |Y| |Z|
	\,,
	$$
	and if $|Y| < p^{48/97}$, then 
	$$
	T(X,Y,Z;\rho,\sigma,\tau) \ll p^{1/8} |X|^{7/8} |Y|^{29/32} |Z|^{29/32} (|Y||Z|)^{-1/3072} \,.
	$$
	\label{c:PS_new}
\end{corollary}


A series of applications of upper bounds for $S(X,Y,Z;\a,\beta,\gamma)$, $T(X,Y,Z;\rho,\sigma,\tau)$ can be found in the same paper \cite{PS}. 
Now we obtain a quantitative form of the main result of 
\cite{B_multilinear}.



\begin{theorem}
	Let $X,Y,Z \subseteq \F_p$ be arbitrary sets.
	Then for any $k \ge 2$ one has
	\begin{equation}\label{f:multilinear}
	\sum_{x \in X,\, y\in Y,\, z\in Z} e(xyz) \ll |X||Y||Z| 
	\left( |Z|^{-2^{-(k+1)}}  + \left(\frac{p}{|X||Y||Z|} \right)^{2^{-(k+1)}}  (|X||Y|)^{2^{-2^k}} \right) \,.
	\end{equation}
	More generally, for any non--negative functions $\a(x)$, $\beta (y)$, $\gamma (z)$  the following holds
	$$
	\sum_{x,y,z}  \a(x) \beta (y) \gamma (z) e(xyz) \ll p^{o(1)} \|\a \|_1 \| \beta \|_1 \| \gamma \|_1 \times
	$$
	\begin{equation}\label{f:multilinear'} 
	\times
	\left( \left( \frac{\| \gamma\|^2_2}{\| \gamma\|^2_1} \right)^{-2^{-(k+1)}}  + \left(\frac{p \|\a\|^2_2 \|\beta\|^2_2 \|\gamma\|^2_2}{\|\a\|^2_1 \|\beta\|^2_1 \| \gamma\|^2 _1} \right)^{2^{-(k+1)}}  \left( \frac{\|\a\|_1 \| \beta \|_1}{\|\a\|_2 \|\beta \|_2} \right)^{2^{-2^k}} \right) \,.
	\end{equation} 
	\label{t:multilinear}	
\end{theorem}
\begin{proof}
	Let $S$ be the sum from \eqref{f:multilinear}.
	Using the Cauchy--Schwarz inequality several times, we get for any $k$
	$$
	|S|^{2^k} \le |Z|^{2^k-1} \sum_{\la, z} (r_{XY} *_{2^k} r_{XY}) (\la) Z(z) e(\la z) \,.
	$$ 
	Applying the Cauchy--Schwarz inequality  one more time, combining with the Parseval identity, we obtain
	\begin{equation}\label{tmp:25.02_1}
	|S|^{2^{k+1}} \le |Z|^{2^{k+1}-2} \Do'_{2^k} (X,Y) p |Z| \,. 
	\end{equation}
	Put $l=2^k$. 
	By an analogue of Theorem \ref{t:D_times} for $\Do'_k (X,Y)$, see Remark \ref{r:D'(A)}, and bound (\ref{f:E_CS}), we have 
	$$
	|S|^{2^{k+1}} \ll p|Z|^{2^{k+1}-1} \left( \frac{(|X||Y|)^{2^{k+1}}}{p} + (\log |X||Y|)^8  \left( \frac{\E^\times (X,Y)}{|X||Y|} \right)^{2^{-l+1}}  
	(|X||Y|)^{2^{k+1}-1} \right)  
	$$
	$$
	\le 
	p|Z|^{2^{k+1}-1} \left( \frac{(|X||Y|)^{2^{k+1}}}{p} + (\log |X||Y|)^8  (|X||Y|)^{2^{-l}}  
	(|X||Y|)^{2^{k+1}-1} \right) 
	$$
	as required. 
	Similarly, to obtain  \eqref{f:multilinear'} just	use Corollary \ref{cor:weight_inc}
	(actually, in this case we do not need in sharp asymptotic formulae but just in the incidences results as in Theorem \ref{t:Misha+})  
	to estimate the required number of incidences with weights $\a$, namely,
	$$
	\Do'_{2^k} (\a,\beta)  
	\ll 
	L^8 \cdot \frac{(\|\a \|_1 \|\beta \|_1)^{2^{k+1}}}{p} + 
	L^8 \cdot 
	(\| \a\|_1 \|\beta\|_1)^{2^{k+1}-2}
	(\| \a\|_2 \|\beta\|_2)^{2-2^{-l+2}}
	(\E^\times (\a,\beta))^{2^{-l+1}} 
	\ll
	$$
	$$
	\ll 
	L^8 \cdot (\|\a \|_1 \|\beta \|_1)^{2^{k+1}} \left( \frac{1}{p} + 
	\frac{\|\a\|^2_2 \|\beta\|^2_2}{\|\a\|^2_1 \|\beta\|^2_1} \left( \frac{\|\a\|_1 \| \beta \|_1}{\|\a\|_2 \|\beta \|_2} \right)^{2^{-l+1}} \right) 
	$$
	and apply the previous arguments. 
	Here we have put 
	$$L = \log (\|\a\|_1 \|\beta\|_1 \| \gamma\|_1 (\|\a\|_2 \|\beta\|_2 \| \gamma\|_2)^{-1} ) \ll \log p$$ 
	and also we have used the bound  $\E^\times (\a,\beta) \le \|\a\|_2 \|\beta \|_2 \| \a \|_1 \| \beta \|_1$ as in  \eqref{f:E_ab_norms}.
	This completes the proof. 
	$\hfill\Box$
\end{proof}

\bigskip

One can obtain Theorem \ref{t:multilinear} for general weights like in \cite{B_multilinear}. 
Also, it is known that	an analogue of our result (and un upper bound for  the multilinear exponential sums as well) for more than three sets follows from the case of three sets, see \cite[Section 8]{B_multilinear}.
We demonstrate this just for exponential sums with  three and four sets, see explicit bounds \eqref{f:etropy_exp_1}, \eqref{f:etropy_exp_2} below.
In general, one can use a simple inequality which takes place for any even $k_j$ and non--negative functions $\a_j$ 
$$
\left| \sum_{a_1, \dots, a_r} \a_1 (a_1) \dots \a_r (a_r) e(a_1 \dots a_r) \right|^{k_1 \dots k_r} 
\le
$$
$$
\le
\frac{(\prod_{j=1}^{r} \| \a_j \|_1)^{k_1 \dots k_r}}{\prod_{j=1}^r \|\a_j\|^{k_j}_1}  
\sum_{a_1, \dots, a_r} (\a_1 *_{k_1} \a_1) (a_1) \dots (\a_r *_{k_r} \a_r) (a_r) e(a_1 \dots a_r)
$$
and insure that the dependence on $r$ in the saving has the form  $p^{-\d/(C_1 \log (C_2 r/\d))^r}$, 
where $C_1,C_2 >0$ are  absolute constants. 
In this case we do not need in sharp asymptotic formulae but just the incidences results as in Theorem \ref{t:Misha+}. 
The dependence in \cite[Theorem A]{B_multilinear} was $p^{-(\d/r)^{Cr}}$, where $C>0$ is another absolute constant,
so our result is better.

\begin{corollary}
	Let $X,Y,Z \subseteq \F_p$ be arbitrary sets such that for some $\d >0$ the following holds
	\begin{equation}\label{f:c_mult_3}
	|X| |Y| |Z| \ge p^{1+\d} \,.
	\end{equation}
	Then 
	\begin{equation}\label{f:etropy_exp_1}
	\sum_{x \in X,\, y\in Y,\, z\in Z} e(xyz) \ll |X||Y||Z| \cdot p^{-\frac{\d }{8 \log (8/\delta)+4}} \,.
	\end{equation}
	Finally, let $r=4$, then for any sets $A_1,\dots,A_r \subseteq \F_p$ with $\prod_{j=1}^r |A_j| \ge p^{1+\d}$ one has
	\begin{equation}\label{f:etropy_exp_2}
	\sum_{a_1\in A_1,\, \dots,\, a_r \in A_r} e(a_1 \dots a_r) \ll \prod_{j=1}^r |A_j| \cdot p^{-\frac{\d }{16 \lceil 0.5 \log (200/\d) \rceil^2}} \,.
	\end{equation}
	\label{c:etropy_exp}
\end{corollary}
\begin{proof}
	To obtain \eqref{f:etropy_exp_1} we want to use  Theorem \ref{t:multilinear}. 
	Put $l=2^k$.
	Then using crude  bounds $|X|, |Y| \le p$, we get 
	$$
	\left(\frac{p}{|X||Y||Z|} \right)^{2^{-(k+1)}}  (|X||Y|)^{2^{-2^k}}
	\le 
	p^{-\d/(2l) + 2^{-l+1}} \le p^{-\d/(4l)} \,,
	$$
	provided
	$$
	\frac{2^l}{l} \ge \frac{8}{\d} \,. 
	$$
	It is easy to see that $l = \lceil 2 \log(8/\d) \rceil \ge 6$ is enough. 
	Applying Theorem \ref{t:multilinear}, we see that the second term in \eqref{f:multilinear} is at most  $|X| |Y| |Z| p^{-\frac{\d}{8\log (8/\delta)+4}}$.
	The first term in this formula equals 
	$$
	|X| |Y| |Z| \cdot |Z|^{-1/2l} \le p^{-\d/(4l)} \,,
	$$
	provided $|Z| \ge p^{\d/2}$ and hence this bound has the same quality. 
	Suppose that $|Z| < p^{\d/2}$.
	Then by \eqref{f:c_mult_3}, we obtain $|X||Y| \ge p^{1+\d/2}$ and by a trivial bound for double exponential sums, we get
	$$
	\sum_{x \in X,\, y\in Y,\, z\in Z} e(xyz) \ll |Z| \cdot \sqrt{p |X| |Y|} = |X| |Y| |Z| \cdot \sqrt{p/(|X| |Y|)}
	\le
	|X| |Y| |Z| \cdot p^{-\d/4}
	$$ 
	which is even better.

	Now to get \eqref{f:etropy_exp_2} we can replace all $A_i (x)$ to $A_i (x) - |A_i|/p$ if we want. 
	Let $S$ be the sum from \eqref{f:etropy_exp_2}. 
	Also, let $\eta (x) = r_{A_1 A_2} (x)$ and $L = \log (|A_1||A_2||A_3|)$.
	Applying the H\"older inequality, we obtain
	$$
	|S|^{2n} \le (|A_3| \dots |A_r|)^{2n-1} \sum_{a_3\in A_3,\, \dots,\, a_r \in A_r} \left| \sum_x  \eta (x) e(x a_3 \dots a_r) \right|^{2n}
	=
	$$
	$$
	=
	(|A_3| \dots |A_r|)^{2n-1}  \sum_x (\eta*_{2n} \eta) (x) \sum_{a_3\in A_3,\, \dots,\, a_r \in A_r}  e(x a_3 \dots a_r) \,.
	$$
	Applying the H\"older inequality as in Theorem \ref{t:multilinear} (see estimate (\ref{tmp:25.02_1})), combining with an analogue of \eqref{f:D_times_3} for 
	$\Do'$  and using $r=4$, we obtain for any $l$ 
	$$
	|A_3|^{-1} \sum_x (\eta*_{2n} \eta) (x) \sum_{a_3\in A_3,\, \dots,\, a_r \in A_r}  e(x a_3 \dots a_r) 
	\le
	\frac{|A_4|}{|A_3|} \left( \frac{p \Do'_{2l} (\eta *_{2n} \eta, A_3) }{|A_4|} \right)^{1/4l}
	\ll
	$$
	$$
	\ll \frac{|A_4|}{|A_3|} \left( p |A_4|^{-1} 			L^8 
	(|A_1|^{2n} |A_2|^{2n} |A_3|)^{4l-2}
	(\| \eta *_{2n} \eta \|^2_2 |A_3|)^{1-2^{-2l+1}}
	(\E^\times (\eta *_{2n} \eta,A_3))^{1/2^{2l-1}} 
	\right)^{1/4l}
	$$
	$$
	\ll
	\frac{|A_4|}{|A_3|} \left( p |A_4|^{-1} 			L^{16} 
	(|A_1|^{2n} |A_2|^{2n} |A_3|)^{4l-2}
	((|A_1||A_2|)^{4n-1+2^{-2n+1}} |A_3|)^{}
	|A_3|^{1/2^{2l-1}} 
	\right)^{1/4l}
	$$
	$$
	\ll
	|A_4| (|A_1| |A_2|)^{8ln} 
	\left( \frac{p}{|A_1||A_2||A_3||A_4|} \cdot L^{16} (|A_1| |A_2|)^{2^{-2n+1}} |A_3|^{2^{-2l+1}}  \right)^{1/4l} \,.
	$$
	Here we have used the fact that $\E^\times (\eta *_{2n} \eta,A_3) \le \| \eta *_{2n} \eta \|^2_2 |A_3|^2$. 
	So, taking $l=n$ such that $n=\lceil 0.5 \log (200/\d) \rceil$, we obtain the required result.
	This completes the proof. 
	$\hfill\Box$
\end{proof}

\section{An asymptotic variant of the Balog--Wooley decomposition Theorem}
\label{sec:BW}

Now we prove a result in the spirit of \cite{BW}, \cite{KS2}, \cite{RRS}, \cite{RSS}. 
The difference between our Theorem \ref{t:BW_as} and these results is that  we have an asymptotic formula for the energy. 
Of course in formulae (\ref{f:BW_as_1}), (\ref{f:BW_as_2}) below the additive and multiplicative energy can be swapped (for some other sets $B$ and $C$) 
and moreover can be replaced to other energies (see \cite{RSS}). 

\begin{theorem}
	Let $A\subseteq \F_p$ be a set and let $1\le M \le p/(2|A|)$ be a parameter.
	There exist two disjoint subsets $B$ and $C$ of $A$ such that $A = B\sqcup C$ 
	and 
	\begin{equation}\label{f:BW_as_1}
	\E^{+} (B) - \frac{|B|^4}{p} \le \frac{|A|^{2/3} |B|^{7/3}}{M} \,,
	\end{equation}
	and for  any set $X \subseteq \F_p$ one has
	\begin{equation}\label{f:BW_as_2}
	\E^{\times} (C,X)  \lesssim  \frac{M^2 |X|^2 |A|^2 }{p} + M^{3/2} |A|^{} |X|^{3/2} \,.
	\end{equation}
	\label{t:BW_as}
\end{theorem}
\begin{proof}
	Our proof is a sort of an algorithm similar to the arguments of the proof of Theorem \ref{t:D_uncond}. 
	At the first step put $B=A$ and $C=\emptyset$. 
	Suppose that we have constructed 
	$B$ at some step of our algorithm. 
	Write $f_B(x) = B(x) - |B|/p$.  
	Then $\E^{+} (B) = \frac{|B|^4}{p} + \E^{+} (f_B, B)$. 
	If 
	$$
	\E^{+} (f_B, B) = 
	\sum_{x} B(x) r_{f_B+B-B} (x)
	\le \frac{|A|^{2/3} |B|^{7/3}}{M} \,,  
	$$
	then we are done. 
	If not, then $\E^{+} (f_B, B)  \ge \frac{|A|^{2/3} |B|^{7/3}}{M}$ and by the pigeonhole principle we find a set $P$ such that 
	$\D < |r_{f_B - B} (x)| \le 2 \D$ for all $x\in P$
	and
	$$
	\frac{|A|^{2/3} |B|^{7/3}}{M} \le \E^{+} (f_B, B) \lesssim \D \sum_x B(x) |r_{f_B+P} (x)|  \le \D \sum_x B(x) r_{B+P} (x) + \frac{\D |B|^2 |P|}{p}
	\le
	$$
	$$ 
	\le
	\D \sum_x B(x) r_{B+P} (x) + \frac{|B|^4}{p} \le 2 \D \sum_x B(x) r_{B+P} (x) 
	\,.
	$$
	Here we have used the assumption $M\le p/(2|A|)$. 
	Using Lemma \ref{l:Misha_c} with $P=P$ and $A=B$, we find a set $B_* \subseteq B$ and a number $q$, $q \lesssim |B_*|$
	such that for any $x\in B_*$ one has $r_{B+P} (x) \ge q$,
	and $\sum_x B(x) r_{B+P} (x) \sim |B_*| q$.
	We have
	$$
	\E^{\times} (B_*,X) \le q^{-2} \left| \left\{ (b+p) x = (b'+p')x' ~:~ x,x'\in X,\, b,b'\in B,\, p,p'\in P \right\} \right| \,.
	$$
	Using Theorem \ref{t:Misha+} and the definition of $q$, we obtain
	$$
	\E^{\times} (B_*,X) \ll q^{-2} \left( \frac{|X|^2 |B|^2 |P|^2}{p} + (|X||B||P|)^{3/2} + |X||B||P| \max\{|X|,|B|,|P|\} \right) 
	\lesssim
	$$
	\begin{equation}\label{tmp:28.10_0}
	(\E^{+} (f_B, B))^{-2} |B_*|^2 \D^2 \left( \frac{|X|^2 |B|^2 |P|^2}{p} + (|X||B||P|)^{3/2} + |X||B||P| \max\{|X|,|B|,|P|\} \right) \,.
	\end{equation}
	Now clearly,
	\begin{equation}\label{tmp:28.10_1}
	\D |P| \le \sum_x r_{B-B} (x) + |B|^2 \le 2 |B|^2 \,.
	\end{equation}
	and 
	\begin{equation}\label{tmp:28.10_2}
	\D^2 |P| \le \sum_x r^2_{f_B-B} (x) =  \E^{+} (f_B, B) \,.
	\end{equation}
	Then using the last formulae, the fact that $B\subseteq A$ and returning to (\ref{tmp:28.10_0}), we obtain 
	$$
	\E^{\times} (B_*,X) \lesssim (\E^{+} (f_B, B))^{-2} |B_*|^2 \times
	$$
	$$
	\times 
	\left( \frac{|X|^2 |B|^6}{p} + |X|^{3/2} |B|^{7/2} (\E^{+} (f_B, B))^{1/2} + \D^2 |X||B||P| \max\{|X|,|B|,|P|\} \right) 
	\le
	$$
	\begin{equation}\label{tmp:28.10_3}
	\frac{M^2 |X|^2 |B_*|^2}{p} + M^{3/2} |A|^{-1} |B_*|^2 |X|^{3/2} 
	+
	M^2 |A|^{-4/3} |B|^{-11/3} \D^2 |B_*|^2 |X||P| \max\{|X|,|B|,|P|\} \,.
	\end{equation}
	Suppose that the third term in the last estimate is negligible. 
	After that we consider $B\setminus B_*$ and continue our algorithm with this set. 
	We obtain disjoint sets $A_1=B_*$, $A_2, \dots$ and let $C$ be its union.
	Finally, in view of  
	the norm property of $\E^{\times} (\cdot,X)$,  we get  an upper bound for $\E^{\times} (C,X)$, namely,
	$$
	\E^{\times} (C,X) \le \left( \sum_{j} (\E^{\times} (A_j,X))^{1/2} \right)^2 
	\lesssim  
	\left( \frac{M^2 |X|^2 }{p} + M^{3/2} |A|^{-1} |X|^{3/2} \right) \cdot \left( \sum_{j} |A_j|  \right)^2 
	\le 
	$$
	\begin{equation}\label{tmp:28.10_4}
	\le
	\frac{M^2 |X|^2 |A|^2 }{p} + M^{3/2} |A|^{}  |X|^{3/2} \,.
	\end{equation}
	It remains to check that the third term in (\ref{tmp:28.10_3}) is negligible. 
	From 
	(\ref{tmp:28.10_4}), 
	it follows that 
	\begin{equation}\label{tmp:22.12_1}
	M^3 \le |X| \le |A|^2 /M^3 
	\end{equation}
	because otherwise there is nothing to prove. 
	Since, $\E^{+} (f_B, B)  \ge \frac{|A|^{2/3} |B|^{7/3}}{M}$ we easily derive 
	\begin{equation}\label{tmp:22.12_2}
	|B| \ge \E^{+} (f_B, B)^{1/3} \gg |A|/M^{3/2} \,.
	\end{equation}
	Using these bounds, as well as a trivial upper estimate $\E^{+} (f_B, B) \le |B|^3$ one can 
	quickly 
	check that 
	\begin{equation*}\label{tmp:28.10_cond}
	\E^{+} (f_B, B) \le M^{-1/2} |X|^{-1/2} |B|^{11/3} |A|^{1/3} \,, \quad \E^{+} (f_B, B) \le M^{-1/2} |X|^{1/2} |B|^{8/3} |A|^{1/3} \,, 
	\end{equation*}
	and
	$$
	|B|^{1/3} M^{1/2} \le |X|^{1/2} |A|^{1/3}
	$$
	and 
	thus indeed 
	the third term in (\ref{tmp:28.10_3}) is negligible.
	This completes the proof. 
	$\hfill\Box$
\end{proof}

\bigskip

Notice that one cannot obtain an asymptotic formula as in (\ref{f:BW_as_1}) for both sets $B$ and $C$. 
Indeed, it would imply that 
$|B+B|, |CC| \gg p$ 
but there are sets $A$ having small sumsets and product sets, just put $A=P\cap \G$, where 
$P$ is a suitable arithmetic progression and $\G$ is a subgroup.  




\bigskip 

Now let us obtain a result on the sum--product phenomenon (of course one can replace below $+$ to $*$ and vice versa).

\begin{corollary}
	Let $A\subset \F_p$ be a set.
	Then either 
	$$ 
	|A+A| \ge 5^{-1} \min\{|A|^{6/5}, p/2\} 
	$$ 
	or 
	$$|AA| \gtrsim \min\{ p |A|^{-2/5}, |A|^{6/5} \} \,.$$ 
	\label{c:p_sum-prod}
\end{corollary} 
\begin{proof}
	Apply Theorem \ref{t:BW_as} with $M=|A|^{1/5}$. 
	We find two disjoint subsets $B$ and $C$ of $A$ such that $A = B\sqcup C$ and estimates (\ref{f:BW_as_1}), (\ref{f:BW_as_2}) take place.
	If $|B| \ge |A|/2$ and $M|A| = |A|^{6/5} \le p/2$, then by  (\ref{f:BW_as_1}) and the Cauchy--Schwarz inequality  one has 
	$$
	|B|^4 \le |B+B| \left( \frac{|B|^4}{p}  + \frac{|A|^{2/3} |B|^{7/3}}{M}  \right) \le 
	|A+A| \cdot \frac{3|A|^{2/3} |B|^{7/3}}{2M} 
	$$
	and hence $|A+A| \ge 5^{-1} |A|^{6/5}$.
	If  $|B| \ge |A|/2$,   
	then just consider a maximal set $A' \subseteq A$ of size $|A'|^{6/5} \le p/2$ and use the previous arguments. 
	Finally, if $|C|\ge |A|/2$, then putting $X=A$ in  (\ref{f:BW_as_2}), we obtain 
	$$
	|AA| \ge |AC| \ge \frac{|A|^2 |C|^2}{\E^\times (A,C)}\gtrsim \min\{ p |A|^{-2/5}, |A|^{6/5} \} 
	$$
	as required. 
	$\hfill\Box$
\end{proof}

\section{Asymptotic formulae in $\SL_2 (\F_p)$}
\label{sec:SL}




Now we consider the action of $\SL_2 (\F_p)$ on $\F_p$ and we begin 
with 
our version (see Theorem \ref{t:flattering} below) of so--called  $L_2$--flattering lemma from \cite{B_hyp} 
(also, see \cite{BG_SL}, \cite{SX}) which is a direct consequence of the celebrated Helfgott's  Theorem \ref{t:Harald_SL2}. 
The proof of Theorem \ref{t:flattering} can be found in the Appendix.

\begin{theorem}
	Let
	$\mu$ be a symmetric probability measure on $\SL_2 (\F_p)$ such that for a parameter 
	$K\ge 1$ one has \\
	$\circ~$ $\mu (g \G) \le K^{-1}$ for any proper subgroup $\G\subset \SL_2 (\F_p)$, $g\in \SL_2 (\F_p)$ and \\
	$\circ~$ $\| \mu \|_\infty \le K^{-1}$.\\
	Then for any integer $k \le K^{c_*}$
	the following holds  
	\begin{equation}\label{f:flattering}
	0\le  \| \mu *_{2^k} \mu \|^2_2 - |\SL_2 (\F_p)|^{-1} \le C^k_* K^{-c_* k} \,,
	\end{equation}
	where $c_* \in (0,1)$, $C_* >1$ are absolute constants. 
	\label{t:flattering}
\end{theorem}

Now we derive some consequences of Theorem \ref{t:flattering} to sum--product phenomenon and we begin with some generalizations of arguments from \cite{NG_S}.
Transformations 
$$
y = \frac{-1}{x+a}  \,, \quad \quad y = \frac{-1}{x+a} + b
$$
correspond to $\SL_2 (\F_p)$ matrices 
\[
s' = \left( {\begin{array}{cc}
	0 & -1 \\
	1 & a \\
	\end{array} } \right) \in S' \,,
\quad \quad 
s_{a,b} = 
\left( {\begin{array}{cc}
	b & -1+ab \\
	1 & a \\
	\end{array} } \right) \in S \,.
\]
The 
collections 
$S',S$  
of such matrices 
are clearly connected with 
continued fractions 
$$
[a_1,a_2,\dots ] = \frac{1}{a_1+\frac{1}{a_2+\dots}}  
$$
and correspond to 
classical 
continuants (see, e.g., \cite{Hinchin}), 
as well as 
continuants (entries) of the product of two matrices 
$	\left( {\begin{array}{cc}
	0 & 1 \\
	1 & a_1 \\
	\end{array} } \right) 
\left( {\begin{array}{cc}
	0 & 1 \\
	1 & a_2 \\
	\end{array} } \right)
$.

We need several properties of the set $S$ and the first one can be found in \cite{NG_S} (or see the proof of Lemma \ref{l*:intersection} and Remark \ref{r:one_dim} below).
It is easy to 
check 
that Lemma \ref{l:intersection}
does not hold for the set $S'$.

\begin{lemma}
	Suppose that in the definition of the set $S$ one has $a \in B_1$, $b \in B_2$.
	For any $g_1, g_2 \in \SL_2 (\F_p)$ the following holds 
	\begin{equation}\label{f:intersection}
	|g_1 \B g_2 \cap S| \le \max\{ |B_1|, |B_2| \} \,.
	\end{equation}
	Moreover, 
	for any dihedral subgroup $\G$ one has 
	\begin{equation}\label{f:intersection+}
	|g_1 \G g_2 \cap S| \le 8\max\{ |B_1|, |B_2| \} \,.
	\end{equation}
	\label{l:intersection}
\end{lemma}


Theorem \ref{t:flattering}, combining with Lemma \ref{l:intersection},  gives  a consequence for continued fractions 
(consider the following two--step transformation $\frac{1}{a+\frac{1}{x+b}}$ with the correspondent map from $\SL_2 (\F_p)$
and a well--known connection of continued fractions with continuants, see, e.g., \cite{Hinchin}).
Another  way to derive  Theorem \ref{t:CF_growth} is  
iteratively apply Corollary \ref{c:pol} below 
but 
this way gives worse bounds.

\begin{theorem}
	Let $A\subseteq \F_p$ be a set $|A| >p^\eps$, $\eps>0$.
	Then for any $k > C^{1/\eps}$, where $C>0$ is an absolute constant and for any $x\in \F_p$  one has 
	$$
	| \{ x = [a_1, a_2, \dots , a_k ] ~:~ a_j \in A \}|  = \frac{|A|^k}{p} (1+o(1)) \,.
	$$
	\label{t:CF_growth}
\end{theorem}


Now we can formulate our "counting lemma".
Here $S$ can be any set of matrices satisfying (\ref{f:intersection}), (\ref{f:intersection+}).
Having a function $f : \F_p \to \C$ by $\langle f \rangle$ denote $\sum_{x\in \F_p} f(x)$.

\begin{lemma}
	Let $f_1,f_2 : \F_p \to \C$ be functions and $|S| > p^{\eps}$. 
	The 
	number of the solutions to the equation
	\begin{equation}\label{f:sa=a'}
	s a_1 = a_2 \,,
	\end{equation}
	counting with weights $f_1 (a_1)$, $f_2 (a_2)$, 
	and with the restriction 
	$s\in S$ 
	is 
	\begin{equation}\label{f:counting}
	\frac{|S| \langle f_1 \rangle \langle f_2 \rangle }{p} + 2 \theta \|f_1 \|_2 \|f_2 \|_2  |S| p^{-1/2^{k+2}} \,,
	\end{equation}
	where $|\theta| \le 1$ and $k=k(\eps)$. 
	\label{l:counting}
\end{lemma}
\begin{proof}
	Denote by $\sigma$ the number of the solutions to equation (\ref{f:sa=a'}).
	In terms of the generalized convolution, 
	we have 
	$$
	\sigma = \sum_x f_2 (x) (S*f_1) (x) \,.
	$$
	Let $f(x) = f_1(x) - \langle f_1 \rangle /p$. 
	Then 
	\begin{equation}\label{f:sigma_balanced}
	\sigma = \frac{|S| \langle f_1 \rangle \langle f_2 \rangle}{p} + \sum_x f_2 (x) (S* f) (x) = \frac{|S| \langle f_1 \rangle \langle f_2 \rangle}{p} + \sigma_* \,.
	\end{equation}
	Using the Cauchy--Schwarz inequality, we get 
	$$
	\sigma^2_* \le 
	\|f_2 \|_2^2 \sum_x (S*f)^2 (x) = \|f_2 \|_2^2 \sum_{x} f (x) (S^{-1} * S * f) (x)   \,.
	$$
	Here $\mu(x) := (S^{-1} * S) (x) : \SL_2 (\F_p) \to \R$ is the usual convolution on the group $\SL_2 (\F_p)$. 
	Notice that $\mu(x) = \mu(x^{-1})$.
	Also, $\| f\|_2^2  = \|f_1 \|_2^2 - \langle f_1 \rangle^2/p \le  \|f_1 \|_2^2$.
	Thus
	$$
	\sigma_1 := \sigma^2_* \le \|f_2 \|_2^2 \sum_{x} f (x) (\mu * f) (x)  
	$$
	and we obtain by the iteration of the previous arguments
	(also, see the proof of Corollary \ref{c:E_k_sigma})
	that for any $k$ one has 
	\begin{equation}\label{f:sigma_1}
	\sigma^{2^k}_1 \le \|f_2 \|_2^{2^{k+1}} \|f_1 \|_2^{2^{k+1}-2} \sum_{x} f (x) ( \mu *_{2^k} \mu * f) (x) \,,
	\end{equation}
	where in $\mu *_{2^k} \mu$ the convolution on $\SL_2 (\F_p)$
	is taken $2^k-1$ times (so, we have written the function $\mu$ exactly $2^k$ times). 
	Now applying Lemma \ref{l:Frobenious}, we get
	\begin{equation}\label{f:sigma_pred}
	\sigma_* \le \|f_1 \|_2 \|f_2 \|_2  \cdot (2p \| \mu*_{2^k} \mu \|_2 )^{1/2^{k+1}} := \|f_1 \|_2 \|f_2 \|_2  (2p \T^{1/2}_{2^{k+1}} (S) )^{1/2^{k+1}}\,. 
	\end{equation}
	Here 
	$$
	\T_{2l} (S) = |\{ s^{-1}_1 s'_1 \dots  s^{-1}_l s'_l = w^{-1}_1 w'_1 \dots  w^{-1}_l w'_l  ~:~ s_j, s'_j, w_j, w'_j \in S \} |\,.
	$$
	Trivially, we have $\| \mu \|_\infty = |S|^{-1}$. 
	Suppose that $p$ is sufficiently large such that $|S| > p^{\eps} > 60$, say, 
	and hence in view of Lemma  \ref{l:intersection} we avoid all subgroups (1)---(3) from Theorem \ref{t:classification} in the sense that the conditions of 
	Theorem \ref{t:flattering} take place with $K= |S|^{1/2}/8$.
	Thus by Theorem \ref{t:flattering}, we 
	find 
	some $k = k(\eps)$ such that 
	$$\T^{}_{2^k} (S) \le 2|S|^{2^{k+1}} p^{-3} \,. $$
	Hence in view of (\ref{f:sigma_pred}), we get 
	$$
	\sigma_* \le 2 \|f_1 \|_2 \|f_2 \|_2   |S| p^{-1/2^{k+2}}
	$$
	as required.
	$\hfill\Box$ 
\end{proof}


\begin{remark}
	From the proof of Theorem \ref{t:flattering},  it follows that  the optimal choice of $k$ is 
	$k \sim \log p / \log |S|$.
	On the other hand, bound (\ref{f:counting}) is nontrivial if $k\ll \log\log p$. 
	So, one can check that the assumption $|S| > p^\eps$ can be relaxed to $\log |S| \gg \log p / \log \log p$
	and under this condition  we obtain a nontrivial bound in (\ref{f:counting}). 	
\end{remark}


Now we obtain an  interesting consequence of Lemma \ref{l:counting} to sets with small doubling
(another result of the same sort about the  products of sets with small doubling is contained in \cite{NG_S}). 
Combining Corollary \ref{c:1/A} 
and Lemma  \ref{l:1/A_energy} from Section \ref{sec:preliminaries} 
we derive 
Theorem \ref{t:1/A_intr} from the Introduction.

\begin{corollary}
	Let $A_1,A_2, B\subseteq \F_p$, 
	$|B| \ge p^\eps$, 
	$\eps > 0$ 
	and $|A_1+B| \le K_1 |A_1|$, $|A_2+B| \le K_2 |A_2|$.
	Then the number of the  solutions to the equation
	\begin{equation}\label{f:1/A}
	r_{A^{-1}_1 - A^{-1}_2} (1) = \left| \left\{ \frac{1}{a_1} - \frac{1}{a_2} = 1 ~:~ a_1 \in A_1,\, a_2 \in A_2 \right\} \right| 
	\end{equation}
	is at most 
	\begin{equation}\label{f:1/A_2}
	\frac{K_1 K_2 |A_1||A_2|}{p} + 2 (K_1 K_2 |A_1| |A_2|)^{1/2 }p^{-1/2^{k+2}} \,,
	\end{equation}
	where $k=k(\eps)$. 
	\label{c:1/A}
\end{corollary}
\begin{proof}
	Clearly, the number of the  solutions to the equation (\ref{f:1/A}) 
	does not exceed 
	$$
	|B|^{-2} \left| \left\{ \frac{1}{x - b} - \frac{1}{y- c} = 1 ~:~ x \in A_1+B,\, y \in A_2  + B,\,  b,c \in B \right\} \right| \,.
	$$
	In other words, we have
	$$
	xy -(b+1) y - (c-1)x + (b+1) (c-1) + 1 = 0
	$$
	or, equivalently,  in terms of $\SL_2 (\F_p)$ actions $s_{-(b+1),c-1} x = y$, where 
	\[
	s_{-(b+1),c-1} = 
	\left( {\begin{array}{cc}
		c-1 & -1-(b+1)(c-1) \\
		1 & -(b+1) \\
		\end{array} } \right) \in S \,.
	\]
	Applying Lemma \ref{l:counting} to sets $-(B+1)$, $C-1$, we obtain the required bound.
	This completes the proof. 
	$\hfill\Box$ 
\end{proof}

\bigskip

Thus when $K_1,K_2$ are small and sizes of $A_1$, $A_2$ are close to $p$ our upper bound (\ref{f:1/A_2})  is close to the 
right 
asymptotic formula for 
$r_{A^{-1}_1 - A^{-1}_2} (1)$. 
The same can be proved in the case of $\G$--invariant sets $Q_1,Q_2$ for its intersection $|Q_1 \cap (Q_2+x)|$, where $x\ne 0$ is an arbitrary, see Section \ref{sec:first_results}. 
So, these two  phenomena  are parallel to each other.

\bigskip

Now let us obtain an application to estimates for some exponential sums.

\begin{corollary}
	For any functions $f,g : \F_p \to \C$, 
	$\langle f \rangle = 0$,  
	and for any set $B$ with $|B| \ge p^\eps$, 
	$\eps > 0$
	one has 
	\begin{equation}\label{f:new_exp_sums_1}
	\sum_{x,y} f(x) g(y) \sum_{b_1,b_2 \in B}  e\left( y\left( \frac{1}{x+b_1} + b_2 \right) \right) \ll \| f \|_2 \| g \|_2 \sqrt{p} |B|^2 p^{-\d} \,.
	\end{equation}
	Further for any nontrivial multiplicative character $\chi$, we get  
	\begin{equation}\label{f:new_exp_sums_2}
	\sum_{x,y} f(x) g(y)  \sum_{b_1,b_2 \in B} \chi \left( y+ b_2 + \frac{1}{x+b_1}  \right) \ll \| f \|_2 \| g \|_2 \sqrt{p} |B|^2 p^{-\d} \,,
	\end{equation}
	and for $|Y| \ge p^\eps$,  one has 
	\begin{equation}\label{f:new_exp_sums_3}
	\sum_{x,y} f(x) Y(y)  \sum_{b_1,b_2 \in B} \chi \left( y+ b_2 + \frac{1}{x+b_1}  \right) \ll_\eps 	
	\|f \|_1 |B|^2 |Y| \cdot \left( \frac{p^{1/2 -\d} \|f \|^2_2}{ \|f\|^2_1} \right)^{\eps} 	\,.
	\end{equation}
	Here 
	$\d = \d (\eps) > 0$. 
	\label{c:new_exp_sums}
\end{corollary}
\begin{proof}
	Using the Cauchy--Schwarz inequality, we 
	obtain 
	$$
	\left|\sum_{x,y} f(x) g(y) \sum_{b_1,b_2 \in B}  e\left( y\left( \frac{1}{x+b_1} + b_2 \right) \right) \right|^2
	\le
	$$
	$$
	\le
	\| g\|^2 _2 p \cdot \sum_{x,x'} f(-x) \ov{f(-x')} \cdot |\{  s_{b_1,b_2} x = s_{b'_1,b'_2} x' ~:~ b_1,b'_1 \in -B,\, b_2,b'_2 \in B \}| 
	= 
	\| g\|^2 _2 p \cdot \sigma' \,.
	$$
	Applying the arguments of the proof of Lemma \ref{l:counting} and the assumption $\sum_x f(x) =0$, we have 
	$$
	\sigma' \le 2 \| f\|^2_2 |B|^4 p^{-1/2^{k+2}} 
	$$
	as required.

	To obtain (\ref{f:new_exp_sums_2}), we use the usual  properties of multiplicative characters (see, e.g., \cite{BEW_book}) to derive 
	$$
	\sigma:=\left| \sum_{x,y} f(x) g(y)  \sum_{b_1,b_2 \in B} \chi \left( y+ b_2 + \frac{1}{x+b_1}  \right)  \right|
	\ll
	$$
	$$
	\ll 
	\| g\|^2 _2 p \cdot \sum_{x,x'} f(-x) \ov{f(-x')} \cdot |\{  s_{b_1,b_2} x = s_{b'_1,b'_2} x' ~:~  b_1,b'_1 \in -B,\, b_2,b'_2 \in B \}| 
	$$
	and repeat the arguments.

	Now let us use the usual Burgess' method, see, e.g., \cite{IK}.
	Namely, by the H\"older inequality and Weil's result (see \cite[Theorem 11.23]{IK}), we get  for any positive integer $k$  
	$$
	\sigma^{2k} \le(\| f\|_1 |B|^2)^{2k-2}  \cdot \sum_{x,x'} f(x) \ov{f(x')} \left| \left\{ b_2 + \frac{1}{x+b_1} = b'_2 + \frac{1}{x'+b'_1} ~:~ b_1, b'_1, b_2, b'_2 \in B \right\}\right| 
	\times
	$$
	$$
	\times 
	\left( (2k)^k p |Y|^k + 2k \sqrt{p} |Y|^{2k} \right) \,.
	$$
	By Lemma \ref{l:counting} we find 
	$l=l(\eps)$ 
	such that 
	\begin{equation}\label{tmp:04.12_1'}
	\sum_{x,x'} f(x) \ov{f(x')} \left| \left\{ b_2 + \frac{1}{x+b_1} = b'_2 + \frac{1}{x'+b'_1} ~:~ b_1, b'_1, b_2, b'_2 \in B \right\}\right| 
	\le 
	2 \|f \|^2_2 |B|^4 p^{-1/2^{l+2}} \,.
	\end{equation}
	Taking  constant $k = \lceil 1/2\eps \rceil$
	such that $ |Y|^{k} \gg_\eps (2k)^k \sqrt{p}$,  
	we obtain
	$$
	\sigma \ll_\eps \|f \|_1 |B|^2  p^{-1/k2^{l+3}} \cdot \left( \frac{\sqrt{p} \|f \|^2_2}{ \|f\|^2_1} \right)^{1/2k} 
	\le 
	\|f \|_1 |B|^2 |Y| \cdot \left( \frac{p^{1/2 -\d} \|f \|^2_2}{ \|f\|^2_1} \right)^{\eps} 	\,.
	$$
	Here we have denoted $1/2^{l+2}$ as $\delta$. 
	This completes the proof. 
	$\hfill\Box$ 
\end{proof}


\begin{remark}
	The results above  are nontrivial (suppose for simplicity that $f(x) = X(x)$
	for some set $X$) if $|X| \gg p^{1/2-\d}$
	and the restriction to the lower bound for size of $B$ can be extended to  $\log |B| \gg \log p / \log \log p$. 
	\label{r:chi_Burgess}
\end{remark}



Now we consider some {\it one--parametric} families of matrices in $\SL_2 (\F_p)$ for which the above methods can be applied.

\begin{lemma}
	Let $B\subseteq \F_p$ and $S_{r_1,r_2} \subseteq \SL_2 (\F_p)$ be a set of the form
	\[
	s_{b} = 
	\left( {\begin{array}{cc}
		1 & r_1 (b) \\
		r_2 (b) &  1+ r_1 (b) r_2 (b) \\
		\end{array} } \right) \in S_{r_1,r_2} \subseteq \SL_2 (\F_p)\,,
	\]
	where $r_1 = p_1/q_1,r_2 = p_2/q_2$ are 
	non--constant 
	rational functions 
	such that 
	$$
	\{ p_1 p_2, p_1 q_2, p_2 q_1, q_1 q_2 \}\,, 
	\{ p_1 q_1 q_2, p_1 p_2 q_1, p_1^2 p_2, q_1^2 q_2, q^2_1 p_2\}
	$$
	are linearly independent over $\F_p$. 
	Put 
	$$
	M:= 
	\max \{ \deg(p_1), \deg(p_2), \deg(q_1), \deg(q_2) \} \,.
	$$
	Then for any $g_1, g_2 \in \SL_2 (\F_p)$ one has 
	\begin{equation}\label{f*:intersection-}
	|g_1 \B g_2 \cap S_{r_1,r_2}| \le 2M \,.
	\end{equation}
	Moreover, 
	for any dihedral subgroup $\G$ one has 
	\begin{equation}\label{f*:intersection+}
	|g_1 \G g_2 \cap S_{r_1,r_2}| \le 12 M \,.
	\end{equation}
	The same holds when $\{ 1,r_1,r_2 \}$ are linearly dependent. 
	\label{l*:intersection}
\end{lemma}
\begin{proof}
	Take $a\in B_1$, $b\in B_2$ and consider the equation 
	\[
	\left( {\begin{array}{cc}
		xr & qx+y/r \\
		zr & qz+w/r \\
		\end{array} } \right)
	=
	\left( {\begin{array}{cc}
		x & y \\
		z & w \\
		\end{array} } \right) 
	\left( {\begin{array}{cc}
		r & q \\
		0 & r^{-1} \\
		\end{array} } \right) 
	=
	\left( {\begin{array}{cc}
		1 & r_1 \\
		r_2 & 1+r_1 r_2 \\
		\end{array} } \right) 
	\left( {\begin{array}{cc}
		X & Y \\
		Z & W \\
		\end{array} } \right) 
	=
	\]
	\[
	=
	\left( {\begin{array}{cc}
		X+r_1 Z & Y+r_1 W \\
		r_2 X + (1+r_1 r_2) Z & r_2 Y+ (1+r_1 r_2) W \\
		\end{array} } \right) \,.
	\]
	From $xr = X+r_1 Z$, $zr = r_2 X + (1+r_1 r_2) Z$, we have 
	\begin{equation}\label{tmp:02.12_1}
	zX + r_1 zZ = r_2 xX + xZ + r_1 r_2 xZ \,.
	\end{equation}
	If $Z=0$, then from $XW-YZ=1$ one derives  $X\neq 0$ and we arrive to $x r_2 = z$. 
	Since  $xw-yz=1$, it follows that $x,z$ cannot be zero simultaneously and hence $r_2$ is a constant.
	Similarly, we see that 
	$x\neq 0$. 
	By assumption $p_1 p_2, p_1 q_2, p_2 q_1, q_1 q_2$ are linearly independent.
	Hence multiplying (\ref{tmp:02.12_1}) by $q_1 q_2$, we obtain a non--zero polynomial (with the non--vanishing term $xZ p_1 p_2$) of degree at most $2M$. 
	Thus equation (\ref{tmp:02.12_1}) has at most $2M$ solutions.

	Now consider any dihedral subgroup which is just a product of a cyclic group of order four (of order two in ${\rm PSL\,}_2 (\F_p)$)
	and a cyclic group of order $(p\pm 1)/2$. 
	It is easy to see that the conjugate class of any element of $\SL_2 (\F_p)$ is the set of elements having the same trace and that an element with  trace 
	$\pm 2$ is conjugated  to 
	$
	\pm 
	\left( {\begin{array}{cc}
		1 & \la \\
		0 & 1 \\
		\end{array} } \right)
	$, see \cite[(6.3)]{Suzuki}.
	We have considered the case of elements with  trace $\pm 2$ already. 
	As for the remained case take any matrix of trace $2\a$ and of the form 
	$
	r_{\eps} (\a,\beta) = 
	\left( {\begin{array}{cc}
		\a & \eps \beta \\
		\beta & \a \\
		\end{array} } \right)
	$,
	where $\a^2 - \eps \beta^2 =1$ and $\a\neq \pm 1$ (hence $\eps \neq 0$). 
	One can check that  for any $n$ the element $r^n_{\eps} (\a,\beta)$ has the same form, i.e. $r^n_{\eps} (\a,\beta) = r_{\eps} (\a_n,\beta_n)$ for some 
	$\a_n, \beta_n \in \F_p$ and $\a^2_n - \eps \beta^2_n =1$. 
	Then as above
	\[
	\left( {\begin{array}{cc}
		\a x+ \beta y & \eps \beta x + \a y \\
		\a z + \beta w & \eps \beta z + \a w \\
		\end{array} } \right)
	=
	\left( {\begin{array}{cc}
		x & y \\
		z & w \\
		\end{array} } \right) 
	\left( {\begin{array}{cc}
		\a & \eps \beta \\
		\beta & \a \\
		\end{array} } \right) 
	=
	\left( {\begin{array}{cc}
		1 & r_1 \\
		r_2 & 1+r_1 r_2 \\
		\end{array} } \right)  
	\left( {\begin{array}{cc}
		X & Y \\
		Z & W \\
		\end{array} } \right) 
	=
	\]
	\[
	=
	\left( {\begin{array}{cc}
		X+r_1 Z & Y+r_1 W \\
		r_2 X + (1+r_1 r_2) Z & r_2 Y+ (1+r_1 r_2) W \\
		\end{array} } \right) \,.
	\]
	From this we have $X+r_1 Z = \a x+ \beta y$, $Y+r_1 W = \eps \beta x + \a y$. 
	Hence we can find $\a,\beta$ via $r_1$, provided $\eps x^2 \neq y^2$.
	After that, we get
	\begin{equation}\label{tmp:02.12_2}
	Z = \a z + \beta w -r_2 (X+r_1 Z) \,.
	\end{equation}
	Since $\a,\beta$ can be linearly expressed 
	via 
	$r_1$, 
	we obtain a contradiction with the linear independence of $p_1 p_2, p_1 q_2, p_2 q_1, q_1 q_2$
	provided $Z\neq 0$  (the term $Z p_1 p_2$ does not vanish). 
	If $Z=0$, then  
	$W\neq 0$ and 
	\begin{equation}
	\eps \beta z + \a w	= r_2 Y  + (1+r_1 r_2) W \,.
	\end{equation}
	We know that $\a,\beta$ can be founded via $r_1$. 
	It gives us a contradiction with linear independence of $p_1 p_2, p_1 q_2, p_2 q_1, q_1 q_2$
	(the term $W p_1 p_2$ does not vanish).

	It remains to consider the case   $\eps x^2 = y^2$. 
	Then we obtain an analogue of (\ref{tmp:02.12_2})
	\begin{equation}\label{tmp:02.12_3}
	W =  \eps \beta z + \a w -r_2 (Y+r_1 Z) \,.
	\end{equation}	
	If $\eps z^2 \neq  w^2$, then from 
	$\a z + \beta w = r_2 X + (1+r_1 r_2) Z$,  $\eps \beta z + \a w = r_2 Y+ (1+r_1 r_2) W$  
	we can find $\a,\beta$ which depends   linearly on $Z+r_2 (X+r_1Z)$, $W+r_2 (Y+r_1 W)$ and substituting them into $X+r_1 Z = \a x+ \beta y$, we obtain a contradiction with linear independence of $p_1 p_2, p_1 q_2, p_2 q_1, q_1 q_2$, provided $Z\neq 0$.
	More precisely, we get a dependence of the form
	$$
	Ar_1 r_2 + B r_2 +Z r_1 + C = 0 \,,
	$$
	where $A,B,C \in \F_p$ are some constants.
	If $A\neq 0$, then in the linear dependence between $p_1 p_2, p_1 q_2, p_2 q_1, q_1 q_2$ we have the non--vanishing term $A p_1 p_2$.
	Now if $A=0$, then $\{1,r_1,r_2\}$ are linearly dependent. 
	Let $\omega = X+r_1 Z$.
	Then from $\a x+ \beta y = \o$, $\a z + \beta w = Z + r_2 \o$, we obtain $\a = \o (w-yr_2) - yZ$, $\beta = \o (xr_2-z) + xZ$.  
	It gives us in view of $\eps x^2 = y^2$ that 
	$$
	1 = \a^2 - \eps \beta^2 = (\o (w-yr_2) - yZ)^2 - \eps (\o (xr_2-z) + xZ)^2 
	=
	$$
	$$
	= \o^2 ( (w-yr_2)^2 - 2\eps (xr_2-z)^2 ) - 2 \o Z (y (w-yr_2) + \eps x (xr_2-z)) 
	=
	$$
	\begin{equation}\label{tmp:23.01_1}
	= \o^2 ( w^2 - \eps z^2 - r_2 (yw - xz \eps) ) - 2 \o Z (yw - xz\eps) \,.
	\end{equation}
	Since $\eps z^2 \neq  w^2$, further $yw-xz \eps \neq 0$, $r_2$ depends linearly on $r_1$ and $\omega = X+r_1 Z$, $Z\neq 0$ it follows that 
	\eqref{tmp:23.01_1} gives a nontrivial equation  on $r_1$ of degree three. 
	Now if $Z=0$, then $\o = X \neq 0$ and from \eqref{tmp:23.01_1} we have a linear equation on $r_2$. 

	Finally, consider the case   $\eps x^2 = y^2$ and $\eps z^2 = w^2$. 
	Put $Q_1= X+r_1 Z$ and $Q_2 = Z+r_2 (X+r_1 Z) = Z+ r_2 Q_1$. 
	Then from $Q_1 = \a x+ \beta y$, $Q_2 = \a z + \beta w$, we obtain
	$
	\a = wQ_1 - yQ_2 \,, \beta = -zQ_1 + x Q_2
	$. 
	Hence using $\eps x^2 = y^2$ and $\eps z^2 = w^2$, we have 
	$$
	1 = \a^2 - \eps \beta^2 = (wQ_1 - yQ_2)^2 - \eps (-zQ_1 + x Q_2)^2 = 2 Q_1 Q_2 (\eps xz - wy) 
	=
	$$
	$$
	=2 (Q_1 Z + r_2 (X+r_1 Z)^2)  (\eps xz - wy)
	\,.
	$$
	If $Z\neq 0$, then the last identity gives us a contradiction with linear independence of  \\
	$\{ p_1 q_1 q_2, p_1 p_2 q_1, p_1^2 p_2, q_1^2 q_2, q^2_1 p_2\}$ because the term $2Z^2  (\eps xz - wy) p_1^2 p_2$ does not vanish.
	If $Z=0$, then from the same equation we obtain $1= 2X^2  (\eps xz - wy) r_2$ and hence  $r_2$ is a constant.

	Lastly, a careful analysis of the proof shows that the same arguments work in the case when $\{1, r_1, r_2 \}$ are linearly dependent.  
	This completes the proof. 
	$\hfill\Box$
\end{proof}

\begin{example}
	Let rational functions $r_1,r_2$ are  just non--constant polynomials. Then $q_1=q_2=1$ and our independency conditions
	are satisfied. 
\end{example}

\begin{remark}
	By a similar argument Lemma \ref{l*:intersection} takes place for rational functions $r_1,r_2$ of several variables. 
	Thus ideologically Lemma \ref{l:intersection} follows from Lemma \ref{l*:intersection} (up to constants). 
	\label{r:one_dim}
\end{remark}

\begin{corollary}
	Let $p_1,p_2 \in \F_p [x]$ be any non--constant polynomials.
	Then for any $A, B \subseteq \F_p$, $|B| \ge p^\eps$, $\eps>0$ one has 
	$$
	\left| \left\{ p_1(b) + \frac{1}{a+p_2 (b)} = p_1(b') + \frac{1}{a'+p_2 (b')} ~:~ a,a'\in A,\, b,b' \in B \right\} \right|
	- \frac{|A|^2 |B|^2}{p}
	\le
	$$
	\begin{equation}\label{f:pol} 
	\le
	2 |A| |B|^2 p^{-1/2^{k+2}} \,,
	\end{equation}
	where $k=k(\eps, \deg p_1, \deg p_2)$. 
	In particular, 
	\begin{equation}\label{f:pol'}
	\left| \left\{ p_1(b) + \frac{1}{a+p_2 (b)} ~:~ a \in A,\, b \in B \right\} \right| \gg \min \{ p, |A| p^{1/2^{k+2}} \} \,.
	\end{equation}
	\label{c:pol}
\end{corollary}
\begin{proof}
	Indeed, in terms of the set $S_{p_1,p_2}$ the number of the solutions to (\ref{f:pol}) is
	$$
	s a = s' a' \,, \quad \quad a,a'\in A\,, s,s' \in S_{p_2,p_1} 
	$$
	or, equivalently, $(s')^{-1} s a = a'$ (we can assume that $a+p_2$ and $ap_1+1+p_1p_2$ are nonzero). 
	Using Lemma \ref{l:counting}, combining with Lemma \ref{l*:intersection}, 
	we obtain the required result. 
	$\hfill\Box$
\end{proof}

\bigskip 

Of course Corollary \ref{c:pol} does not hold if either $p_1$ or $p_2$ has zero degree. 
Using the same arguments as in the proof of Corollary \ref{c:new_exp_sums}, we derive

\begin{corollary}
	Let $B\subseteq \F_p$ and $p_1,p_2,q_1,q_2 \in \F_p [x]$
	such that 
	$$\{p_1,q_1\}\,,  \{p_2,q_2\}\,, \{ p_1 p_2, p_1 q_2, p_2 q_1, q_1 q_2 \}\,, 
	\{ p_1 q_1 q_2, p_1 p_2 q_1, p_1^2 p_2, q_1^2 q_2, q^2_1 p_2\} 
	$$
	are linearly independent over $\F_p$. 
	Put 
	$$
	M:= 
	\max \{ \deg(p_1), \deg(p_2), \deg(q_1), \deg(q_2) \} \,.
	$$
	For any functions $f,g : \F_p \to \C$, $\sum_x f(x) = 0$, and for any set $B$ with $|B| \ge p^\eps$, 
	$\eps > 0$
	one has 
	\begin{equation}\label{f:new_exp_sums_1_a}
	\sum_{x,y} f(x) g(y) \sum_{b \in B}  e\left( y\left( \frac{q_1 (b) q_2 (b) x+ p_1 (b) q_2 (b)}{p_2 (b) q_1 (b) x+ q_1 (b) q_2 (b) + p_1 (b) p_2 (b)} \right) \right) \ll_M \| f \|_2 \| g \|_2 \sqrt{p} |B| p^{-\d} \,.
	\end{equation}
	Further for any nontrivial multiplicative character $\chi$ and $\la \neq 0$, we get 
	\begin{equation}\label{f:new_exp_sums_2_a}
	\sum_{x,y} f(x) g(y)  \sum_{b \in B} \chi \left( y+  \frac{q_1 (b) q_2 (b) x+ p_1 (b) q_2 (b)}{p_2 (b) q_1 (b) x+ q_1 (b) q_2 (b) + p_1 (b) p_2 (b)}  \right) 
	\ll_M \| f \|_2 \| g \|_2 \sqrt{p} |B| p^{-\d} \,.
	\end{equation}
	Here $\d = \d (\eps, M) > 0$. 
	\label{c:R[A,B,B]}
\end{corollary}

Again, using the usual Burgess method
one can obtain a non--trivial bound for  sum (\ref{f:new_exp_sums_2_a}) in the regime when (let $f(x) = X(x)$, $g(y) =Y(y)$ for simplicity)  
$|Y|  \ge p^\eps$, $|B| \ge p^{\eps}$ and $|X| \gg_\eps p^{1/2-\d}$, see Remark \ref{r:chi_Burgess} and the proof of Corollary \ref{c:new_exp_sums}. 

\bigskip


The same method, combining with 
the results 
from \cite{Brendan_rich} concerning {\it rich lines} (not rich hyperbolas)
in $\F_p \times \F_p$ allows to prove

\begin{theorem}
	For any functions $f, g : \F_p \to \C$ and any sets $A$, $B$ with $|B| \ge p^\eps$, 
	$|A| < p^{1-\eps}$, 
	$\eps > 0$
	one has 
	$$
	\sum_{x,y} f(x) g(y) \sum_{b_1,b_2 \in B}  e(y (a+b_1) b_2) \ll \| f \|_2 \| g \|_2 \sqrt{p} |B|^2 p^{-\d} \,,
	$$
	where $\d = \d (\eps) > 0$. 
\end{theorem}

Again the result is nontrivial (suppose for simplicity that $f(x) = X(x)$ for some set $X$) if $|X| \gg p^{1/2-\d}$
and the restriction to the lower bound for size of $B$ can be extended to  $\log |B| \gg \log p / \log \log p$.


\section{On $\GL_2 (\F_p)$--actions}
\label{sec:GL}

In this section we consider the set of all non--degenerate matrices $\GL_2 (\F_p)$ with coefficients from $\F_p$.
Also, let $G$ be its subset.
By  
$\det G$ 
denote the set $\det G := \{ \det g ~:~ g\in G \} \subseteq \F^*_p$ and because $\GL_2 (\F_p)$ is acting on $\F_p$ we can consider 
$G(A) := \{ ga ~:~ g\in G,\, a\in A\}$ for any $A\subseteq \F_p$. 
Of course there is no expanding result similar to Theorem \ref{t:Harald_SL2} in $\GL_2 (\F_p)$ but nevertheless one can easily obtain 

\begin{proposition}
	Let $G\subseteq \GL_2 (\F_p)$ be a set of matrices, $A\subseteq \F_p$ and $\eps >0$ be a real number. 
	Suppose that \\ 
	$\circ~$ $|G| > p^\eps$, \\
	$\circ~$ $|\det G| \le |G| p^{-\eps}$, \\
	$\circ~$ $\sum_{x \in s \G} (G^{-1} * G ) (x) \le p^{-\eps} \sum_{x \in \SL_2 (\F_p)} (G^{-1} * G ) (x)$ for any proper subgroup $\G \subset \SL_2 (\F_p)$, $s\in \SL_2 (\F_p)$.\\
	Then there is $\d = \d (\eps) >0$ such that 
	$$
	|G(A)| \gg \min \{ p, |A| p^\d \} \,. 
	$$
	\label{p:GL2}
\end{proposition}
\begin{proof}
	Put $L = \log p$, $X= G(A)$. 
	For any $\la \in \F^*_p$ consider  $G_\la = \{ g \in G ~:~ \det g = \la \}$.
	Using the Dirichlet principle, we find $\L \subseteq \F^*_p$ and a number $\D$ such that  
	$\D <|G_\la| \le 2\D$ on the set $\L$ and 
	\begin{equation}\label{tmp:20.12_3}
	\rho :=\sum_{x \in \SL_2 (\F_p)} (G^{-1} * G ) (x) = \sum_{\la} |G_\la|^2 \ll L \sum_{\la \in \L} |G_\la|^2 := L \rho_1 \,.
	\end{equation}
	Using 
	the assumption $|\L| \le |\det G| \le |G| p^{-\eps}$ and the Cauchy--Schwarz inequality, we see that
	$$
	\rho \ge  |\det (G)|^{-1} \left( \sum_{\la} |G_\la| \right)^2   = \frac{|G|^2}{|\det G|} \ge |G| p^\eps 
	$$ 
	and hence $\D \gg p^{\eps} / L$.  
	By  the definition of the set $\L$ and the Cauchy--Schwarz inequality, we obtain  
	$$
	|A|^2 \D^2 |\L|^2 
	\le
	\left( \sum_{\la \in \L} |G_\la| |A| \right)^2
	=
	\left( \sum_{\la \in \L} \sum_{x \in X} (G_\la * A) (x) \right)^2 
	\le
	$$
	\begin{equation}\label{tmp:20.12_3-}
	\le
	|\L| |X| \sum_{\la \in \L} \sum_{x} (G_\la * A)^2  (x)
	=
	|\L| |X| \sum_{\la \in \L} \sum_{x} (G^{-1}_\la * G_\la * A)  (x) A(x) = |\L| |X|  \sigma\,.
	\end{equation}
	Further consider $f(x) = \sum_{\la \in \L} (G^{-1}_\la * G_\la ) (x)$ with $\|f\|_1 = \rho_1$ and the measure  $\mu (x) = f(x)/ \rho_1$   and notice that 
	$$
	\sigma = \sum_{x\in A} (f * A) (x) \,.
	$$
	Moreover 
	$$
	\| \mu \|_\infty \ll |\L|\D/(|\L|\D^2) = \D^{-1} \ll L p^{-\eps} \,,
	$$
	and by the assumption we see that 
	for any proper subgroup $\G \subset \SL_2 (\F_p)$, $s\in \SL_2 (\F_p)$ the following holds
	$$
	\mu (s \G)	= \rho_1^{-1} \sum_{x \in s\G} \sum_{\la \in \L} (G^{-1}_\la * G_\la ) (x)
	\le
	\rho_1^{-1}  \sum_{x \in s \G} (G^{-1} * G ) (x)
	\ll
	L \rho^{-1} \sum_{x \in s \G} (G^{-1} * G ) (x) \le L p^{-\eps} \,.
	$$
	Using the arguments as in the proof of  Lemma \ref{l:counting}, we see that 
	for some $\d = \d (\eps) > 0$ one has 
	$$
	\sum_{\la \in \L} \sum_{x} (G^{-1}_\la * G_\la * A)  (x) A(x) -\frac{|A|^2 \rho_1}{p} \ll |A|\rho_1 p^{-\d} \,.
	$$
	Hence, returning to (\ref{tmp:20.12_3-}), we 
	obtain  
	$$
	|A|^2 \D^2 |\L|^2 \ll |\L| |X| \left( \frac{|A|^2}{p} + |A| p^{-\d} \right) \sum_{\la \in \L} |G_\la|^2 
	\ll 
	|\L|^2 \D^2 |X| \left( \frac{|A|^2}{p} + |A| p^{-\d} \right) \,. 
	$$
	It follows that 
	$$
	|X| \gg \min \{ p, |A| p^\d \} \,.
	$$
	This completes the proof. 
	$\hfill\Box$
\end{proof}

\begin{remark}
	\label{r:det_cond}
	One can see from the proof that the condition $|\det G| \le |G| p^{-\eps}$ can be refined to $\sum_{\la } |G_\la|^2 \ge |G| p^\eps$.
\end{remark}


Now we give the proof of a simple consequence of  the theorem above 
(the constants $100$
in (\ref{f:GL2-}), (\ref{f:GL2+}) are not really important
and can be certainly decreased). 

\begin{lemma}
	Let $B_1, B_2, B_3 \subseteq \F_p$ and $G \subseteq \GL_2 (\F_p)$ be a set of the form
	\[
	s_{b_1,b_2,b_3} = 
	\left( {\begin{array}{cc}
		1 & b_1  \\
		b_2 &  b_3 \\
		\end{array} } \right) \in G \,. 
	\]
	Let $M=\max\{ |B_1|, |B_2|, |B_3| \}$.
	Then for any $g_1, g_2 \in \GL_2 (\F_p)$ one has 
	\begin{equation}\label{f:GL2-}
	\sigma_{g_1 \B g_2} (S) \le 100M^4 \,.
	\end{equation}
	Moreover, 
	for any dihedral subgroup $\G$ one has 
	\begin{equation}\label{f:GL2+}
	\sigma_{g_1 \G g_2} (S) \le 100 M^4 \,.
	\end{equation}
	\label{l:GL2_intr}
\end{lemma}
\begin{proof}
	Take any $r_0, r_1, r_2, r_3 \in \F_p$  and consider the equation
	\[
	\left( {\begin{array}{cc}
		xr & qx+y/r \\
		zr & qz+w/r \\
		\end{array} } \right)
	=
	\left( {\begin{array}{cc}
		x & y \\
		z & w \\
		\end{array} } \right) 
	\left( {\begin{array}{cc}
		r & q \\
		0 & r^{-1} \\
		\end{array} } \right) 
	=
	\left( {\begin{array}{cc}
		r_0 & r_1 \\
		r_2 & r_3 \\
		\end{array} } \right) 
	\left( {\begin{array}{cc}
		X & Y \\
		Z & W \\
		\end{array} } \right) 
	=
	\]
	\[
	=
	\left( {\begin{array}{cc}
		r_0 X+r_1 Z & r_0 Y+r_1 W \\
		r_2 X + r_3 Z & r_2 Y+ r_3 W \\
		\end{array} } \right) \,.
	\]
	From $xr = r_0 X+r_1 Z$, $zr = r_2 X +  r_3 Z$, we have 
	\begin{equation}\label{tmp:02.12_1+}
	r_0 zX + r_1 zZ = r_2 xX + r_3 xZ \,.
	\end{equation}
	If $Z=0$, then from $XW-YZ \neq 0$ one obtains $X\neq 0$ and we arrive to $x r_2 = z r_0$. 
	Since  $xw-yz \neq 0$, it follows that either $x$ or $z$ does not vanish and hence either $r_0$ or $r_2$ can be found uniquely.  
	Similarly, one can consider the case $x=0$ ($Z$ is zero or not) and arrive to $r_0 X = - r_1 Z$.
	Hence either  $r_0$ or $r_1$ can be found uniquely. 
	Finally, if $Z\neq 0$ and $x\neq 0$, then we can find $r_3$ from (\ref{tmp:02.12_1+}). 
	Now our matrix from $\SL_2 (\F_p) \cap G^{-1} G$ has the form
	\[
	(b'_3- b'_1 b'_2)^{-2}
	\left( {\begin{array}{cc}
		b'_3 & -b'_1 \\
		-b'_2 & 1 \\
		\end{array} } \right)
	\left( {\begin{array}{cc}
		1 & b_1 \\
		b_2 & b_3 \\
		\end{array} } \right)
	= 
	(b'_3- b'_1 b'_2)^{-2}
	\left( {\begin{array}{cc}
		b'_3-b'_1 b_2 & b_1 b'_3 - b'_1 b_3 \\
		b_2 - b'_2 & b_3 - b_1 b'_2 \\
		\end{array} } \right) 
	\]
	and such that 
	\begin{equation}\label{f:det_b}
	b_3- b_1 b_2 = b'_3- b'_1 b'_2 \neq 0 \,.
	\end{equation} 
	We need to estimate the number of the solutions to \eqref{f:det_b} with some restrictions as $x r_2 = z r_0$, $r_0 X = - r_1 Z$ and so on.
	The appeared  systems of two polynomial equations  are rather concrete, so it is not difficult task. 
	Another way is to use the B\'ezout Theorem.  
	
	First of all notice that fixing a variable of equation (\ref{f:det_b}), we obtain the contribution  at most $M^4$ into the sum  $\sigma_{g_1 \B g_2} (S)$.
	Secondly, if one expression among $r_0,r_2, r_3$ 
	is fixed, 
	then we substitute one appropriate variable into \eqref{f:det_b} and obtain the contribution at most  
	$M^4$ into the sum  $\sigma_{g_1 \B g_2} (S)$.
	If $r_1$ is a constant, then we consider two cases: $b_1=0$ and $b_1 \neq 0$. 
	The last case allows to substitute $b'_3$ into \eqref{f:det_b} and 
	obtain a linear equation relatively to $b_3$ with the main coefficient $(b_1-b'_1)$.
	Totally it gives the contribution at most $4M^4$. 
	Now if $r_0 = C r_2$ for some $C\neq 0$, then we substitute $b'_2$  into \eqref{f:det_b}, obtain a linear equation relatively to $b_2$ 
	and get the contribution at most $2M^4$.
	If $r_0 = C r_1$, then either $b_1 = C^{-1}$ (contribution $M^4$) or $b_1 \neq C^{-1}$ and hence the substitution  $b'_3$ into \eqref{f:det_b} gives  us 
	an equation of degree  at most three and with five variables from $B_j$.
	Similarly, in the case when we find $r_3$ from \eqref{tmp:02.12_1+} one obtains an equation degree at most three and with five variables from $B_j$. 
	It gives at most $3M^4$ solutions and the total contribution into the sum  $\sigma_{g_1 \B g_2} (S)$ can be estimated roughly as   $100M^4$.

	Now let us have deal with the case of a 
	dihedral subgroup which is just a product of a cyclic group of order four (of order two in ${\rm PSL\,}_2 (\F_p)$)
	and a cyclic group of order $(p\pm 1)/2$. 
	Then we use arguments from the proof of Lemma \ref{l*:intersection} and consider  
	$
	r_{\eps} (\a,\beta) = 
	\left( {\begin{array}{cc}
		\a & \eps \beta \\
		\beta & \a \\
		\end{array} } \right)
	$,
	where $\a^2 - \eps \beta^2 =1$ and $\a\neq \pm 1$ (hence $\eps \neq 0$). 
	Again, one can check that  for any $n$ the element $r^n_{\eps} (\a,\beta)$ has the same form, i.e. $r^n_{\eps} (\a,\beta) = r_{\eps} (\a_n,\beta_n)$ for some 
	$\a_n, \beta_n \in \F_p$ and $\a^2_n - \eps \beta^2_n =1$. 
	Then as above
	\[
	\left( {\begin{array}{cc}
		\a x+ \beta y & \eps \beta x + \a y \\
		\a z + \beta w & \eps \beta z + \a w \\
		\end{array} } \right)
	=
	\left( {\begin{array}{cc}
		x & y \\
		z & w \\
		\end{array} } \right) 
	\left( {\begin{array}{cc}
		\a & \eps \beta \\
		\beta & \a \\
		\end{array} } \right) 
	=
	\left( {\begin{array}{cc}
		r_0 & r_1 \\
		r_2 & r_3 \\
		\end{array} } \right)  
	\left( {\begin{array}{cc}
		X & Y \\
		Z & W \\
		\end{array} } \right) 
	=
	\]
	\[
	=
	\left( {\begin{array}{cc}
		r_0 X+r_1 Z & r_0 Y+r_1 W \\
		r_2 X + r_3 Z & r_2 Y+ r_3 W \\
		\end{array} } \right) \,.
	\]
	From this we have $r_0 X+r_1 Z = \a x+ \beta y$, $r_0 Y+r_1 W = \eps \beta x + \a y$. 
	Hence we can find $\a,\beta$ via $r_0$, $r_1$, provided $\eps x^2 \neq y^2$.
	After that, we get
	\begin{equation}\label{tmp:02.12_2+}
	Z = Z (r_0 r_3-r_1 r_2) = r_0 (\a z + \beta w)  
	-r_2 (r_0 X+r_1 Z) \,.
	\end{equation}
	Since $\a,\beta$ can be linearly expressed via  $r_0$, $r_1$, we find $r_2 = (b_2 - b'_2)/ (b_3 - b_1 b_2)^2$ from the last expression, 
	provided 
	$r_0 X+r_1 Z \neq 0$.
	It is easy to check that it  gives us a nontrivial linear equation relatively to $b'_2$ because $r_0,r_1$ do not depend on $b'_2$. 
	The 
	possibility 
	$r_0 X+r_1 Z = 0$ was considered above.
	In any case it gives 
	the contribution at most  $3M^4 + 6 M^4 = 9M^4$ into the sum  $\sigma_{g_1 \G g_2} (S)$.

	It remains to consider the case   $\eps x^2 = y^2$. 
	Then we obtain an analogue of (\ref{tmp:02.12_2+})
	\begin{equation}\label{tmp:02.12_3+}
	W = W (r_0 r_3-r_1 r_2) =  r_0 (\eps \beta z + \a w) -r_2 (r_0 Y+r_1 W) \,.
	\end{equation}	
	If $\eps z^2 \neq  w^2$, then from 
	$\a z + \beta w = r_2 X + r_3 Z$,  $\eps \beta z + \a w = r_2 Y+ r_3  W$  
	we can find $\a,\beta$ as some linear combinations of $r_2,r_3$
	and substituting them into $r_0 X+r_1 Z = \a x+ \beta y$, we find $r_0$ or $r_1$ uniquely, provided $X \neq 0$ or $Z\neq 0$
	(or use the arguments as above applying  (\ref{tmp:02.12_2+}), (\ref{tmp:02.12_3+})).
	In any case it gives 
	the contribution at most  $18M^4$ into the sum  $\sigma_{g_1 \G g_2} (S)$.
	Finally, consider the case  $\eps x^2 = y^2$ and  $\eps z^2 = w^2$. 
	Put $Q_1= r_0 X+r_1 Z$ and $Q_2 = r_2 X + r_3 Z$. 
	Also, let 
	$
	d = \det 
	\left( {\begin{array}{cc}
		x & y \\
		z & w \\
		\end{array} } \right) \neq 0. 
	$
	Then from $Q_1 = \a x+ \beta y$, $Q_2 = \a z + \beta w$, we obtain
	$
	\a = d^{-1} (wQ_1 - yQ_2) \,, \beta = d^{-1} (-zQ_1 + x Q_2) 
	$. 
	Hence
	$$
	d^2 = d^2 (\a^2 - \eps \beta^2) = (wQ_1 - yQ_2)^2 - \eps (-zQ_1 + x Q_2)^2 = 2 Q_1 Q_2 (\eps xz - wy) 
	=
	$$
	$$
	=
	2  (\eps xz - wy)  (r_0 X+r_1 Z) (r_2 X + r_3 Z)
	\,.
	$$
	Using the definitions of $r_j$ and formula (\ref{f:det_b}) to exclude $b'_3$, we have
	$$
	d^2 (b_3 - b_1 b_2 )^2 = 	2  (\eps xz - wy) ( (b_3 - b_1b_2 +b'_1 b'_2 - b'_1 b_2)X+ (b_1b_3-b^2_1 b_2 + b_1 b'_1 b'_2 - b'_1 b_3) Z) ( (b_2-b'_2) X + (b_3 - b_1 b'_2) Z) \,.
	$$
	In gives us 
	a quadratic equation on $b_2$ and the leading  coefficient of this equation is 
	$$ 
	d^2 b_1^2  + 2  (\eps xz - wy) X(Z b^2_1 +  X(b'_1 + b_1)) \,.
	$$
	If this coefficient does not vanished, then it gives us the contribution at most $2M^4$ into the sum  $\sigma_{g_1 \G g_2} (S)$. 
	Suppose that $b_1 \neq 0$. 
	In this case because $d\neq 0$ 
	our 
	coefficient  
	vanishes  iff $X\neq 0$ and  $(2 (\eps xz - wy)  Z+ d^2 X^{-1}) b^2_1 + 2  (\eps xz - wy) X (b'_1 + b_1)  =0$. 
	Since $(\eps xz - wy) \neq 0$, $X\neq 0$,  we find $b'_1$ uniquely. 
	So,  it gives us the contribution at most $2M^4$ into the sum  $\sigma_{g_1 \G g_2} (S)$. 
	This completes the proof. 
	$\hfill\Box$
\end{proof}

\bigskip 

Now we are ready to 
obtain 
a consequence of Proposition \ref{p:GL2}.
Of course one can replace an upper bound for   $|B_3-B_1 B_2|$ to a lower bound for $\sum_x r^2_{B_1-B_2 B_3} (x)$, see Remark \ref{r:det_cond} and formula (\ref{f:r_BB-B}) below. 
Other refinements are possible if one can 
estimate 
the number of the solutions to equation \eqref{f:det_b} with a fixed variable 
and under other restrictions, say, 
$$
C_0 r_0 + C_1 r_1 + C_2 r_2 + C_3 r_3 = 0 \,, \quad \quad  \vec{0} \neq (C_0, C_1, C_2, C_3) \in \F^4_p \,.
$$ 
more 
accurately (we need $p^{-\eps}\sum_x r^2_{B_1-B_2 B_3} (x)$ bound).

\begin{corollary}
	Let $A\subseteq \F_p$, $B_1, B_2, B_3 \subseteq \F_p$, $B:=|B_1|=|B_2|=|B_3| >p^\eps$. 
	Suppose that $|B_3-B_1 B_2| \le B^2 p^{-\eps}$.
	Then there is $\d = \d (\eps) >0$ such that 
	\begin{equation}\label{f:GL2_rational}
	\left| \left\{ \frac{a+b_1}{a b_2  + b_3 } ~:~ a\in A\,, b_j \in B_j \right\} \right| \gg \min \{ p, |A| p^\d \} \,.
	\end{equation}
	\label{c:GL2_rational}
\end{corollary}
\begin{proof}
	Let $G= S_{b_1,b_2,b_3}$, $b_j \in B_j$ and define $G_\la$ as in Proposition \ref{p:GL2}. 
	Since $|B_3-B_1 B_2| \le B^2 p^{-\eps}$, it follows that $|\det G| \le |G| B^{-1}  p^{-\eps}$. 
	Further by Lemma \ref{l:GL2_intr} for any proper subgroup $\G \subset \SL_2 (\F_p)$, we have  
	$$
	\sum_{x \in s \G} (G^{-1} * G ) (x) 
	\ll  B^4 \,,
	$$
	and 
	\begin{equation}\label{f:r_BB-B}
	\sum_{x \in \SL_2 (\F_p)} (G^{-1} * G ) (x) = \sum_z r^2_{B_1-B_2 B_3} (z) \ge \frac{(|B_1||B_2||B_3|)^2}{|B_3-B_1 B_2|}
	\ge
	B^4 p^\eps \,.
	\end{equation}
	Thus all conditions of Proposition \ref{p:GL2} take place for sufficiently large $p$. 
	This completes the proof. 
	$\hfill\Box$
\end{proof}

\section{Appendix}
\label{sec:appendix}

This section contains the proof of Lemma \ref{l:Frobenious} and Theorem \ref{t:flattering}.
Also, we obtain an upper bound for the energy $\E_{k}$, see Theorem \ref{t:E_k_SL2} and Corollary \ref{c:E_k_SL2} below.

\bigskip

{\it The proof of Lemma \ref{l:Frobenious}.}
Having a function $f$ with $\sum_x f(x) =0$, we consider the matrix $M(g,x) := f(g^{-1}x)$ of size $|\SL_2 (\F_p)| \times p$ and its singular value decomposition \cite{Horn-Johnson}
\begin{equation}\label{f:singular}
M(g,x) = \sum_{j=1}^{p} \la_j u_j (g) \ov{v_j (x)} \,.
\end{equation}
One can assume that $\la_1 \ge \la_2 \ge \dots \ge \la_p \ge 0$.
Let $\vec{u} = (1,\dots,1)$ be the vector having $p$ ones. 
Since for any $y\in \F_p$ one has 
$$
(M^* M \vec{u} ) (y) = \sum_x \sum_g M (g,x) \ov{M(g,y)} = \sum_x \sum_g  f(g^{-1}x)  \ov{f(g^{-1}y)} = 
$$
$$
=
\sum_g  \ov{f(g^{-1}y)} \sum_x f(g^{-1}x) 
=
p^{-1} |\SL_2 (\F_p)| \cdot \left|\sum_x f(x) \right|^2 = 0 \,,
$$
it follows that $\la_p=0$. 
Here we have used the fact that for any $x\in \F_p$ 
the following holds 
$|\mathrm{Stab} (x)| = p^{-1} |\SL_2 (\F_p)|$. 
Further 
$$
\sum_{j=1}^{p-1} |\la_j|^2 = \sum_{j=1}^{p} |\la_j|^2 = \sum_{g} \sum_x |M(g,x)|^2 = \sum_{g} \sum_x |f(g^{-1}x)|^2 
=
$$
\begin{equation}\label{tmp:05.11.2017_1}
= |\SL_2 (\F_p)| \cdot \sum_x |f(x)|^2 
= (p^3-p) \| f\|^2_2 \,.
\end{equation}
It is easy to check that if $\vec{v} (x) \in \F^p_p$ is an eigenvector of $T:=M^* M$, then for any $g\in \SL_2 (\F_p)$ the vector   $\vec{v} (g x)$ is another eigenvector of  $T$
and 
moreover 
$T g \vec{v} = g T \vec{v}$.
Thus the following linear operator $Y_g$ defined by the formula $(Y_g h) (x):= h(gx)$, where $h$ belongs to any eigenspace 
defines 
a representation because, obviously,  
$Y_{g_1} Y_{g_2} = Y_{g_1 g_2}$.  
By the famous Frobenius result \cite{Frobenius} the dimension of all nontrivial irreducible representations of $\SL_2 (\F_p)$ is at least $(p-1)/2$. 
It follows that for any eigenfunction $\vec{v}$, $\vec{v} \neq \vec{u}$ the multiplicity of the correspondent eigenvalue is at least $(p-1)/2$ (see details in \cite{SX}, \cite{BG}, \cite{Gill}).  
Hence in view of (\ref{tmp:05.11.2017_1}), we obtain $\la_1 \le 2p \| f\|_2$. 
Finally, by formula (\ref{f:singular}), the orthogonality of the systems of functions $u_j$ and $v_j$, the H\"older inequality,  as well as our upper bound for $\la_1$, 
we have 
$$
\sum_x (F*f) (x) \_phi(x) = \sum_{g,x} M(g,x) F(g) \_phi(x) = \sum_{j=1}^{p-1} \la_j \langle F,\ov{u}_j \rangle \langle \_phi,v_j \rangle
\le	
2p \| F\|_2 \|\_phi\|_2 \| f\|_2 \,.
$$
This completes the proof. 
$\hfill\Box$

\bigskip

{\it The proof of Theorem \ref{t:flattering}.} 
%
Clearly, one can assume that $K$ is sufficiently large because otherwise (\ref{f:flattering}) is trivial. 
Put $s= |\SL_2 (\F_p)|$ and  write $f(x) = \mu(x) - s^{-1}$. 
Then $\sum_x f(x) = 0$, $f(x^{-1}) = f(x)$,  $\| f\|_1 \le 1+1 = 2$ and
by induction one can check that $(f *_{l} f ) (x) = (\mu *_{l} \mu ) (x) - s^{-1}$, $l \in \N$.  
It 
gives 
that 
$$
\T_{2^k} (\mu) = \| \mu *_{2^k} \mu \|^2_2 = s^{-1} +  \| f *_{2^k} f \|^2_2 = s^{-1} + \T_{2^k} (f)\,.
$$
So, our task is to estimate $\| f *_{2^k} f \|^2_2$. 
Clearly, it is enough to prove 
the following bound 
\begin{equation}\label{f:basic_ineq_SL}
\| f *_{2l} f \|^2_2 \ll \| f *_{l} f \|^2_2 \cdot K^{-c_*} := \| f *_{l} f \|^2_2  / M \,,
\end{equation} 
where $c_*>0$ is an absolute constant and $l \in \N$. 
Since 
\begin{equation}\label{tmp:19.12_1}
\| f *_{2l} f \|^2_2 = \sum_x (f *_{2l} f)(xy) (f *_{2l} f)(x) (f *_{2l} f)(y) \,,
\end{equation} 
it follows that by an analogue of Lemma \ref{l:Frobenious} (applied to the natural action of $\SL_2 (\F_p)$ onto $\SL_2 (\F_p)$, see details in \cite{Gow_random}, say) that 
$$
\| f *_{2l} f \|^2_2 \le 2p \| f *_{2l} f \|_2 \| f *_{l} f \|^2_2 
$$
and this implies 
$$
\| f *_{2l} f \|^2_2 \le 4p^2 \| f *_{l} f \|^4_2 \,.
$$
Thus we can assume that
\begin{equation}\label{f:lower_E}
\| f *_{l} f \|^2_2 \ge \frac{1}{4 Mp^2}  
\end{equation}
because otherwise there is nothing to prove.
Now put $r(x) = (f *_{l} f) (x)$. 
Then, clearly, 
$$
\| r\|_\infty \le \| r\|_1 
=
\sum_x \left| (\mu *_{l} \mu) (x) - \frac{1}{s} \right| 
\le
1+ 1  = 2 \,.
$$
From (\ref{tmp:19.12_1}), we have 
$$
\T_{2l} (f) := \| f *_{2l} f \|^2_2 = \sum_{xy = zw} r(x) r(y) r(z) r(w) \,.
$$
Put $\rho = \T_{2l} (f)/(8 \|r\|^3_1)$. 
Since 
$$
\sum_{xy = zw ~:~ |r(x)| < \rho} r(x) r(y) r(z) r(w) 
\le \rho \|r\|^3_1 \le  \T_{2l} (f)/8 
\,,
$$
it follows that
$$
\T_{2l} (f) \le 2 \sum_{xy = zw ~:~ |r(x)|, |r(y)|, |r(z)|, |r(w)| \ge \rho} r(x) r(y) r(z) r(w) \,.
$$
Put $P_j = \{ x\in \SL_2 (\F_p) ~:~ \rho 2^{j-1} < |r (x)| \le \rho 2^{j}\}$ and 
$L = 8 + \log (2l) \cdot \log K$. 
By the assumption $k\le K^{c_*}$ and hence $L\ll k  \log K \le K^{c_*} \log K$. 
Since $\sum_x \mu (x) = 1$, it follows that $\| \mu\|_\infty \ge s^{-1}$ and we obtain a rough upper bound for $K$, namely, $K\le s$.
So, choosing $c_*$ to be sufficiently small, we can suppose that for sufficiently large $p$ one has $L < p^\epsilon$ with a given $\epsilon >0$.   
Clearly, we can assume that $\T_{2^k} (f) \ge K^{-k}$ and hence  for any $x$ such that $|r(x)| \ge \rho$ one has 
$$
2^j \cdot 2^{-4} K^{-\log 2l} \| r\|_1^{-3} \le 2^j \cdot 2^{-4} \T_{2l} (f) \| r\|_1^{-3} = \rho 2^{j-1} < |r(x)| \le \| r\|_\infty \le 2 \,.
$$
It follows that the possible number of the sets $P_j$ does not exceed $L$.
Thus as in the proof of Theorem \ref{t:T_k_G}, we see that there is $P=P_{j_0}$, $\D=2^{j_0-1} \rho$  such that 
$$
\T_{2l} (f) \le 2 L^4 (2\D)^4 \E(P) \,.
$$
Clearly, $\D^2 |P| \le \T_l (f)$ and  
if (\ref{f:basic_ineq_SL})  does not hold, then 
it gives us 
$$
M^{-1} \T_{l} (f) \le \T_{2l} (f) \le 2^5 L^4 \D^4 |P|^3 \le 2^5 L^4 \T^2_{l} (f) |P| \,.
$$
On the other hand, by the second assumption of the theorem and $K\le s$, we get
\begin{equation}\label{f:mu_infty_using}
\T_{l} (f) = \T_{l} (\mu)  + s^{-1} \le \| \mu*_l \mu \|_\infty + s^{-1} \le \| \mu\|_\infty + s^{-1} \le 2K^{-1} \,.
\end{equation}
Hence, combining the last two bounds, we obtain 
\begin{equation}\label{f:P_below}
|P| \ge \frac{K}{2^{6} L^{4} M^{}} \,.
\end{equation}
Similarly, we have 
$$
M^{-1} \T_{l} (f) \le \T_{2l} (f) \le 2^5 L^4 \D^4 |P|^3 \le 2^5 L^4 \T_{l} (f) (|P| \D)^2 \,,
$$
and hence 
\begin{equation}\label{f:PD_below}
|P| \D \ge \frac{1}{2^{3} L^{2} M^{1/2}} \,.
\end{equation}
Also, we have
$$
|P| \D \le \sum_{x\in P} |r (x)| \le \sum_x |r_{} (x)| 
= 
\| r\|_1 \,.
$$
So, 
we see that if (\ref{f:basic_ineq_SL}) has no place, then
$$
\T_{l} (f)/M \le \T_{2l} (f) \le 2 L^4 (2\D)^4 \E(P) = 2^5 L^4 \D^4 |P|^3 (\E(P)/|P|^3) \le  2^{5} L^4 (\D|P|)^2 \T_l (f) (\E(P)/|P|^3) \,.
$$
In other words
$$
\E(P) \ge |P|^3 \cdot  2^{-5} L^{-4} M^{-1} (\D|P|)^{-2} \ge 2^{-7} L^{-4} M^{-1} = \zeta |P|^3 \,.
$$
By the Balog--Szemer\'edi--Gowers Theorem in non--commutative case 
(see, e.g., \cite[Corollary 2.46]{TV}), we see that there is 
$P_* \subseteq P$, $|P_*| \gg \zeta^C |P|$, $|P_* P^{-1}_*| \ll \zeta^{-C} |P_*|$, where $C>0$ is an absolute constant. 
The fact $|P_* P^{-1}_*| \ll \zeta^{-C} |P_*|$ implies that there is a symmetric set $H$, $|H| \ll \zeta^{-C'} |P_*|$,
containing 
the identity, and  a set $|X| \ll \zeta^{-C'}$ 
such that $P_* \subseteq XH$ and
\begin{equation}\label{f:almost_subgr}
HH \subseteq XH\subseteq HXX \,, \quad \quad \mbox{ and } \quad \quad HH \subseteq HX \subseteq XXH 
\end{equation}
see \cite[Proposition 2.43]{TV}.
Here $C'>0$ is another absolute constant. 
Clearly, 
inclusions 
(\ref{f:almost_subgr}) combining with $|X| \ll \zeta^{-C'}$, imply 
\begin{equation}\label{f:H^3}
|H^3| = |H\cdot H^2| \le |H^2 \cdot X| \le |H \cdot X^2| \le |X|^2 |H| \ll \zeta^{-2C'} |H| \,.
\end{equation}
Further since $P_* \subseteq XH$, we see that there is $x\in X$ such that 
\begin{equation}\label{f:H_tmp}
|H| \ge |P_* \cap xH| \ge |P_*|/|X| \gg \zeta^{C'} |P_*| \,.
\end{equation}
By the 
inclusion 
$P_* \subseteq P$ and the definition of the set $P$, we have
\begin{equation}\label{tmp:16.11_3}
\D \zeta^{C'} |P_*| \ll  \D |P_* \cap xH| \le \sum_{z\in P_* \cap xH} |r(x)| \le \mu(xH) + |P_*|/s \le 2 \mu(xH)
\end{equation}
because otherwise in view of our choice of $\D \ge \rho$, we obtain
$$
\T_{2l} (f)/(8 \|r\|^3_1) = \rho \le \D \ll \zeta^{-C'}/s 
$$
and hence by (\ref{f:lower_E}), we derive ($M$ is sufficiently small comparable to $p$ and 
$L < p^\epsilon$ for sufficiently small $\epsilon$, depending on $C'$ only)
$$
\T_{2l} (f) \ll  L^{4C'} M^{C'}  p^{-3} \ll L^{4C'} M^{C'+1} \T_{l} (f) /p \le \T_{l} (f) / M 
$$
as required in (\ref{f:basic_ineq_SL}). 
Thus we see that (\ref{tmp:16.11_3}) takes place and whence by (\ref{f:PD_below})
\begin{equation}\label{f:mu_low}
\mu (xH) \gg \D \zeta^{C'} |P_*| \gg \D |P| \zeta^{C'_1}  \gg \zeta^{C'_1} L^{-2} M^{-1/2} \gg L^{-4 C''} M^{-C''} \,.
\end{equation}
Finally, 
\begin{equation}\label{f:5/2}
|H| \ll \zeta^{-C'} |P_*| \le \zeta^{-C'} \T_l (f) /\D^2 \le \zeta^{-C'} \T_l (f) /\rho^2 \ll  \zeta^{-C'} \T_l (f) / \T^2_{2l} (f) \le p^{5/2} \,,
\end{equation}
say, because otherwise (just the last inequality should be explained) in view of (\ref{f:lower_E}) and sufficiently small $M$, say, $M \le p^{1/12}$,  we get 
$$
\T_{2l} (f) \ll \zeta^{-C'/2} p^{-5/4} \T^{1/2}_l (f) \ll \T^{}_l (f) / M \,.
$$	
Using (\ref{f:5/2}), lower bound  (\ref{f:P_below}), 
estimate (\ref{f:H_tmp}) and recalling (\ref{f:H^3}), we see in view of Theorem \ref{t:Harald_SL2} 
that either our set $H$ belongs to some proper subgroup $\G$ or 
\begin{equation}\label{tmp:16.11_1}
M^{2C'} L^{8C'} \gg \zeta^{-2C'} \gg |H|^c  \gg (\zeta^{C'} |P_*|)^c  \gg  (\zeta^{C'+C} |P|)^c \gg K^c \zeta^{C'''}\,.
\end{equation}
In the last 
inequality 
we have used lower bound (\ref{f:P_below}). 
Now suppose that $H$ belongs to a subgroup.  
By our assumption $\mu (g \G) \le K^{-1}$ for any proper subgroup $\G\subset \SL_2 (\F_p)$, $g\in \SL_2 (\F_p)$.
Hence (\ref{f:mu_low}) gives us 
\begin{equation}\label{tmp:16.11_2}
K^{-1} \ge \mu (x\G) \ge \mu (xH) \gg L^{-4 C''} M^{-C''} \,.
\end{equation}
Finding $M$ satisfying both (\ref{tmp:16.11_1}), (\ref{tmp:16.11_2}) and using the assumption $k\le K^{c_*}$, we obtain the required dependence $M$ on $K$. 
This completes the proof. 
$\hfill\Box$

\bigskip 

%

Using the same method of the proof (one can check that a non--commutative analogue of Lemma \ref{l:change_QG}	takes place and also, see \cite[Theorem 27]{s_Bourgain} and Theorem \ref{t:Q_shift} above), we obtain 

\begin{theorem}
	Let
	$\mu$ be a symmetric probability measure on $\SL_2 (\F_p)$ such that for a parameter 
	$K\ge 1$ one has \\
	$\circ~$ $\mu (g \G) \le K^{-1}$ for any proper subgroup $\G\subset \SL_2 (\F_p)$, $g\in \SL_2 (\F_p)$ and \\
	$\circ~$ $\| \mu \|_\infty \le K^{-1}$.\\
	Put $f(x) = \mu (x) - |\SL_2 (\F_p)|^{-1}$.
	Then there is  $k \ll \log (\| \mu \|^{-1}_2) /\log K$ such that $\E_{2^k} (f) \le 2 \| f \|^{2^{k+1}}_2$. 
	\label{t:E_k_SL2}
\end{theorem}

Theorem 
\ref{t:E_k_SL2}
immediately implies (see the proof of bound (\ref{f:E_k_sigma+}) from Corollary \ref{c:E_k_sigma}).

\begin{corollary}
	Let
	$\mu$ be a symmetric probability measure on $\SL_2 (\F_p)$ such that for a parameter 
	$K\ge 1$ one has \\
	$\circ~$ $\mu (g \G) \le K^{-1}$ for any proper subgroup $\G\subset \SL_2 (\F_p)$, $g\in \SL_2 (\F_p)$ and \\
	$\circ~$ $\| \mu \|_\infty \le K^{-1}$.\\
	Put $f(x) = \mu (x) - |\SL_2 (\F_p)|^{-1}$.
	Then there is  $k \ll \log (\| \mu \|^{-1}_2) /\log K$ such that for an arbitrary function $h : \SL_2 (\F_p) \to \C$ 
	and any set $A \subseteq \SL_2 (\F_p)$ and a positive integer  $s \le 2^k$  one has
	\begin{equation}\label{f:E_k_sigma+_SL}
	\left| \sum_{x\in A} (h  * f)^s (x) \right| \le  |A| \|h\|^s_2 \| f \|^s_2  \left( \frac{2\E (A)}{|A|^4}\right)^{s/2^{k+2}} \,.
	\end{equation}
	\label{c:E_k_SL2}
\end{corollary}

\bigskip

\noindent{I.D.~Shkredov\\
	Steklov Mathematical Institute,\\
	ul. Gubkina, 8, Moscow, Russia, 119991}
\\
and
\\
IITP RAS,  \\
Bolshoy Karetny per. 19, Moscow, Russia, 127994\\
and 
\\
MIPT, \\ 
Institutskii per. 9, Dolgoprudnii, Russia, 141701\\
{\tt ilya.shkredov@gmail.com}


\end{document}